\documentclass[11pt]{amsart}

\usepackage{amsfonts,hyperref}
\usepackage{graphicx}
\usepackage{xypic}
\usepackage{amsbsy}
\usepackage{amsmath}
\usepackage{amssymb}
\usepackage{amsthm, amscd, cite} 
\usepackage{enumitem}
\usepackage[mathscr]{eucal}
\usepackage[dvipsnames]{xcolor}
\usepackage{longtable}
\usepackage{multirow}
\usepackage{array}

\newcolumntype{P}[1]{>{\centering\arraybackslash}m{#1}}

\numberwithin{equation}{section}

\newtheorem{theorem}{\bf Theorem}[section]

\newtheorem{lemma}[theorem]{\bf Lemma}
\newtheorem{proposition}[theorem]{\bf Proposition}

\newtheorem{main}{\bf Theorem}

\newtheorem{mtheorem}[main]{Theorem}
\newtheorem{mcorollary}[main]{\bf Corollary}
\newtheorem{mproposition}[main]{\bf Proposition}

\theoremstyle{definition}

\newtheorem{example}[theorem]{\bf Example}
\newtheorem{definition}[theorem]{\bf Definition}
\newtheorem{hypothesis}[theorem]{\bf Hypothesis}
\newtheorem{notation}[theorem]{\bf Notation}

\theoremstyle{remark}
\newtheorem{remark}[theorem]{\bf Remark}

\numberwithin{equation}{section}
\makeatletter
\let\c@theorem\c@equation
\makeatother

\newtheorem*{namedtheorem}{\theoremname}
\newcommand{\theoremname}{testing}

\renewcommand{\leq}{\leqslant}
\renewcommand{\geq}{\geqslant}
\newcommand{\norm}{\trianglelefteq}

\newcommand{\nin}{\notin}

\newcommand{\gen}[1]{\langle #1 \rangle}
\newcommand{\ol}{\overline}

\newcommand{\incl}{\operatorname{incl}}

\newcommand{\FF}{\mathbb{F}}

\renewcommand{\AA}{\mathfrak{A}}
\newcommand{\CC}{\mathfrak{C}}

\newcommand{\C}{\mathcal{C}}
\newcommand{\D}{\mathcal{D}}
\newcommand{\E}{\mathcal{E}}
\newcommand{\F}{\mathcal{F}}
\renewcommand{\H}{\mathcal{H}}
\newcommand{\I}{\mathcal{I}}

\newcommand{\K}{\mathcal{K}}
\renewcommand{\L}{\mathcal{L}}
\newcommand{\M}{\mathcal{M}}

\renewcommand{\O}{\mathcal{O}}
\newcommand{\Q}{\mathcal{Q}}

\newcommand{\X}{\mathcal{X}}

\newcommand{\Z}{\mathcal{Z}}
\renewcommand{\phi}{\varphi}

\newcommand{\id}{\operatorname{id}}

\newcommand{\Hom}{\operatorname{Hom}}
\newcommand{\Mor}{\operatorname{Mor}}

\newcommand{\Ob}{\operatorname{Ob}}

\newcommand{\Spin}{\operatorname{Spin}}
\newcommand{\Sol}{\operatorname{Sol}}

\newcommand{\Aut}{\operatorname{Aut}}
\newcommand{\Out}{\operatorname{Out}}
\newcommand{\Outdiag}{\operatorname{Outdiag}}
\newcommand{\Inn}{\operatorname{Inn}}
\newcommand{\Inndiag}{\operatorname{Inndiag}}

\newcommand{\Syl}{\operatorname{Syl}}

\newcommand{\Comp}{\operatorname{Comp}}
\newcommand{\sub}{\operatorname{sub}}
\newcommand{\larg}{\operatorname{large}}
\newcommand{\Chev}{\operatorname{Chev}}
\newcommand{\Ac}{\operatorname{A}^\circ}

\newcommand{\hyp}{\mathfrak{hyp}}
\newcommand{\foc}{\mathfrak{foc}}

\newcommand{\<}{\langle}
\renewcommand{\>}{\rangle}
\newcommand{\m}{\mathcal}
\newcommand{\ov}{\overline}

\title[Benson-Solomon components]{Fusion systems with Benson-Solomon components}
\author{Ellen Henke}
\address{Institut f{\"u}r Algebra, Fakult{\"a}t Mathematik, Technische Universit{\"a}t Dresden, 01062 Dresden, Germany}
\email{ellen.henke@tu-dresden.de}
\author{Justin Lynd}
\address{Department of Mathematics \\ University of Louisiana at Lafayette \\
Maxim Doucet Hall \\ Lafayette, LA 70504} 
\email{lynd@louisiana.edu}
\thanks{J. L. was partially supported by NSF Grant DMS-1902152.  This project
has received funding from the European Union's Horizon 2020 research and
innovation programme under the Marie Sk{\l}odowska-Curie grant agreement No.
707758.}
\date{\today}

\begin{document}
\begin{abstract}
The Benson-Solomon systems comprise a one-parameter family of simple exotic
fusion systems at the prime $2$.  The results we prove give significant
additional evidence that these are the only simple exotic $2$-fusion systems,
as conjectured by Solomon.  We consider a saturated fusion system $\F$ having
an involution centralizer with a component $\C$ isomorphic to a Benson-Solomon
fusion system, and we show under rather general hypotheses that $\F$ cannot be
simple. Furthermore, we prove that if $\F$ is almost simple with these
properties, then $\F$ is isomorphic to the next larger Benson-Solomon system
extended by a group of field automorphisms. Our results are situated within
Aschbacher's program to provide a new proof of a major part of the
classification of finite simple groups via fusion systems.  One of the most
important steps in this program is a proof of Walter's Theorem for fusion
systems, and our first result is specifically tailored for use in the proof of
that step. We then apply Walter's Theorem to treat the general Benson-Solomon
component problem under the assumption that each component of an involution
centralizer in $\F$ is on the list of currently known quasisimple $2$-fusion
systems.
\end{abstract}

\maketitle

\section{Introduction}

This paper is situated within Aschbacher's program to classify a large class of
saturated fusion systems at the prime $2$, and then use that result to rework
and simplify the corresponding part of the classification of the finite simple
groups. A saturated fusion system is a category $\F$ whose objects are the
subgroups of a fixed finite $p$-group $S$, and whose morphisms are injective
group homomorphisms between objects such that certain axioms hold. Each finite
group $G$ leads to a saturated fusion system $\F_S(G)$, where $S$ is a Sylow
$p$-subgroup of $G$ and the morphisms are the conjugation maps induced by
elements of $G$. Fusion systems which do not arise in this fashion are called
\emph{exotic}. While exotic fusion systems seem to be relatively plentiful at
odd primes, there is as yet one known family of simple exotic fusion systems at
the prime $2$. These are the \emph{Benson-Solomon fusion systems}
$\F_{\Sol}(q)$ ($q$ an odd prime power) whose existence was foreshadowed in the
work of Solomon \cite{Solomon1974} and Benson \cite{Benson1998c}, and which
were later constructed by Levi--Oliver \cite{LeviOliver2002, LeviOliver2005}
and Aschbacher-Chermak \cite{AschbacherChermak2010}.  Here for any odd prime
power $q$, the underlying $2$-group $S$ of $\F_{\Sol}(q)$ is isomorphic to a
Sylow $2$-subgroup of $\Spin_7(q)$, all involutions in $\F_{\Sol}(q)$ are
conjugate, and the centralizer of an involution is isomorphic to the fusion
system of $\Spin_7(q)$.

It has been conjectured by Solomon that the fusion systems $\F_{\Sol}(q)$ are
indeed the only simple exotic saturated $2$-fusion systems
\cite[Conjecture~57.12]{GuidosBook}.  Some recent evidence for Solomon's
conjecture is provided by a project by Andersen, Oliver, and Ventura, who
carried out a systematic computer search for saturated fusion systems over
small $2$-groups and found that each saturated fusion system over a $2$-group
of order at most $2^9$ is realizable by a finite group. (The smallest
Benson-Solomon system is based on a $2$-group of order $2^{10}$.) Theorems
within Aschbacher's program can be expected to give yet stronger evidence for
Solomon's conjecture, and the results we prove are particularly relevant in
this context.  In order to explain this, we now summarize a bit more of the
background.

The major case distinction in the proof of the classification of finite simple
groups is given by the Dichotomy Theorem of Gorenstein and Walter, which
partitions the finite simple groups of $2$-rank at least $3$ into the groups of
component type and the groups of characteristic $2$-type.  A finite group $G$
is said to be of component type if some involution centralizer modulo core in
$G$ has a component. Here a \emph{component} is a subnormal subgroup which is
quasisimple (i.e. perfect, and simple modulo its center), and the core
$O(C)$ of a finite group $C$ is the largest normal subgroup of $C$ of odd
order. The largest and richest collection of simple groups of component type
are the simple groups of Lie type in odd characteristic.  In the classification
of finite simple groups, one proceeds by induction on the group order. Thus, if
$G$ is a finite group of component type, one assumes that the components of
involution centralizers in $G$ are known, and the objective is then to show
that the simple group itself is known. More precisely, one usually assumes that
a specific quasisimple group $K$ is given as a component in $C_G(t)/O(C_G(t))$
for some involution $t$ of $G$, and then tries to show that $G$ is known. We
refer to such a task as an involution centralizer problem, or a component
problem.  

As several involution centralizer problems in 1960s and 1970s gave rise to
previously unknown sporadic simple groups, this suggests that solving such
problems in fusion systems is a good way to search for new exotic $2$-fusion
systems. Here, we consider an involution centralizer problem in which the
component $\C$ in an involution centralizer of $\F$ is a Benson-Solomon system,
and our main theorems can be viewed as essentially determining the structure of
the ``subnormal closure'' of $\C$ in $\F$. Thus, we provide the treatment of a
problem that has no analogue in the original classification. The results we
prove give additional evidence toward the validity of Solomon's conjecture, or
at least toward the absence of additional exotic systems arising in some direct
fashion from the existence of $\F_{\Sol}(q)$.  

Our work is also an important step in Aschbacher's program.  We refer to the
survey article \cite{AschbacherOliver2016} and the memoir \cite{AschbacherFSCT}
for more details on an outline and first steps of his program. Much of the
background material is also motivated and collected in Section~\ref{S:prelim},
which can serve as a detailed guide to the proof of the main theorems here for
readers not familiar with the classification program. This material has at its
foundation many useful constructions from finite group theory that have been
established in the context of saturated fusion systems. In particular, these
constructions allow one to speak of centralizers of $p$-subgroups, normal
subsystems, simple fusion systems, quasisimple fusion systems, components, and
so on. We refer to the standard reference for those constructions
\cite{AschbacherKessarOliver2011}. 

A saturated $2$-fusion system $\F$ is said to be of \emph{component type} if
some involution centralizer in $\F$ has a component.  Aschbacher defines the
class of $2$-fusion systems of \textit{odd type} as a certain subclass of the
fusion systems of component type.  The fusion systems of odd type are further
partitioned into those of \emph{subintrinsic component type} and those of
\emph{J-component type}. The classification of simple fusion systems of
subintrinsic component type constitutes the first part of the program.  Our
first theorem is tailored for use in the proof of Walter's Theorem
\cite{AschbacherWT}, one of the main steps in the subinstrinsic case.  We then
apply Walter's Theorem to give a treatment of the general Benson-Solomon
component problem in the second main theorem. As a corollary
(Corollary~\ref{C:main}), we show that if $\F$ is almost simple (that is, the
generalized Fitting subsystem $F^*(\F)$ is simple) and $\F$ has an involution
centralizer with a Benson-Solomon component $\C \cong \F_{\Sol}(q)$, then
$F^*(\F) \cong \F_{\Sol}(q^2)$ with the involution inducing an outer
automorphism of $F^*(\F)$ of order $2$. 

To state our main theorems in detail, we introduce now some more notation which
we explain further in Section~\ref{S:prelim}. Fix a saturated fusion system
$\F$ over the $2$-group $S$.  Following Aschbacher, we denote by $\CC(\F)$ the
collection of components of centralizers in $\F$ of involutions in $S$, roughly
speaking. Accordingly, $\F$ is of \emph{component type} if $\CC(\F)$ is
nonempty. The E-balance Theorem in the form of the Pump-Up Lemma
(Section~\ref{SS:pumpup}) allows one to define an ordering on $\CC(\F)$, and
thus obtain the notion of a \emph{maximal} member of $\CC(\F)$.  For $\C \in
\CC(\F)$, we denote by $\I(\C)$ the set of involutions $t$ such that $\C$ is a
component of $C_\F(t)$, roughly speaking, up to replacing $(\C,t)$ by a
suitable conjugate in $\F$.  Finally a member $\C \in \CC(\F)$ is said to be
\emph{subintrinsic} in $\CC(\F)$ if there is $\H \in \CC(\C)$ such that $Z(\H)
\cap \I(\H)$ is not empty. This means in particular that $\H$ itself is in
$\CC(\F)$, as witnessed by some involution in the center of $\H$.

\begin{mtheorem}\label{T:main}
Fix a saturated fusion system $\F$ over a $2$-group $S$ and a quasisimple
subsystem $\C$ of $\F$ over a fully $\F$-normalized subgroup of $S$. Assume
that $\C$ is a subintrinsic, maximal member of $\CC(\F)$ and isomorphic to
a Benson-Solomon system. Then $\C$ is a component of $\F$. 
\end{mtheorem}

As mentioned above, in the logical structure of Aschbacher's classification
program, Theorem~\ref{T:main} is situated within the proof of Walter's Theorem
for fusion systems \cite{AschbacherWT}.  Walter's Theorem in particular implies
that, if a simple saturated $2$-fusion system $\F$ has a member of
$\CC(\F)$ that is the $2$-fusion system of a group of Lie type in odd
characteristic and not too small, then either $\F$ is the fusion system of a
group of Lie type in odd characteristic, or $\F \cong \F_{\Sol}(q)$.  One
assumption of Walter's Theorem is that each member of $\CC(\F)$ is on the list
of currently known quasisimple fusion systems, i.e. either one of the
Benson-Solomon systems or a fusion system of a finite simple group.

A simple saturated $2$-fusion system with an involution centralizer having a
Benson-Solomon component would necessarily be exotic, since involution
centralizers in fusion systems of groups are the fusion systems of involution
centralizers (see also Lemma~\ref{L:ComponentsFSGroups}). Because of the
subintrinsic hypothesis, Theorem~\ref{T:main} does not rule out the possibility
of this happening. However, in Section~\ref{S:general}, we apply Walter's
Theorem for fusion systems to solve the general Benson-Solomon component
problem assuming that all members of $\CC(\F)$ are on the list of known
quasisimple $2$-fusion systems. 

\begin{mtheorem}
\label{T:main2}
Let $\F$ be a saturated fusion system over the $2$-group $S$.  Assume that each
member of $\CC(\F)$ is known and that some fixed member $\C \in \CC(\F)$ is
isomorphic to $\F_{\Sol}(q)$ for some odd prime power $q$. Then for each $t \in
\I(\C)$, there exists a component $\D$ of $\F$ such that one of the following
holds.
\begin{enumerate}
\item $\D = \C$; 
\item $\D \cong \C$, $\D^t \neq \D$, and $\C$ is diagonally embedded in the
direct product $\D\D^t$ with respect to $t$; or
\item $\D \cong \F_{\Sol}(q^2)$, $t \nin \D$, and $\C = C_{\D}(t)$.
\end{enumerate}
\end{mtheorem}

The automorphism groups and almost simple extensions of the Benson-Solomon
systems were determined in \cite{HenkeLynd2018}. The outer automorphism group
of $\F_{\Sol}(q)$ is generated by the class of an automorphism uniquely
determined as the restriction of a standard Frobenius automorphism of
$\Spin_7(q)$ to a Sylow $2$-subgroup, and each extension of $\F_{\Sol}(q)$ is
uniquely determined by the induced outer automorphism group. So in the
situation of Theorem~\ref{T:main2}(3), for example in the case in which $\F$ is
almost simple, the extension $\D\gen{t}$ is known and is the expected one.

\begin{mcorollary}\label{C:main}
Let $\F$ be a saturated fusion system over the $2$-group $S$ such that $\D =
F^*(\F)$ is simple. Assume that each member of $\CC(\F)$ is known, and that
some member $\C \in \CC(\F)$ is isomorphic to $\F_{\Sol}(q)$ for some odd prime
power $q$. Then $\D$ is isomorphic to $\F_{\Sol}(q^2)$. Moreover, for each $t
\in \I(\C)$, we have $t \notin \D$, some conjugate of $t$ induces a standard
field automorphism on $\D$, and $C_\D(t) = \C$. 
\end{mcorollary}

Corollary~\ref{C:main} follows immediately from Theorem~\ref{T:main2} and the
following proposition, which extends the results of \cite{HenkeLynd2018}.  A
more precise statement is found in Proposition~\ref{P:fieldconjugate} as one of
our preliminary results. 

\begin{mproposition}
Let $\F$ be a saturated fusion system over the $2$-group $S$ such that $F^*(\F)
= \F_{\Sol}(q^2)$. Then, writing $S_0$ for Sylow subgroup of $F^*(\F)$, all
involutions in $S-S_0$ are $\F$-conjugate, and there exists an involution $t\in
S-S_0$ such that $C_{F^*(\F)}(t) \cong \F_{\Sol}(q)$. 
\end{mproposition}

We now give an outline of the paper. Section~\ref{S:prelim} provides the
requisite background material, much of it due to Aschbacher, together with
motivation coming from the group case and some new lemmas needed later on.
The various definitions used in later sections are summarized in a
large table in Section~\ref{SS:prelim-summary}. The proof of
Theorem~\ref{T:main} begins in Section~\ref{S:showStandard}, where we show that
a subintrinsic maximal Benson-Solomon component is necessarily a standard
subsystem in the sense of Section~\ref{SS:standard}. When combined with results
of Aschbacher in \cite{AschbacherFSCT}, this allows the consideration of a
subsystem $\Q$ which plays the role of the centralizer of $\C$, and with a
little more work shows that the Sylow subgroup $Q$ of $\Q$ is either of
$2$-rank $1$ or elementary abelian. Next, in Section~\ref{S:elemab}, we handle
the case in which $Q$ is elementary abelian and prove a lemma regarding the
$2$-rank $1$ case. In Section~\ref{S:quaternion}, we handle the case in which
$Q$ is quaternion using Aschbacher's classification of quaternion fusion
packets \cite{AschbacherQFP}.  Finally, in Section~\ref{S:cyclic} we handle the
cyclic case and complete the proof of Theorem~\ref{T:main}. We then prove
Theorem~\ref{T:main2} in Section~\ref{S:general}. 

\subsection*{Acknowledgements}
We would like to thank Michael Aschbacher for providing us with early copies of
his preprints on the various steps of the program, and for suggesting
Lemma~\ref{L:ComponentsFSGroups} to us. We also express our gratitude to the
referees for comments and suggestions which led to numerous improvements to the
paper.

\section{Preliminaries}\label{S:prelim}

\subsection{Local theory of fusion systems}\label{SS:LocalTheoryFS}
Throughout let $\F$ be a saturated fusion system over a finite $p$-group $S$.
For general background on fusion systems, in particular for the definition of a
saturated fusion system, we refer the reader to
\cite[Chapter~I]{AschbacherKessarOliver2011}.  In addition to the notations
introduced there, we will write $\F^f$ for the set of fully $\F$-normalized
subgroups of $S$. Moreover, we write $\E\leq \F$ to indicate that $\E$ is
a (not necessarily saturated) subsystem of $\F$. Conjugation-like maps will be
written on the right and in the exponent.  In particular, if $\m{E}$ is a
subsystem of $\F$ over $T$ and $\alpha\in\Hom_\F(T,S)$, then $\E^\alpha$
denotes the subsystem of $\F$ over $T^\alpha$ with
$\Hom_{\m{E}^\alpha}(P^\alpha,Q^\alpha)=\{\alpha^{-1}\circ
\phi\circ\alpha\colon \phi\in\Hom_{\m{E}}(P,Q)\}$ for all $P,Q\leq T$. 

\subsubsection{Local subsystems}
We recall that, for any subgroup $X$ of $S$, we have the normalizer and the
centralizer of $X$ defined. The normalizer $N_\F(X)$ is a fusion subsystem of
$\F$ over $N_S(X)$, and the centralizer $C_\F(X)$ is a fusion subsystem of $\F$
over $C_S(X)$. These subsystems are not necessarily saturated, but if $X$ is
fully $\F$-normalized, then $N_\F(X)$ is saturated, and if $X$ is fully
centralized, then $C_\F(X)$ is saturated. Thus, we will often move from a
subgroup of $S$ to a fully $\F$-normalized (and thus fully $\F$-centralized)
conjugate of this subgroup. In this context it will be convenient to use the
following notation, which was introduced by Aschbacher.

\begin{notation}\label{N:AA}
For a subgroup $X \leq S$, denote by $\AA(X)$ or $\AA_\F(X)$ the set of morphisms $\alpha \in
\Hom_{\F}(N_S(X),S)$ such that $X^{\alpha} \in \F^f$.  
\end{notation}

Throughout, we will use often without reference that $\AA(X)$ is non-empty for
every subgroup $X$ of $S$. In fact, the following lemma holds.

\begin{lemma}\label{AAnonempty}
If $X\leq S$ and $Y\in X^\F\cap \F^f$, then there exists $\alpha\in \AA(X)$ with $X^{\alpha}=Y$.
\end{lemma}
\begin{proof}
See e.g. \cite[Lemma~I.2.6(c)]{AschbacherKessarOliver2011}. 
\end{proof}

If $x\in S$, then we often write $C_\F(x)$, $N_\F(x)$ and $\AA(x)$ instead of
$C_\F(\<x\>)$, $N_\F(\<x\>)$ and $\AA(\<x\>)$ respectively. Similarly, we call
$x$ fully centralized (fully normalized), if $\<x\>$ is fully centralized
(fully normalized respectively). If $x$ is an involution, then the reader
should note that $C_\F(x)=N_\F(\<x\>)$, and $x$ is fully centralized if and only
if $\<x\>$ is fully normalized. 

\subsubsection{Normal and subnormal subsystems}
Recall that a subgroup $T$ of $S$ is called \emph{strongly closed} in $\F$ if
$P^\phi\leq T$ for every subgroup $P\leq T$ and every $\phi\in\Hom_\F(P,S)$.
The following elementary lemma will be useful later on.

\begin{lemma}\label{GetStronglyClosedinNormalizer}
Let $T$ be strongly closed in $\F$ and suppose we are given two $\F$-conjugate
subgroups $U$ and $U'$ of $S$. If $T\leq N_S(U)$ and $U'$ is fully normalized,
then $T\leq N_S(U')$. 
\end{lemma}
\begin{proof}
By Lemma~\ref{AAnonempty}, there
exists $\alpha\in\AA(U)$ such that $U^\alpha=U'$. Then, as $T$ is strongly
closed, $T=T^\alpha\leq N_S(U)^\alpha\leq N_S(U')$ and this proves the assertion.
\end{proof}

A subsystem $\m{E}$ of $\F$ over $T\leq S$ is called \emph{normal} in $\F$ if
$\m{E}$ is saturated, $T$ is strongly closed, $\m{E}^\alpha=\m{E}$ for every
$\alpha\in\Aut_\F(T)$, the Frattini condition holds, and a certain technical
extra property is fulfilled (see
\cite[Definition~I.6.1]{AschbacherKessarOliver2011}). Here the Frattini
condition says that, for every $P\leq T$ and every $\phi\in\Hom_\F(P,T)$, there
are $\phi_0\in\Hom_{\m{E}}(P,T)$ and $\alpha\in\Aut_\F(T)$ such that
$\phi=\phi_0\circ\alpha$. 

Particularly important cases of normal subsystems include the (unique) smallest
normal subsystem of $\F$ over $S$, which is denoted by $O^{p^\prime}(\F)$ (cf.
\cite[Theorem~I.7.7]{AschbacherKessarOliver2011}), and the normal subsystem
$O^p(\F)$ of $\F$ (cf.  \cite[Theorem~I.7.4]{AschbacherKessarOliver2011}) over
the hyperfocal subgroup $\hyp(\F)$ of $S$. This last theorem shows additionally
that there is a one-to-one correspondence between the subgroups $T$ of $S$
containing $\hyp(\F)$ and the saturated fusion systems $\E$ of $\F$ of
$p$-\emph{power index} in $\F$. Here, a subsystem $\E$ over $T$ is said to
be of $p$-power index if $T \geq \hyp(\F)$ and $\Aut_{\E}(P) \geq
O^p(\Aut_{\F}(P))$ for each $P \leq T$. 

Once normal subsystems are defined, there is then a natural definition of a
subnormal subsystem by transitive extension. We will need the following
lemma.

\begin{lemma}\label{L:Subnormalfn}
If $\E$ is a subnormal subsystem of $\F$ over $T$, then every fully
$\F$-normalized subgroup of $T$ is also fully $\E$-normalized.
\end{lemma}
\begin{proof}
In the case that $\E$ is normal in $\F$, this is
\cite[Lemma~3.4.5]{AschbacherNormal}. The general case follows by induction on
the length of a subnormal series for $\E$ in $\F$.
\end{proof}

\subsubsection{Normal and subnormal closures}\label{SSS:subnormalclosure}

In Section~\ref{S:quaternion}, we will need to work with fusion system
analogues of normal and subnormal closures in groups. By \cite[Theorem~1]{AschbacherGeneralized}, given normal subsystems $\E_i$ over
$T_i$ for $i=1,2$, one can define a unique normal subsystem
$\E_1\wedge\E_2$ of $\F$ over $T_1\cap T_2$ which is normal in $\E_1$ and
$\E_2$. By iteration of this process, given normal subsystems
$\E_1,\dots,\E_r$, there is a subsystem $\bigwedge_{i=1}^r\E_i$ of $\F$
which plays the same role as the intersection of normal subgroups plays in the
theory of groups. For any subgroup $Q\leq S$, the \emph{normal closure}
of $Q$ in $\F$ is defined as \[\bigwedge_{\E\unlhd\F,Q\subseteq\E}\E.\] 
Set further $\sub_0(\F,Q) = \F$, and for each $i \geq 0$, define
$\sub_{i+1}(\F,Q)$ to be the normal closure of $Q$ in $\sub_i(\F,Q)$.  Then
$\sub_{i+1}(\F,Q) \norm \sub_i(\F,Q)$ for each $i \geq 0$. Since $\F$ is
finite, the series is eventually stationary. The \emph{subnormal closure}
$\F^{\circ}$ of $Q$ is defined to be the terminal member of this series.

\subsubsection{Product systems}\label{SSS:Products}
We have seen $O^p(\F)$ as a particularly important example of a normal
subsystem. Conversely, given a normal subsystem $\E$ of $\F$ over $T$ and a
subgroup $P$ of $S$, one may construct the product system $\E P$ of $\F$, which
contains $\E$ as a normal subsystem of $p$-power index. This is a saturated
subsystem of $\F$, and furthermore it is the unique saturated subsystem $\D$ of
$\F$ over $TP$ such that $O^p(\D) = O^p(\E)$.  See
\cite[Section~8]{AschbacherGeneralized}, and also \cite{Henke2013} for a
simplified construction of $\E P$.  

Note however that the uniqueness of the construction of the product depends on
the ambient system $\F$ in which it is defined, as is seen in the following
example.
\begin{example}[{\!\!\cite[Example~7.4]{Henke2013}}] 
\label{E:EP}
Let $S$ be a $2$-dimensional vector space over $\FF_q$ with $q \geq 3$ a prime
power, let $U$ be a one-dimensional subspace of $S$, and let $W_1 \neq W_2$ be
two complements to $U$ in $S$. Let $1 \neq \lambda \in \FF_q^\times$, and
define $\alpha_i \in GL(S)$ to be the transformation which acts as
multiplication by $\lambda$ on $U$ and which is the identity on $W_i$ ($i =
1,2$). Then $\alpha_1|_U = \alpha_2|_U$. So if we set $\F_i = \F_S(S \rtimes
\gen{\alpha_i})$, then $O^p(\F_1) = \F_U(U\rtimes \gen{\alpha_1|_U}) =
\F_U(U\rtimes \gen{\alpha_2|_U}) = O^p(\F_2)$. Let $\E$ be this subsystem and
set $P = W_1$.  Then $P$ is central in $S$, so fully $\F_i$-normalized, and
$(\E P)_{\F_1} = \F_1 \neq \F_2 = (\E P)_{\F_2}$. 
\end{example}

The following lemma about factorization of morphisms in product systems will be
needed later in Lemma~\ref{L:ConjugateDDt}.
\begin{lemma}\label{L:ProductFactorize}
Let $\E$ be a normal subsystem of $\F$ over $T$. Let $P\leq S$, $X\leq
TP$, and $\phi\in\Hom_{\E S}(X,S)$. Then $\phi=\psi c_s$ for some
$\psi\in\Hom_{\E P}(X,TP)$ and some $s\in S$.  
\end{lemma}
\begin{proof}
We will use the definition of the product given in \cite{Henke2013}. By this
definition, $\E P=\E (TP)$. Hence, to ease notation we may assume $T\leq P$.
Setting 
\[
\Ac(Q):=\<\alpha \in \Aut_\F(Q) \colon \alpha\mbox{ of $p'$
order}, [Q,\alpha]\leq Q\cap T,\mbox{ and }\alpha|_{Q\cap T}\in
\Aut_{\E}(Q\cap T)\>
\] 
for every $Q\leq S$, the definition of the product says that 
\[
\E S=\<\Ac(Q)\colon Q\leq S,\;Q\cap T\in\E^c\>_S \quad \mbox{and} \quad \E
P=\<\Ac(Q)\colon Q\leq P,\;Q\cap T\in \E^c\>_P.\]
Note that $\Ac(Q)^{c_s}=\Ac(Q^s)$ and $Q^s\cap T\in\E^c$ for
every $Q\leq S$ with $Q\cap T\in\E^c$ and every $s \in S$.  Hence,
$\phi$ is of the form $\phi=\psi c_s$ for some $s\in S$ and some morphism
$\psi\in\Hom_{\E S}(X,S)$ such that $\psi$ decomposes as a composition of
restrictions of automorphisms in $\Ac(Q)$ for various $Q \leq S$ with
$Q\cap T\in\E^c$. (That is, all conjugation homomorphisms induced by elements
of $S$ appearing in a decomposition of $\phi$ as a morphism in $\E S$ can be
moved to the end.) Write $\psi=(\alpha_1|_{X_0})(\alpha_2|_{X_1})\cdots
(\alpha_k|_{X_{k-1}})$ where $X=X_0,\dots,X_k$ and $Q_1,\dots,Q_k$ are
subgroups of $S$, and $\alpha_i\in \Ac(Q_i)$ is such that
$\<X_{i-1},X_i\>\leq Q_i$ and $X_{i-1}^{\alpha_i}=X_i$. By definition of
$\Ac(Q_i)$, we have $[Q_i,\beta]\leq Q_i\cap T\leq Q_i\cap P$ for each
$\beta \in \Ac(Q_i)$, and hence $\alpha_i|_{Q_i\cap P}\in \Ac(Q_i\cap
P)$. In particular, as $X_0=X\leq P$ by assumption, we have $X_i\leq P$ for
$i=0,1,\dots,k$. Hence, $\psi\colon X\rightarrow P$ is a morphism in $\E P$. 
\end{proof}

\subsubsection{Normalizers of $p$-subgroups in normal subsystems}
We will make use of the following definition and two lemmas concerning
local subsystems in product systems later in Lemma~\ref{L:CCM}. 

\begin{definition}\label{D:NEP}
Let $\E$ be a
normal subsystem of $\F$ over $T\leq S$. Then for any subgroup $P\leq S$ such
that $P\in (\E P)^f$, we define $N_\E(P)$ to be the unique normal subsystem of
$N_{\E P}(P)$ over $N_T(P)$ of $p$-power index. 
\end{definition}

Notice that the above definition makes sense. For if $P\in (\E P)^f$, $N_{\E
P}(P)$ is saturated. Moreover, $\hyp(N_{\E P}(P))\leq \hyp(\E P)\leq T$ and
thus $\hyp(N_{\E P}(P))\leq N_T(P)$ with $N_T(P)$ is strongly closed in $N_{\E
P}(P)$. So by \cite[Theorem~I.7.4]{AschbacherKessarOliver2011}, there exists a
unique normal subsystem of $N_{\E P}(P)$ over $N_T(P)$ of $p$-power index.

As before, the definition of $N_\E(P)$ depends on the fusion system $\F$,
since $\E P=(\E P)_\F$ depends on $\F$.  For instance, in
Example~\ref{E:EP}, $P = W_1$ is normal in $\F_1$ but not in $\F_2$.  So
$N_{\E}(P)_{\F_1}$ is the unique normal subsystem of $\F_1$ over $N_{U}(P) = U$
of index a power of $p$, namely $\E$. But $N_{\F_2}(P) = \F_S(S)$, so
$N_{\E}(P)_{\F_2}$ is the unique normal subsystem of $\F_S(S)$ over $U$ of
$p$-power index, namely $\F_U(U)$. Hence, $N_{\E}(P)_{\F_1}$ and
$N_{\E}(P)_{\F_2}$ are not equal, and indeed are not even isomorphic to each
other.

We write $N_\E(P)_\F$ for $N_\E(P)$ if we want to make clear that we formed
$N_\E(P)$ inside of $\F$. If $p=2$ and $t$ is an involution, then we write
$C_\E(t)=C_\E(t)_\F$ for $N_\E(\<t\>)$. 

\begin{lemma}\label{L:NEP1}
Let $\E$ be a normal subsystem of $\F$ over $T\leq S$. If $P\in\F^f$, then
$P\in (\E P)^f$ and $N_\E(P)$ is normal in $N_\F(P)$.  
\end{lemma}
\begin{proof}
Let $P\in\F^f$ and fix and $\E P$-conjugate $Q$ of $P$ with $Q\in(\E P)^f$. By
construction of $\E P$, we have $TP=TQ$. Let $\alpha\in\AA(Q)$ with
$Q^\alpha=P$. As $T$ is strongly closed, we have $N_T(Q)^\alpha\leq N_T(P)$. So
$N_{TP}(Q)^\alpha=N_{TQ}(Q)^\alpha=(N_T(Q)Q)^\alpha\leq N_T(P)P=N_{TP}(P)$.
Hence, $|N_{TP}(Q)|\leq |N_{TP}(P)|$. As $Q$ is fully $\E P$-normalized, it
follows that $P$ is fully $\E P$-normalized.

By \cite[Theorem~1]{Henke2013}, $O^p(N_{\E P}(P))N_T(P)$ is the unique
saturated subsystem $\D$ of $N_{\E P}(P)$ over $N_T(P)$ with $O^p(\D)=O^p(N_{\E P}(P))$.
Looking at the construction of normal subsystems of $p$-power index given in
\cite[Theorem~I.7.4]{AschbacherKessarOliver2011}, one observes that
$O^p(N_\E(P))=O^p(N_{\E P}(P))$. Thus, $O^p(N_{\E P}(P))N_T(P)$ equals
$N_\E(P)$. Hence, if $P\in\F^f$, our notation is consistent with the one
introduced by Aschbacher in \cite[8.24]{AschbacherGeneralized}, where it is
proved that $N_\E(P)$ is normal in $N_\F(P)$. 
\end{proof}

\begin{lemma}\label{L:NEP2}
Let $\E$ be a normal subsystem of $\F$ over $T\leq S$. Fix $P\leq S$ such that $P\in (\E
P)^f$, and let $\phi\in\Hom_\F(N_T(P)P,S)$. Then $P^\phi\in (\E P^\phi)^f$,
$N_T(P)^\phi=N_T(P^\phi)$ and $\phi|_{N_T(P)}$ induces an isomorphism from
$N_\E(P)$ to $N_\E(P^\phi)$.
\end{lemma}

\begin{proof}
It is sufficient to show that $(N_T(P)P)^\phi=N_T(P^\phi)P^\phi$ and $N_{\E
P}(P)^\phi=N_{\E P^\phi}(P^\phi)$. For if this is true then, as $T$ is strongly
closed, $N_T(P)^\phi=N_T(P^\phi)$. So, since $\phi$ induces an isomorphism from
$N_{\E P}(P)$ to $N_{\E P^\phi}(P^\phi)$, it will take the unique normal
subsystem of $N_{\E P}(P)$ over $N_T(P)$ of $p$-power index to the unique
normal subsystem of $N_{\E P^\phi}(P^\phi)$ over $N_T(P^\phi)$ of $p$-power
index.

By \cite[1.3.2]{AschbacherFSCT}, we have $\F=\<\E S,N_\F(T)\>$. Hence, it is
sufficient to prove the claim in the case that $\phi$ is a morphism in
$N_\F(T)$ or a morphism in $\E S$. 

If $\phi$ is a morphism in $N_\F(T)$, then $\phi$ extends to
$\alpha\in\Hom_\F(TP,TP^\phi)$ with $T^\alpha=T$. Notice that
$P^\alpha=P^\phi$, and $\alpha\colon TP\rightarrow TP^\phi$ is an isomorphism
of groups, which induces by the construction of $\E P$ and $\E P^\phi$ in
\cite{Henke2013} an isomorphism from $\E P$ to $\E P^\phi$. So
$P^\phi=P^\alpha\in (\E P^\phi)^f$ and $\phi=\alpha|_{N_T(P)P}$ induces an
isomorphism from $N_{\E P}(P)$ to $N_{\E P^\phi}(P^\phi)$. Hence, the assertion
holds if $\phi$ is a morphism in $N_\F(T)$.

Assume now that $\phi$ is a morphism in $\E S$. By the construction of $\E S$
and $\E P$ in \cite{Henke2013}, $\phi$ is the composition of a morphism in $\E
P$ and a morphism in $\F_S(S)$. As $\F_S(S)\leq N_\F(T)$ and the assertion
holds if $\phi$ is a morphism in $N_\F(T)$, we may thus assume without loss of
generality that $\phi$ is a morphism in $\E P$. However, then $TP=TP^\phi$, $\E
P=\E P^\phi$ and it follows from $P\in (\E P)^f$ that
$(N_T(P)P)^\phi=N_{TP}(P)^\phi=N_{T P^\phi}(P^\phi)=N_T(P^\phi)P^\phi$. So
$\phi$ induces an isomorphism from $N_{\E P}(P)$ to $N_{\E P}(P^\phi)=N_{\E
P^\phi}(P^\phi)$. So the assertion holds if $\phi$ is a morphism in $\E P$ and
thus also if $\phi$ is a morphism in $\E S$. As argued above, this shows that
the lemma holds.
\end{proof}

\subsection{Automorphisms and extensions of fusion and linking systems}\label{SS:aut}
At several later points, we will need to construct various extensions
of fusion systems and to determine the structure of extensions where they
arise. For example, if $\F$ is a saturated fusion system over $S$ and $\E$ is a
normal subsystem of $\F$, then we want to be able to construct certain
subsystems of $\F$ containing $\E$ and determine their structure from the
structure of $\E$.  In the category of groups, this is a relatively painless
process when the normal subgroup is quasisimple. However, in fusion systems
there are technical difficulties that necessitate in many cases the
consideration of linking systems associated to $\F$ and $\E$. 

We refer to \cite[Section III.4]{AschbacherKessarOliver2011} or \cite{AOV2012}
for the definition of an abstract linking system as used here, and for more
details on automorphisms of fusion and linking systems.  Fix a linking system
$\L$ for $\F$ with object set $\Delta$ and structural functors $\delta$ and
$\pi$, which we write on the left of their arguments. The group of
automorphisms of $\F$ is defined by
\[
\Aut(\F) = \{\alpha \in \Aut(S) \mid \F^\alpha = \F\}. 
\]
Then $\Aut_\F(S)$ is normal in $\Aut(\F)$, and the quotient
$\Aut(\F)/\Aut_{\F}(S)$ is denoted $\Out(\F)$.  

An automorphism of $\L$ is an equivalence $\alpha \colon \L \to \L$ that is
both \emph{isotypical} and \emph{sends inclusions to inclusions}. Since we do
not use those conditions explicitly, we refer to
\cite[Section~III.4]{AschbacherKessarOliver2011} for their precise meanings.
Each automorphism of $\L$ is indeed an automorphism of the category $\L$, not
merely a self-equivalence. We shall write $\Aut(\L)$ for the group of
automorphisms of $\L$. There is always a \emph{conjugation map}
\[
c\colon \Aut_\L(S) \longrightarrow \Aut(\L)
\]
which sends an element $\gamma \in \Aut_\L(S)$ to the functor $c_\gamma \in
\Aut(\L)$ defined on objects by $P \mapsto P^\gamma := P^{\pi(\gamma)}$. For a
morphism $\phi \in \Mor_{\L}(P,R)$, the map $c$ sends $\phi$ to $\phi^\gamma$,
namely the morphism
\[
\phi^\gamma := \gamma^{-1}|_{P^\gamma, P} \circ \phi \circ \gamma|_{R,R^\gamma} \in \Mor_\L(P^\gamma, R^\gamma),
\]
where, for example, $\gamma|_{R, R^\gamma}$ is the \emph{restriction} of
$\gamma$, uniquely determined by the condition that $\delta_{R,S}(1) \circ
\gamma = \gamma|_{R,R^\gamma} \circ \delta_{R^\gamma,S}(1)$ in $\L$. The image
of $c$ in $\Aut(\L)$ is a normal subgroup of $\Aut(\L)$, and 
\[
\Out(\L) := \Aut(\L)/\{c_\gamma \mid \gamma \in \Aut_\L(S)\}
\]
is the group of outer automorphisms of $\L$.

There are natural maps $\tilde{\mu}\colon \colon \Aut(\L) \to \Aut(\F)$ and
$\mu\colon \Out(\L) \to \Out(\F)$ which, at least when $\Delta = \F^c$, fit
into a commutative diagram
\begin{equation}\label{E:diagram}
\begin{gathered}
\xymatrix{
    &     1 \ar[d]      &      1  \ar[d]    &      1 \ar[d] &   \\
  Z(\F) \ar@{=}[d] \ar[r]^{\incl} &  Z(S) \ar[r] \ar[d]^{\delta_S} & \widehat{Z}^1(\O(\F^c), \Z_\F) \ar[d]^{\widetilde{\lambda}} \ar[r] & \varprojlim{\!}^1(\Z_\F) \ar[r] \ar[d]^{\lambda} & 1 \\
  Z(\F) \ar[r] & \Aut_\L(S)  \ar[r]^c \ar[d]^{\pi_S}  &  \Aut(\L) \ar[r] \ar[d]^{\widetilde{\mu}} &  \Out(\L) \ar[r] \ar[d]^{\mu} & 1 \\
1 \ar[r] & \Aut_\F(S)  \ar[r] \ar[d]  &  \Aut(\F) \ar[r] \ar[d]  & \Out(\F)  \ar[r]\ar[d] & 1\\
    &   1  & 1 & 1 &  \\ 
}
\end{gathered}
\end{equation}
with all rows and columns exact. Here, $\widehat{Z}^i(\O(\F^c),\Z_\F)$ denotes
a certain subgroup of the group of normalized cocycles of the center functor
defined on the orbit category of $\F$-centric subgroups
\cite[pp.186,174]{AschbacherKessarOliver2011}, and $\varprojlim^{i}\Z_\F$
denotes the corresponding cohomology group. The diagram is an updated version
of the one appearing in \cite[p.186]{AschbacherKessarOliver2011}. The last two
columns were not known to be exact until Chermak's proof of the uniqueness of
centric linking systems \cite{Chermak2013}. For example, by
\cite[Proposition~III.5.12]{AschbacherKessarOliver2011} the cokernel of the map
$\mu$ injects into $\lim^2 \Z_\F$, which is zero by
\cite[Theorem~3.4]{Oliver2013} or \cite[Theorem~1]{GlaubermanLynd2016}. Then
using a diagram chase like that given in a five lemma for groups, one sees that
the penultimate column is also exact.

\begin{lemma}
\label{L:fusiontriv}
Let $\F$ be a saturated fusion system over $S$ with associated centric linking
system $\L$, and suppose that $\mu\colon \Out(\L) \to \Out(\F)$ is injective.
Then $\ker(\tilde{\mu}) = \{c_{\delta_S(z)} \mid z \in Z(S)\}$ consists of
automorphisms of $\L$ induced by conjugation by elements of $Z(S)$. 
\end{lemma}
\begin{proof}
By assumption on $\mu$, we see from \eqref{E:diagram} that $\lim^1(\Z_\F) = 0$
by the exactness of the third column.  The assertion follows from exactness
of the top row \eqref{E:diagram}, together with commutativity of the square
containing $Z(S)$ and $\Aut(\L)$.
\end{proof}

In the situation where $\F$ is realized by a finite group $G$ with Sylow
subgroup $S$, there are maps which compare certain automorphism groups of $G$
with the automorphism groups of $\L$ and $\F$.  For example, there is a group
homomorphism $\tilde{\kappa}_G\colon \Aut(G,S) \to \Aut(\L)$, where $\Aut(G,S)$
is the subgroup of $\Aut(G)$ consisting of those automorphisms which normalize
$S$.  Then $\tilde{\kappa}_G$ sends the image of $N_G(S)$ to
$\operatorname{Im}(c) \leq \Aut(\L)$, and so induces a homomorphism $\kappa
\colon \Out(G) \to \Out(\L)$. 

\begin{definition}\label{D:tame}
A finite group $G$ with Sylow subgroup $S$ is said to \emph{tamely realize}
$\F$ if $\F \cong \F_S(G)$ and the map $\kappa \colon \Out(G) \to \Out(\L)$ is
split surjective. The fusion system $\F$ is said to be \emph{tame} if it is
tamely realized by some finite group. 
\end{definition}

From work of Andersen-Oliver-Ventura and Broto-M{\o}ller-Oliver, the fusion
systems of all finite simple groups at all primes are now known to be tamely
realized by some finite group \cite[Section 3.3]{AschbacherOliver2016}. To give
one example of the importance of tameness for getting a hold of extensions of
fusion systems of finite quasisimple groups, we mention the following result of
Oliver that will be useful later.

\begin{theorem}
\label{T:reduct}
Let $\F$ be a saturated fusion system over the finite $p$-group $S$ and let
$\E$ be a normal subsystem over the subgroup $T \leq S$. Assume that $F^*(\F) =
O_p(\F)\E$ with $\E$ quasisimple and that $\E$ is tamely realized by the finite
group $H$. Then $\F$ is tamely realized by a finite group $G$ such that $F^*(G)
= O_p(G)H$. 
\end{theorem}
\begin{proof}
This is Corollary~2.5 of \cite{OliverReductions}. See also a correction and
a strengthening to the results of \cite{OliverReductions} in
\cite{OliverReductionsCor}.
\end{proof}

\subsection{Components and the generalized Fitting subsystem}\label{SS:Components}
Aschbacher \cite[Chapter~9]{AschbacherGeneralized} introduced components and
the generalized Fitting subsystem $F^*(\F)$ of $\F$. By analogy with the
definition for groups, a component is a subnormal subsystem of $\F$ which is
quasisimple. Here $\F$ is called quasisimple if $O^p(\F)=\F$ and $\F/Z(\F)$ is
simple. By \cite[9.8.2,9.9.1]{AschbacherGeneralized}, the generalized Fitting
subsystem of $\F$ is the central product of $O_p(\F)$ and the components of
$\F$. Moreover, for every set $J$ of component of $\F$, there is a well-defined
subsystem $\Pi_{\C\in J}\C$, which is the central product of the components in
$J$. Writing $E(\F)$ for the central product of all components of $\F$,
$F^*(\F)$ is the central product of $O_p(\F)$ with $E(\F)$. We will use the
following lemma.

\begin{lemma}\label{L:ConjugateComponents}
If $\C$ is a component of $\F$ over $T$ then the following hold:
\begin{itemize}
\item [(a)] $\m{C}$ is normal in $F^*(\F)$.
\item [(b)] If $\gamma\in\Hom_\F(T,S)$, then $\C^\gamma$ is a component of $\F$.
\end{itemize}
\end{lemma}

\begin{proof}
By definition of a component, $\m{C}$ is subnormal and thus saturated. As
mentioned above, by \cite[9.8.2,9.9.1]{AschbacherGeneralized}, $F^*(\F)$
is the central product of $O_p(\F)$ (more precisely $\F_{O_p(\F)}(O_p(\F))$)
and the components of $\F$. It is elementary to check that each of the central
factors in a central product of saturated fusion systems is normal. Hence,
every component of $\F$ is normal in $F^*(\F)$ and (a) holds. 

For the proof of (b) let $S_0\leq S$ such that $F^*(\F)$ is a fusion system
over $S_0$. The Frattini condition (applied to the normal subsystem $F^*(\F)$)
says that we can factorize $\gamma$ as $\gamma=\gamma_0\circ \alpha$ with
$\gamma_0\in\Hom_{F^*(\F)}(T,S_0)$ and $\alpha\in\Aut_\F(S_0)$. By (a),
$\C^{\gamma_0}=\C$ and thus $\C^\gamma=\C^\alpha$.  As $F^*(\F)$ is a normal
subsystem, $\alpha$ induces an automorphism of $F^*(\F)$. Thus, $\C^\alpha$ is
normal in $F^*(\F)$ as $\C$ is normal in $F^*(\F)$. So $C^\gamma=\C^\alpha$ is
subnormal in $\F$. Hence, $\C^\gamma$ is a component of $\F$, since
$\C^\gamma\cong \C$ is quasisimple.
\end{proof}

\begin{lemma}\label{L:DirectProductComponent}
Let $\F$ be a saturated fusion system which is the central product of saturated
subsystems $\F_1,\dots,\F_n$.  If $\C$ is a component of $\F$, then there
exists $i\in\{1,2,\dots,n\}$ such that $\C$ is a component of $\F_i$.
\end{lemma}
 
\begin{proof}
Assume that $\C$ is a component of $\F$ which, for all $i=1,\dots,n$, is not a
component of $\F_i$. Let $\C$ be a subsystem on $T\leq S$, and let $\F_i$ be a
subsystem on $S_i$ for $i=1,\dots,n$. Since $\F$ is the central product of
$\F_1,\dots,\F_n$, each of the subsystems $\F_1,\dots,\F_n$ is normal in $\F$.
So for each $i=1,\dots,n$, it follows from \cite[9.6]{AschbacherGeneralized}
and the assumption that $\C$ is not a component of $\F_i$ that $T$ centralizes
$S_i$. As $S=\Pi_{i=1}^nS_i$, this yields that $T$ centralizes $S$ and is thus
abelian. Now \cite[9.1]{AschbacherGeneralized} yields a contradiction to $\C$
being quasisimple.
\end{proof}

\subsection{Components of involution centralizers}

Suppose now that $\F$ is a saturated fusion system over a $2$-group $S$. If
$\F$ is of component type, then in analogy to the group theoretical case, one
wants to work with components of involution centralizers (or more generally
with components of normalizers of subgroups of $S$). In fusion systems, the
situation is slightly more complicated than in groups, since only components of
saturated fusion systems are defined. Therefore, we can only consider
components of normalizers of \emph{fully normalized} subgroups. It makes sense
to work also with conjugates of such components. Following Aschbacher
\cite[Section~6]{AschbacherFSCT} we will use the following notation.

\begin{notation}\label{N:XCIC}
If $\C$ is a quasisimple subsystem of $\F$ over $T$, then define the following sets:
\begin{itemize}
\item $\X(\C)$ is the set of subgroups or elements $X$ of $C_S(T)$ such that
$C_\F(X)$ contains $\C$. 
\item $\tilde{\X}(\C)$ is the set of subgroups or elements $X$ of $S$ such that
$\C^\alpha$ is a component of $N_\F(X^{\alpha})$ for some $\alpha \in \AA(X)$. 
\item $\I(\C)$ is the set of involutions in $\tilde{\X}(\C)$.  
\end{itemize}
If we want to stress that these sets depend on $\F$, we write $\X_\F(\C)$,
$\tilde{\X}_\F(\C)$ and $\I_\F(\C)$ respectively. Moreover, we write $\CC(\F)$
for the set of quasisimple subsystems $\C$ of $\F$ such that $\I(\C)$ is
nonempty.
\end{notation}

\begin{lemma}\label{L:CCconjugate}
Let $\C$ be a quasisimple subsystem of $\F$ over $T$ and $X\in\tilde{\X}(\C)$.
Then for any $\phi\in\Hom_\F(\<X,T\>,S)$ the following hold:
\begin{itemize}
\item[(a)] If $X^\phi\in\F^f$, then $\C^\phi$ is a component of $N_\F(X^\phi)$.
\item[(b)] We have $X^\phi\in\tilde{\X}(\C^\phi)$.
\end{itemize}
\end{lemma}
\begin{proof}
Assume first $X^\phi\in\F^f$. Let $\alpha\in\AA(X)$ such that $\C^\alpha$ is a
component of $N_\F(X^\alpha)$. By Lemma~\ref{AAnonempty}, there exists
$\beta\in\AA(X^\alpha)$ such that $X^{\alpha\beta}=X^\phi$. Then
$N_S(X^\alpha)^\beta=N_S(X^\phi)$ and $\beta$ induces an isomorphism from
$N_\F(X^\alpha)$ to $N_\F(X^\phi)$. So $\C^{\alpha\beta}$ is a component of
$N_\F(X^\phi)$. As $X^{\alpha\beta}=X^\phi$, the map
$\beta^{-1}\alpha^{-1}\phi$ is a morphism in $N_\F(X^\phi)$. Moreover
$\C^{\alpha\beta}$ is conjugate to $\C^\phi$ under $\beta^{-1}\alpha^{-1}\phi$.
Thus, $\C^\phi$ is a component of $N_\F(X^\phi)$ by
Lemma~\ref{L:ConjugateComponents}. This proves (a). If we drop the assumption
that $X^\phi\in\F^f$ and pick $\alpha\in\AA(X^\phi)$, then applying (a) with
$\phi\alpha$ in place of $\phi$ gives that $(\C^\phi)^\alpha=\C^{\phi\alpha}$
is a component of $N_\F(X^{\phi\alpha})$. This gives (b).  
\end{proof}

\begin{lemma}\label{L:CClocalsubsystem}
Let $\C$ be a quasisimple subsystem of $\F$ over $T$ and let
$X\in\tilde{\X}(\C)$ be a subgroup of $S$. Suppose we are given $Y\in\F^f$
satisfying $[X,Y]\leq X\cap Y$ and $\C\leq N_\F(Y)$. Then
$X\in\tilde{\X}_{N_\F(Y)}(\C)$. In particular, if $X$ has order $2$, then
$\C\in\CC(N_\F(Y))$. 
\end{lemma}
\begin{proof}
Let $\beta\in\AA_{N_\F(Y)}(X)$ so that $X^\beta\in N_\F(Y)^f$. Let
$\alpha\in\AA(X^\beta)$. Then by \cite[2.2.1,2.2.2]{AschbacherGeneration},
we have that $Y^\alpha\in N_\F(X^{\beta\alpha})^f$, $(N_S(Y)\cap
N_S(X^\beta))^\alpha=N_S(Y^\alpha)\cap N_S(X^{\beta\alpha})$, and $\alpha$
induces an isomorphism from $N_{N_\F(Y)}(X^\beta)$ to
$N_{N_\F(X^{\beta\alpha})}(Y^\alpha)$. By Lemma~\ref{L:CCconjugate}(a), we have
that $\C^{\beta\alpha}$ is a component $N_\F(X^{\beta\alpha})$. So by
\cite[2.2.5.2]{AschbacherFSCT}, $\C^{\beta\alpha}$ is a component of
$N_{N_\F(X^{\beta\alpha})}(Y^\alpha)$. As $\alpha$ induces an isomorphism from
$N_{N_\F(Y)}(X^\beta)$ to $N_{N_\F(X^{\beta\alpha})}(Y^\alpha)$, this implies
that $\C^\beta$ is a component of $N_{N_\F(Y)}(X^\beta)$. This proves
$X\in\tilde{\X}_{N_\F(Y)}(\C)$ and the assertion follows.
\end{proof}

\subsection{Pumping up}\label{SS:pumpup}

Crucial in the classification of finite simple groups of component type is the
Pump-Up Lemma, which leads to the definition of a maximal component. As we
explain in more detail in the next subsection, such maximal components have
very nice properties generically, which ultimately allow one to pin down the
group if the structure of a maximal component is known. 

The main purpose of this section is to state the Pump-Up Lemma for fusion
systems. However, to give the reader an intuition, we briefly want to describe
the Pump-Up Lemma for groups. Let $G$ be a finite group. To avoid technical
difficulties which do not play a role in the context of fusion systems, we
assume that none of the $2$-local subgroups of $G$ has a normal subgroup of odd
order. The results we state here are actually true for all almost simple
groups, but to show this one would have to use the $B$-theorem whose proof is
extremely difficult. Avoiding the necessity to prove the B-theorem is one of
the major reasons why it is hoped that working in the category of fusion
systems will lead to a simpler proof of the classification of finite simple
groups. 

Let $t$ be an involution of $G$. If $O(G)=1$, then the $L$-balance theorem of
Gorenstein and Walter gives that $E(C_G(t))\leq E(G)$, where $E(G)$ denotes the
product of the components of $G$. Further analysis shows that a component $C$
of $C_G(t)$ lies in $E(G)$ in a particular way. Namely, either $C$ is a
component of $G$, or there exists a component $D$ of $G$ such that $D=D^t$ and
$C$ is a component of $C_D(t)$, or there exists a component $D$ of $G$ such
that $D\neq D^t$ and $C=\{dd^t\colon d\in D\}$ is the homomorphic image of $D$
under the map $d\mapsto dd^t$. If one applies this property to the centralizer
of a suitable involution $a$ rather than to the whole group $G$, then one
obtains the Pump-Up Lemma. More precisely, consider two commuting involutions
$t$ and $a$ centralized by a quasisimple subgroup $C$ which is a component of
$C_G(t)$, and thus of $C_{C_G(a)}(t)$. The result stated above yields
immediately that one of the following holds: 

\begin{enumerate}
\item $C$ is a component of $C_G(a)$.
\item There exists a component $D$ of $C_G(a)$ such that $D=D^t$ and $C$ is a
component of $C_D(t)$.
\item There exists a component $D$ of $C_G(a)$ such that $D\neq D^t$ and
$C=\{dd^t\colon d\in D\}$ is a homomorphic image of $D$.  
\end{enumerate}
This statement is known as the Pump-Up Lemma. If (2) holds then $D$ is called a
\emph{proper pump-up} of $C$. The component $C$ is called \emph{maximal} if it
has no proper pump-ups.

We now state a similar result for fusion systems, which was formulated by
Aschbacher. Again, the statement is slightly more complicated than the
statement for groups, since we need to pass from an involution $a$ to a fully
centralized conjugate of $a$ for the centralizer to be saturated.

\begin{lemma}[\mbox{\cite[6.1.11]{AschbacherFSCT}}]\label{L:pumpupbasic}
Let $\F$ be a saturated fusion system over a $2$-group $S$ and let $\C$ be a
quasisimple subsystem of $\F$ on $T$. Suppose we are given two commuting
involutions $t,a\in S$ such that $t\in\I(\C)$ and $\gen{t,a}\in\X(\C)$. Fix
$\alpha \in \AA(a)$, and set $\bar{a} = a^\alpha$, $\bar{t} = t^{\alpha}$, and
$\bar{\C} = \C^\alpha$. Then one of the following holds:
\begin{enumerate}
\item (trivial) $\bar{\C}$ is a component of $C_\F(\bar{a})$, so $a \in \I(\C)$,
\item (proper) there is $\zeta \in \Hom_{C_\F(\bar{a})}(C_S(\gen{\bar{a}, \bar{t}}),
C_S(\bar{a}))$ and a $\bar{t}^{\zeta}$-invariant component $\D$ of
$C_\F(\bar{a})$ such that $\bar{\C}^{\zeta}$ is a component of
$C_{\D\gen{\bar{t}^\zeta}}(\bar{t}^{\zeta})$, and we have $\bar{\C}^{\zeta} \neq \D$, 
\item (diagonal) there is a component $\D$ of $C_\F(\bar{a})$ such that $\D
\neq \D^{\bar{t}}$, $\bar{\C} \leq \D_0 := \D\D^{\bar{t}}$, and $\C$ is a homomorphic image of $\D$.  
\end{enumerate}
\end{lemma}

\begin{definition}\label{D:pumpup}
Let $\F$ be a saturated $2$-fusion system and $\C\in\CC(\F)$.
\begin{itemize}
\item Whenever the hypotheses of Lemma~\ref{L:pumpupbasic} occur, and
$\D$ satisfies (2) of the conclusion, then $\D$ is a \emph{proper pump-up} of
$\C$. 
\item $\C$ is called \emph{maximal} (or a \emph{maximal component}) if it has no proper pump-ups.  
\end{itemize}
\end{definition}

\subsection{Standard components}\label{SS:standard}

We explain now in more detail how maximal components play a role in pinning
down the structure of a finite simple group $G$, and in how far these ideas
carry over to fusion systems. As in the previous subsection, we start by
explaining the basic ideas for groups. For that, assume again that $G$ is a
finite group in which no involution centralizer has a non-trivial normal
subgroup of odd order. 

Write $\CC(G)$ for the set of components of involution centralizers of $G$.
Using the Pump-Up Lemma, one can choose $C\in\CC(G)$ such that every element
$D\in\CC(G)$ which maps homomorphically onto $C$ is maximal. For such $C$,
Aschbacher's component theorem says basically that, with some ``small''
exceptions, either $C$ is a homomorphic image of a component of $G$, or the
following two conditions hold: 
\begin{itemize}
\item[(C1')] $C$ does not commute with any of its conjugates; and
\item[(C2')] if $t$ is an involution centralizing $C$, then $C$ is a component
of $C_G(t)$. 
\end{itemize}

Assuming that (C1') and (C2') hold and $C/Z(C)$ is a ``known'' finite simple
group, the structure of $G$ is determined case by case from the structure of
$C$. The problem of classifying $G$ from the structure of such a subgroup $C$
is usually referred to as a \emph{standard form problem}. The key to solving
such a standard form problem is that properties (C1') and (C2') imply that the
centralizer $C_G(C)$ is a tightly embedded subgroup of $G$ and thus has (by
various theorems in the literature) a very restricted structure if $G$ is
simple. Here a subgroup $K$ of $G$ of even order is called \emph{tightly
embedded} in $G$ if $K\cap K^g$ has odd order for any element $g\in G-N_G(K)$.
A \textit{standard subgroup} of $G$ is a quasisimple subgroup $C$ of $G$ such
that $C$ commutes with none of its conjugates, $K:=C_G(C)$ is tightly embedded
in $G$, and $N_G(C)=N_G(K)$. If $C$ is a component of an involution centralizer
which satisfies properties (C1') and (C2'), then it is straightforward to prove
that $C$ is a standard subgroup. So if $G$ is simple, then with some small
exceptions, Aschbacher's component theorem implies that there exists a standard
subgroup $C$ of $G$. 

We will now explain the theory of standard components of fusion systems, which
Aschbacher \cite{AschbacherFSCT} has developed roughly in analogy to the
situation for groups as far as this seems possible. For the remainder of this
subsection let $\F$ be a saturated fusion system over a $2$-group $S$, and let
$\C$ be a quasisimple subsystem of $\F$ on $T$.  The situation for fusion
systems is significantly more complicated, most importantly since the
definition of a standard component of a group involves a statement about its
centralizer, and the centralizer of $\C$ in $\F$ is currently only defined in
certain special cases. For example, Aschbacher has defined the normalizer and
the centralizer of a component of a fusion system \cite[Sections~2.1 and
2.2]{AschbacherFSCT}. In particular, if $\C$ is a component of $C_\F(t)$ for a
fully centralized involution $t$, then $C_{C_S(t)}(\C)$ is defined inside
$C_\F(t)$. If $\C\in\CC(\F)$, then this allows us to define a subgroup of $S$
which centralizes $\C$, dependent on an involution $t\in\I(\C)$. 

\begin{notation}[cf. (6.1.15) in \cite{AschbacherFSCT}]\label{N:PtalphaQt}
Let $\C$ be a quasisimple subsystem of $\F$ over $T$. 
If $t\in\I(\C)$ and $\alpha \in \AA(t)$, then define
\[
P_{t,\alpha} := C_{C_S(t^{\alpha})}(\C^\alpha) \cap C_S(t)^{\alpha}
\]
and 
\[
Q_t := Q_{t,\alpha} = P_{t,\alpha}^{\alpha^{-1}}
\]
\end{notation}

By \cite[6.6.16.1]{AschbacherFSCT}, $Q_{t,\alpha} \leq C_S(t)$ is independent
of the choice of $\alpha$ and so $Q_t$ is indeed well-defined. With this
definition in place, one can formulate conditions on $\C$ which roughly
correspond to conditions (C1') and (C2'). If $\C\in\CC(\F)$ fulfills such
conditions, then $\C$ is called \textit{terminal}. The precise definition is
given below in Definition~\ref{D:terminal}.

\begin{notation}[cf. (6.1.17) and (6.2.7) in \cite{AschbacherFSCT}]\label{N:DeltaC}
Let $\C$ be a quasisimple subsystem of $\F$ over $T$.
\begin{itemize}
\item $\Delta(\C)$ is the set of $\F$-conjugates $\C_1$ of $\C$ such that,
writing $T_1$ for the Sylow of $\C_1$, we have $T^\#_1 \subseteq \tilde{\X}(\C)$ and
$T^\# \subseteq \tilde{\X}(\C_1)$. 
\item $\rho(\C)$ is the set of pairs $(t^{\phi}, \C^{\phi})$ such that
$t \in \I(\C)$ and $\phi \in \Hom_\F(\gen{t,T},S)$. 
\end{itemize} 
\end{notation}

\begin{definition}[{\cite[Definition~8.1.1]{AschbacherFSCT}}]\label{D:terminal}
A subsystem $\C\in\CC(\F)$ over $T$ is called \emph{terminal} if the following conditions hold:
\begin{itemize}
\item[(C0)] $T \in \F^f$,
\item[(C1)] $\Delta(\C)=\varnothing$, and
\item[(C2)]If $(t_1,\C_1) \in \rho(\C)$, then $Q_{t_1}^\#\subseteq\tilde{\X}(\C_1)$.
\end{itemize}
\end{definition}

In this definition, property (C2) corresponds roughly to property (C2') above.
Moreover, assuming (C2), property (C1) should be thought of as roughly
corresponding to property (C1') above. By Lemma~\ref{L:CCconjugate}(b), for
any $(t_1,\C_1)\in\rho(\C)$, we have $t_1\in\I(\C_1)$  and so $Q_{t_1}$ is
well-defined, i.e. the statement in property (C2) makes sense.

Aschbacher proved a version of his component theorem for fusion systems
\cite[Theorem~8.1.5]{AschbacherFSCT}. Suppose $\C\in\CC(\F)$ is such that every
$\D\in\CC(\F)$ mapping homomorphically onto $\C$ is maximal. The component
theorem for fusion systems states essentially that, with some small exceptions,
either $\C$ is the homomorphic image of a component of $\F$, or $\C$ is
terminal. This statement is similar to the statement of the component theorem
in the group case. However, it is not clear that the centralizer of a terminal
component is defined and ``tightly embedded'' in $\F$. This makes it more
complicated to define standard subsystems. We will work with Aschbacher's
definition of a standard subsystem, which we state next. 

\begin{definition}[\!\!{\cite[Section~9.1]{AschbacherFSCT}}]\label{D:standard}
Let $\C$ be a quasisimple subsystem of $\F$ over $T \in \F^f$. Then $\C$ is 
a \emph{standard subsystem} of $\F$ if the following four conditions are
satisfied:
\begin{itemize}
\item [(S1)] $\tilde{\X}(\C)$ contains a unique maximal (with respect to
inclusion) member $Q$.
\item [(S2)] For each $1\neq X\leq Q$ and $\alpha\in\AA(X)$, we have
$\C^\alpha\unlhd N_\F(X^\alpha)$. 
\item [(S3)] If $1\neq X\leq Q$ and $\beta\in\AA(X)$ with $X^\beta\leq Q$, then
$T^\beta=T$.
\item [(S4)] $\Aut_\F(T)\leq \Aut(\C)$.
\end{itemize}
If $\C$ satisfies conditions (S1),(S2),(S3), then $\C$ is called \textit{nearly
standard}.
\end{definition}

\begin{remark}\label{R:QCentralizer}
In the above definition, the first condition (S1) says essentially that the
centralizer of $\C$ in $S$ is well-defined. Namely, the unique maximal member
$Q$ of $\tilde{\X}(\C)$ should be thought of as this centralizer. Given a
standard subsystem $\C$ of $\F$, this allows Aschbacher
\cite[Definition~9.1.4]{AschbacherFSCT} to define a saturated subsystem $\Q$ of
$\F$ over $Q$ which plays the role of the centralizer of $\C$ in $\F$. More
precisely, $\Q$ centralizes $\C$ in the sense that $\F$ contains a subsystem
which is a central product of $\Q$ and $\C$ (cf.
\cite[9.1.6.1]{AschbacherFSCT}). Also, by \cite[9.1.6.2]{AschbacherFSCT},
$\Q$ is a tightly embedded as defined in the next subsection (cf.
Definition~\ref{D:tightlyEmb}). We will refer to $\Q$ as the \emph{centralizer}
of $\C$ in $\F$.
\end{remark}

In general, it is difficult to get control of $C_S(T)$ when $T$ is the Sylow
subgroup of a member $\C$ of $\CC(\F)$. However, $C_S(T) \leq N_S(Q)$ when $\C$
is standard. This inclusion gives much needed leverage, as is shown in
Lemma~\ref{L:CST} below. Since the method of proof of that lemma is more widely
applicable, we next state and prove a more general result which we feel is of
independent interest. In the proof of Proposition~\ref{P:CST}, we reference a
\emph{normal pair} of linking systems $\L \norm \L_1$ as defined in
\cite[Definition~1.27]{AOV2012}. Also, we take the opportunity to write certain
maps on the left-hand side of their arguments. 

\begin{proposition}\label{P:CST}
Let $\F$ be a saturated fusion system over the $2$-group $S$, and let $\F_1$ be
a saturated subsystem over $S_1 \leq S$. Assume that $\C$ is a perfect normal
subsystem of $\F_1$ over $T \in \F^f$ having associated centric linking system
$\L$ such that
\begin{enumerate}
\item[(i)] $C_S(T) \leq S_1$, and
\item[(ii)] $\mu\colon \Out(\L) \to \Out(\C)$ is injective  (see Section~\ref{SS:aut}).
\end{enumerate}
Then $C_S(T) = C_{S_1}(\C)Z(T)$.
\end{proposition}
\begin{proof}
By assumption, $\C$ is normal in $\F_1$, so we may form the product system
$\C_1 := \C S_1$ in this normalizer, as in
\cite[Chapter~8]{AschbacherGeneralized} or \cite{Henke2013}.  Then $O^2(\C_1) =
O^2(\C) = \C$ since $\C$ is perfect, so by \cite[Proposition~1.31(a)]{AOV2012},
there is a normal pair of linking systems $\L \norm \L_1$ associated to the
pair $\C \norm \C_1$ in which $\Ob(\L_1) = \{P \leq S_1 \mid P \cap T \in
\C^c\}$. Note that not only is $\L$ is a subcategory of $\L_1$, but the
structural functors $\delta, \pi$ for $\L$ are the restrictions of the functors
for $\L_1$ by definition of an inclusion of linking systems. Because of this,
we write $\delta, \pi$ also for the structural functors for $\L_1$.

Now by the definition of a normal pair of linking systems
\cite[Definition~1.27(iii)]{AOV2012}, the conjugation map $c \colon
\Aut_{\L}(T) \to \Aut(\L)$ lifts to a map $\Aut_{\L_1}(T) \to \Aut(\L)$, which
we also denote by $c$. So the existence of the pair $\L \norm \L_1$ allows one
to define a homomorphism $\nu\colon S_1 \to \Aut(\L)$ given by the
composition $S_1 \xrightarrow{\delta_T} \Aut_{\L_1}(T) \xrightarrow{c}
\Aut(\L)$. This map has kernel
\begin{eqnarray}
\label{E:sem}
\ker(\nu) = C_{S_1}(\C)
\end{eqnarray}
by \cite[Theorem~A]{Semeraro2015}.

We can now prove the assertion. Clearly $C_{S_1}(\C)Z(T) \leq C_S(T)$. For the
reverse inclusion, fix $s \in C_S(T)$. Then $\nu$ is defined on $s$ by (i). The
map $\tilde{\mu}\colon \Aut(\L)\rightarrow \Aut(\C)$ is more precisely defined
by the equation $t^{\tilde{\mu}(\phi)}=\delta_T^{-1}(\delta_T(t)^{\phi_T})$ for
all $\phi\in\Aut(\L)$ and all $t\in T$. Using this for
$\phi=\nu(s)=c_{\delta_T(s)}$, we obtain for all $t\in T$ that
$t^{\tilde{\mu}(\nu(s))}=\delta_T^{-1}(\delta_T(s)^{-1}\circ
\delta_T(t)\circ\delta_T(s))=\delta_T^{-1}(\delta_T(t^s))=t^s=t$, where the
last equality uses $s\in C_S(T)$. The automorphism $\tilde{\mu}(\nu(s)) \in
\Aut(\C)$ is thus trivial. Hence by Lemma~\ref{L:fusiontriv} and assumption on
$\mu$, $\nu(s)=c_{\delta_T(z)}=\nu(z)$ for some $z\in Z(T)$. It follows that
$\nu(sz^{-1})$ is the identity on $\L$.  Hence, $sz^{-1} \in C_{S_1}(\C)$ by
\eqref{E:sem}, so $s \in C_{S_1}(\C)Z(T)$, which completes the proof.
\end{proof}

\begin{lemma}\label{L:CST}
Let $\F$ be a saturated fusion system over the $2$-group $S$. Suppose $\C$ is a
standard subsystem of $\F$ over $T$ with centralizer $\Q$ over $Q$. Let
$\L$ be a centric linking system associated to $\C$. If $\mu\colon \Out(\L) \to
\Out(\C)$ is injective, then $C_S(T) = QZ(T)$.
\end{lemma}
\begin{proof}
As $Q$ is fully $\F$-normalized by \cite[9.1.1]{AschbacherFSCT},
$N_{\F}(Q)$ is a saturated fusion system over $N_S(Q)$.  Further,
(S2) says that $\C$ is normal in $N_\F(Q)$.  Finally, by
\cite[Proposition~5]{AschbacherFSCT}, $\Q$ is normal in $N_\F(T)$ and
$C_{N_S(Q)}(\C) = Q$. In particular, $C_S(T) \leq N_S(Q)$. Thus, if $\mu$ is
injective, then $C_S(T) = C_{N_S(Q)}(\C)Z(T) = QZ(T)$ by
Proposition~\ref{P:CST}.
\end{proof}

When considering involution centralizer problems for fusion systems, we
generally need to verify in each individual case that the terminal component we
consider is standard. In contrast with the group case, this is not a
straightforward task.  Indeed, in some cases a terminal component is provably
not standard. However, it develops that many of these technically difficult
configurations are ultimately unnecessary to treat, because they arise for
example in almost simple fusion systems that are not simple, or in wreath
products of a simple system with an involution. Thus, we will usually need
assumptions that go beyond supposing merely that the structure of the terminal
component we consider is known, but also information about the embedding of
that subsystem in the ambient system. Such stronger assumptions can be made if
one classifies, as was proposed by Aschbacher, the simple ``odd systems''
rather than the simple fusion systems of component type (cf.
\cite{AschbacherFSCT}).  The hypothesis in Theorem~\ref{T:main} that $\C$ is
subintrinsic should be seen in this context.

\begin{definition}\label{D:subintrinsic}
Let $\C\in\CC(\F)$. Then $\C$ is said to be \emph{subintrinsic in} $\CC(\F)$ if
there exists $\H\in\CC(\C)$ such that $\I_\F(\H)\cap Z(\H)\neq\varnothing$.
\end{definition}

It follows in a fairly straightforward way from results of Aschbacher that a
subintrinsic Benson-Solomon component $\C$ is terminal.  Rather than use the
component theorem for fusion systems, it is more convenient in our case to show
that $\C$ is terminal using a part of the argument for
\cite[Theorem~7.4.14]{AschbacherFSCT}, which is a major ingredient of the proof
of the component theorem. As suggested above, a nontrivial amount of work is
then required to go on and show that $\C$ is standard; see
Section~\ref{S:showStandard}. 

\subsection{Tightly embedded subsystems and tight split extensions}\label{SS:TightSplit}

Recall from the previous subsection that a subgroup $K$ of a finite group $G$
is called \emph{tightly embedded} if $K$ has even order and $K\cap K^g$ has odd
order for every $g\in G\backslash N_G(K)$. This definition does not translate
well to fusion systems as it is, but there exist suitable reformulations. It
follows from Aschbacher \cite[0.7.1]{AschbacherFSCT} that a subgroup $K$ of
$G$ of even order is tightly embedded if and only if the following two
conditions hold:
\begin{itemize}
\item [(T1')] $K$ is normalized by $N_G(X)$ for every non-trivial $2$-subgroup
$X$ of $K$. 
\item [(T2')] For every involution $x$ of $K$, $x^G\cap K=x^{N_G(K)}$.
\end{itemize}

If $K$ is tightly embedded and $Q$ is a Sylow $2$-subgroup of $K$, then note
furthermore that $N_G(Q)\leq N_G(K)$ and $N_G(K)=KN_G(Q)$ by a Frattini
argument. This leads to a definition of tightly embedded subsystem of saturated
fusion systems at arbitrary primes.

\begin{definition}[cf.\mbox{\cite[Definition~3.1.2]{AschbacherFSCT}}]\label{D:tightlyEmb}
Let $\F$ be a saturated fusion system on a $p$-group $S$, and let $\Q$ be a
saturated subsystem of $\F$ on a fully normalized subgroup $Q$ of $\F$. Then
$\Q$ is \emph{tightly embedded} in $\F$ if it satisfies the following three
conditions:
\begin{enumerate}
\item[(T1)] For each $1 \neq X \in \Q^f$ and each $\alpha \in \AA(X)$, 
\[
O^{p'}(N_\Q(X))^{\alpha} \text{\,\, is normal in \,\,} N_\F(X^\alpha).
\]
\item[(T2)] For each $X \leq Q$ of order $p$, 
\[
X^{\F} \cap Q = X^{\Aut_\F(Q)\Q}
\]
where $X^{\Aut_\F(Q)\Q}:=\{X^{\alpha\phi} \mid  \alpha\in\Aut_\F(Q),\;\phi\in\Hom_\Q(X\alpha,Q)\}$.
\item[(T3)] $\Aut_\F(Q) \leq \Aut(\Q)$. 
\end{enumerate}
\end{definition}

When working with standard subsystems later on, we will need the following
lemma on tightly embedded subsystems.

\begin{lemma}\label{QQ}
Let $\F$ be a saturated fusion system on $S$, and suppose $\Q$ is a tightly
embedded subsystem of $\F$ on an abelian subgroup $Q$ of $S$. Then $\F_Q(Q)$ is
tightly embedded in $\F$.  
\end{lemma}
\begin{proof}
As $Q$ is abelian, by Alperin's fusion theorem (cf.
\cite[Theorem~I.3.6]{AschbacherKessarOliver2011}), the following holds:
\begin{itemize}
\item [(*)] The $p$-group $Q$, and thus the subsystem $\F_Q(Q)$, is normal in any saturated fusion system on $Q$.
\end{itemize}

Let $1\neq X\leq Q$ and $\alpha\in\AA(X)$. By (*), we have $Q=N_Q(X)\unlhd
N_\Q(X)$ and thus $Q=O^{p^\prime}(N_\Q(X))$. As $\Q$ is tightly embedded, it
follows $N_Q(X)^\alpha=Q^\alpha=O^{p^\prime}(N_\Q(X))^\alpha\unlhd
N_\F(X^\alpha)$. So (T1) holds for $\F_Q(Q)$.  

Let $X\leq Q$ be of order $p$. Again using (*), we have $Q\unlhd\Q$. So every
morphism in $\Q$ extends to an element of $\Aut_\Q(Q)\leq \Aut_\F(Q)$, and this
implies $X^{\Aut_\F(Q)\Q}=X^{\Aut_\F(Q)}$. Hence, as $\Q$ is tightly embedded,
$X^\F\cap Q=X^{\Aut_\F(Q)\Q}=X^{\Aut_\F(Q)}=X^{\Aut_\F(Q)\F_Q(Q)}$. This shows
that (T2) holds for $\F_Q(Q)$. Clearly (T3) holds for $\F_Q(Q)$.  
\end{proof}

To exploit the existence of standard subsystems, it is useful in many
situations to study certain kinds of extensions involving tightly embedded
subsystems. We summarize the main definitions:

\begin{definition}\label{D:Split}
Let $\F_0$ be a fusion system on a $2$-group $S_0$. 
\begin{itemize}
\item A \emph{split extension} of $\F_0$ is a pair $(\F,U)$, where
\begin{itemize}
\item $\F$ is a saturated fusion system over a $2$-group $S$, 
\item $\F_0$ is normal in $\F$,
\item $O^2(\F) = O^2(\F_0)$, and
\item $U$ is a complement to $S_0$ in $S$. 
\end{itemize}
\item The split extension $(\F,U)$ is \emph{tight} if $\F_U(U)$ is tightly
embedded in $\F$. 
\item A \emph{critical split extension} is a tight split extension in which $U$
is a four group.
\item $\F_0$ is said to be \emph{split} if there exists no nontrivial critical
split extension of $\F_0$; that is, for each such extension $(\F,U)$, the
fusion system $\F$ is the central product of $\F$ with $C_S(\F_0)$. 
\end{itemize}
\end{definition}

Suppose $\F$ is a saturated $2$-fusion system and $\C$ is a standard component
with centralizer $\Q$ on $Q$. If $\C$ is split, then by
\cite[Theorem~8]{AschbacherFSCT}, $\C$ is either a component of $\F$, or $Q$ is
elementary abelian, or the $2$-rank of $Q$ equals $1$.  We show in
Lemma~\ref{L:Csplit} that the Benson--Solomon fusion systems are split. So
after showing that a component $\C$ as in Theorem~\ref{T:main} is standard, we
know that, unless $\C$ is a component of $\F$, its centralizer $Q$ in $S$ is
either elementary abelian or quaternion or cyclic. Accordingly, these are the
cases we will treat. 

\begin{lemma}
\label{L:critsplitbasic}
Let $\C$ be a quasisimple saturated fusion system over the $2$-group $T$, and
let $(\F,U)$ be a critical split extension of $\C$ over the $2$-group $S$. Then
\begin{enumerate}
\item[(a)] $\Aut_\F(U) = 1$ and so $N_\F(U) = C_\F(U)$; and
\item[(b)] $\gen{u} \in \F^f$ and $C_\F(u) = C_\F(U)$ for each $1 \neq u \in U$. 
\end{enumerate}
\end{lemma}
\begin{proof}
By definition of critical split extension, $U$ is a four subgroup of $S$
tightly embedded in $\F$ and a complement to $T$ in $S$. Also, $O^2(\F) =
O^2(\C) = \C$, as $\C$ is quasisimple. Since $O^2(\F) = \C$, this means
$\hyp(\F) = T$. Since $S/\hyp(\F) \cong U$ is abelian, we see from
\cite[Lemma~I.7.2]{AschbacherKessarOliver2011} that also $\foc(\F) = T$.  Thus,
$\Aut_\F(U) = 1$, since otherwise $T \cap U = \foc(\F) \cap U \geq
[U,\Aut_\F(U)] > 1$, which is not the case. This proves the first assertion in
(a), and the second then follows from the definitions of the normalizer
and centralizer systems. 

Now by definition of tight embedding, $U$ is fully normalized in $\F$. Fix $1
\neq u \in U$. By (T2) and part (a), it follows that $u^\F \cap U = \{u\}$.
However, (3.1.5) of \cite{AschbacherFSCT} says that $\gen{u}$ has a fully
$\F$-normalized $\F$-conjugate in $U$, so $\gen{u} \in \F^f$.  Then taking
$\alpha$ to be identity in (T1), we see that $U$ is normal in $N_\F(\gen{u}) =
C_\F(u)$, so that $C_\F(u) \leq N_\F(U) = C_\F(U)$ by (a).  This completes the
proof of (b), as the other inclusion is clear.
\end{proof}

\subsection{The fusion system of $\Spin_7(q)$ and $\F_{\Sol}(q)$}\label{SS:spinsol}

Our main references for $\F_{\Sol}(q)$ and for $2$-fusion systems of
$\Spin_7(q)$ are \cite{LeviOliver2002, LeviOliver2005,
ChermakOliverShpectorov2008, AschbacherChermak2010, HenkeLynd2018}.

We follow Section 4 of Aschbacher and Chermak fairly closely
\cite{AschbacherChermak2010} within this subsection, except that it will be
convenient to restrict the choice of the finite fields $\mathbf{F}_q$ over
which the systems in question are defined, and to make small changes to
notation.  For concreteness, we consider a fixed but arbitrary nonnegative
integer $l$, and set $q_l = 5^{2^l}$. Except for the fact that $l$ is in the
role of ``$k$'', we adopt the notation in
\cite[Section~4]{AschbacherChermak2010}, as follows.

Let $\tilde{\mathbf{F}}$ be an algebraic closure of the field with $5$ elements
(thus, we take $p = 5$ in \cite[Section~4]{AschbacherChermak2010}), and let
$\mathbf{F}$ be the union of the subfields of the form $\mathbf{F}_{5^{2^{n}}}$
in $\tilde{\mathbf{F}}$. Let $H = \Spin_7(\mathbf{F})$, let $T$ be a maximal
torus of $H$, and let $T_{2^\infty}$ be the $2$-power torsion subgroup of $T$.

Let $W$ be the subgroup of $H$ defined on page 911 of
\cite{AschbacherChermak2010}, let $W_S$ be the subgroup of $W$ defined on page
915 of \cite{AschbacherChermak2010}, and set $S = T_{2^\infty} W_S$. Thus,
$S \cap W = W_S$.  The subgroup $B$ of $H$ is defined just before Lemma~4.4 of
\cite{AschbacherChermak2010} as the normalizer of the unique normal four
subgroup $U$ of $S$ (see Notation~\ref{N:spin}).  Finally, the group $K$ is
defined at the top of page 918 as a certain semidirect product of the connected
component $B^0$ of $B$ with a subgroup $\gen{y,\tau} \cong S_3$, such that
$\tau \in B$ is of order $2$, and such that $y$ is of order $3$ and permutes
transitively the involutions in $U$.  A free amalgamated product $G = H *_B K$
having Sylow $2$-subgroup $S$ is then defined by an amalgam in
\cite[Section~5]{AschbacherChermak2010}, which is ultimately constructed
at the top of page 923. Here, Sylow $p$-subgroups of infinite groups are
defined in \cite[Definition~1.4]{AschbacherChermak2010} generalizing properties
of Sylow $p$-subgroups of finite groups in a natural way.

Let $\psi$ be the Frobenius endomorphism of $H$ as defined in (4.2.2) of
\cite{AschbacherChermak2010} and inducing the $5$-th power map on $T$, and set
$\psi_l = \psi^{2^l}$ for each integer $l \geq 0$.  

\begin{theorem}
\label{T:uniqueliftsigma}
For any choice of integer $l \geq 0$, the automorphism $\psi_l$ of $H$ lifts
uniquely to an automorphism $\sigma_l$ of the group $G$ that commutes with
$y$, and hence an automorphism that leaves $K$ invariant. Moreover,
$C_S(\sigma_l)$ is then a finite Sylow $2$-subgroup of
$C_G(\sigma_l)$, and $\F_{C_S(\sigma_l)}(C_G(\sigma_l))$ is
isomorphic to $\F_{\Sol}(q_l)$. 
\end{theorem}
\begin{proof}
The first statement here is shown in Lemma~5.7 of \cite{AschbacherChermak2010}
and in the paragraph before it. The second part is Theorem~A(3) of
\cite{AschbacherChermak2010}: it is shown in
\cite[Lemma~7.4(b)]{AschbacherChermak2010} that $C_G(\sigma_l)$ is isomorphic to
$C_H(\sigma_l) *_{C_B(\sigma_l)} C_K(\sigma_l)$ and in
\cite[Lemma~7.5(b)]{AschbacherChermak2010} that $C_S(\sigma_l)$ is a Sylow
$2$-subgroup of $C_G(\sigma_l)$. Ultimately it is then shown in Theorem~9.9
that $\F_{C_S(\sigma_l)}(C_G(\sigma_l))$ is isomorphic to the fusion system
$\F_{\Sol}(q_l)$ defined by Levi and Oliver in \cite{LeviOliver2002,
LeviOliver2005}. 
\end{proof}

For the automorphism $\sigma_l$ of $G$ given by
Theorem~\ref{T:uniqueliftsigma} and for any subgroup $X \leq G$, we write
$X_{\sigma_l}$ for $C_X(\sigma_l)$. 

\begin{lemma}
\label{L:sylowchoice}
Let $\sigma_l$ be the unique lift of $\psi_l = \psi^{2^l}$ to $G$.  Then for
each such $l \geq 0$,
\begin{enumerate}
\item[(a)] $X$ and $X_{\sigma_l}$ are invariant under $\sigma_0$ for each $X \in
\{H,B,K,S,W,T\}$, 
\item[(b)] $\sigma_l$ centralizes $W$, and
\item[(c)] $S_{\sigma_l} = (T_{2^\infty})_{\sigma_l}W_S$, and $S_{\sigma_l}$ is
a Sylow $2$-subgroup of $H_{\sigma_l}$. 
\end{enumerate}
\end{lemma}
\begin{proof}
Except for the case $X = K$, all parts of (a) follow from the definitions in
\cite[Section~4]{AschbacherChermak2010}, since the statement just expresses
invariance of $X$ under $\psi$ for $X \leq H$. By uniqueness of the lift in
Theorem~\ref{T:uniqueliftsigma}, $\sigma_0^{2^l} = \sigma_l$, so
$[\sigma_0,\sigma_l] = 1$ in particular.  Hence, $(K_{\sigma_l})^{\sigma_0} =
C_{K^{\sigma_0}}(\sigma_l) = K_{\sigma_l}$.  Part (b) is proved in
\cite[Lemma~4.3]{AschbacherChermak2010}, while part (b) is proved in
\cite[Lemma~4.9]{AschbacherChermak2010}.
\end{proof}

Thus, as $H$ is of universal type, 
\[
H_{\sigma_l} = C_{H}(\sigma_l) = C_H(\psi_l)=\Spin_7(q_l),
\]
and so part (a) of the lemma says for example that $(H_{\sigma_l})_{\sigma_0} =
\Spin_7(5)$. Also note that $\psi_l$ acts on $T$ as $t \mapsto t^{q_l}$ (by
definition of $\psi$), so $T_{\sigma_l}$ is a split maximal torus of
$H_{\sigma_l}$, isomorphic to $(C_{q_l-1})^{3}$. 

\begin{notation}[$\F_{\Spin}(q)$ and $\F_{\Sol}(q)$]
\label{N:Fsigma}
Fix $l \geq 0$ and set $\sigma = \sigma_l$ for short.  Write 
\[
\H_\sigma := \F_{S_\sigma}(H_\sigma) = \F_{\Spin}(q) \quad \text{ and } \quad
\F_\sigma := \F_{S_\sigma}(G_\sigma) = \F_{\Sol}(q). 
\]
Let $Z := Z(S_\sigma)$, a group of order $2$. Write $z$ for the
involution in $Z$. 
\end{notation}

Thus, by Theorem~A(3) of \cite{AschbacherChermak2010}, $\F_\sigma$ is
isomorphic to the exotic fusion system defined by Levi and Oliver in
\cite{LeviOliver2002,LeviOliver2005}, and $C_{\F_\sigma}(z) = \H_\sigma$.

We continue to set up notation for some common subgroups of $S_\sigma$, and we
recall the various parts of the set up appearing in \cite[\S
4]{AschbacherChermak2010} that are needed later.

\begin{notation}[Some subgroups of $S_\sigma$]\label{N:spin}
Set $k := l+2$, and set $T_{2^k} := (T_{2^\infty})_\sigma \cong
(C_{2^k})^3$. This is the $2^k$-torsion subgroup of $T_{2^\infty}$ and a Sylow
$2$-subgroup of the finite abelian group $T_\sigma$.  Let $w_0 \in W_S$ be the
element of order $2$ fixed in \cite[Lemma~4.3]{AschbacherChermak2010}. Thus,
$w_0$ is centralized by $\sigma$ by Lemma~\ref{L:sylowchoice}(b) and $w_0$
inverts $T_{2^k}$.  The $2$-group $S_\sigma$ has a sequence of
elementary abelian subgroups
\[
1 < Z < U < E < A,
\]
each of index $2$ in the next, with $Z = Z(S_\sigma)$ as above, $U$ the unique
normal four subgroup of $S_\sigma$, $E = \Omega_1(T_{2^k})$, and $A =
E\gen{w_0}$, an elementary abelian subgroup of order $16$. We also set
$R_\sigma = C_{S_\sigma}(E) = T_{2^k}\gen{w_0} = T_{2^k}A$.  
\end{notation}

\medskip
\textbf{We adopt Notation~\ref{N:Fsigma} and \ref{N:spin} for the remainder of this subsection.}

\medskip
The following lemma collects a number of properties of these subgroups and
their automorphism groups. 
\begin{lemma}
\label{L:standardsequence}
The following hold. 
\begin{enumerate}
\item[(a)] For each $k_0 \geq 2$, $T_{2^{k_0}}$ is the unique homocyclic
abelian subgroup of $S$ of rank $3$ and exponent $2^{k_0}$, and
$T_{2^{k_0}}$ is inverted by $w_0$. 
\item[(b)] $T_{2^k}$ is $\F_\sigma$-centric, $S_\sigma/T_{2^k} \cong C_2 \times D_8$, and
$\Aut_{\F_\sigma}(T_{2^k}) \cong C_2 \times GL_3(2)$.
\item[(c)] $R_\sigma$ is characteristic in $S_\sigma$, and $\Out_{\F_\sigma}(R_\sigma) \cong GL_3(2)$. 
\item[(d)] $A$ is an elementary abelian subgroup of $S_\sigma$ of maximum
order, and so $S_\sigma$ has $2$-rank $4$. 
\item[(e)] $\Aut_{\F_\sigma}(X) = \Aut(X)$ for $X \in \{Z,U,E,A\}$. 
\end{enumerate}
\end{lemma}
\begin{proof}
By Lemma~4.9(b) of \cite{AschbacherChermak2010}, $T_{2^2} \leq T_{2^{k_0}}$ is the
unique homocyclic subgroup of $S$ of rank $3$ and exponent $4$. Moreover,
for the maximal torus $T$ of $H$, we have $T = C_{H}(T_{2^2})$, and
$S_\sigma \cap T = T_{2^k}$  is of rank $3$ and of exponent $2^k$.  This shows
that $T_{2^2}$, and more generally, $T_{2^{k_0}} = \Omega_{k_0}(T_{2^k})$
for $2 \leq k_0 \leq k$ is the unique subgroup of $S$ of its isomorphism type.
Also, $w_0$ inverts $T$ by \cite[Lemma~4.3(a)]{AschbacherChermak2010}.  This
completes the proof of (a).  Again as $T = C_H(T_{2^2})$, one has
$C_{H_\sigma}(T_{2^2}) = T_\sigma = T_{2^k} \times O(T_\sigma)$, and it follows
that $T_{2^k}$ is $\F_\sigma$-centric.  The
second statement in part (b) follows from
\cite[Lemma~4.3(c)]{AschbacherChermak2010}, while the third is the content of
\cite[Theorem~5.2]{AschbacherChermak2010}.

For the proof of (c), note first that $T_{2^k}$ is characteristic in $S$ by (a). So
also $R_\sigma = C_S(\Omega_1(T_{2^k}))$ is characteristic in $S$.  Finally, as
$T_{2^k}$ is fully $\F_\sigma$-normalized by (a) and as $R_\sigma/T_{2^k}$ is of order
$2$ and induces $O_2(\Aut_{\F_\sigma}(T_{2^k}))$ on $T_{2^k}$, the restriction map
$\rho\colon \Aut_{\F_\sigma}(R_\sigma) \to \Aut_{\F_\sigma}(T_{2^k})$ is surjective
by the Extension Axiom. Let $\phi \in \ker(\rho)$.  Then by
\cite[Lemma~A.8]{BrotoLeviOliver2003} and the first statement in (b), $\phi$ is
conjugation by an element of $Z(T_{2^k}) = T_{2^k}$.  It follows that $\ker(\rho) =
\Aut_{T_{2^k}}(R_\sigma)$ is of index $2$ in $\Inn(R_\sigma)$. Thus,
$\Out_{\F_\sigma}(R_\sigma) \cong
\Aut_{\F_\sigma}(T_{2^k})/O_2(\Aut_{\F_\sigma}(T_{2^k})) \cong GL_3(2)$ by the last
statement in (b).

Now as $E = \Omega_1(T_{2^k})$ is elementary abelian of order $8$ by (a), and $w_0$
inverts $T_{2^k}$, it follows that $A$ is elementary abelian of order $16$.  There
are no elementary abelian subgroups of $S$ of rank $5$ by
\cite[Lemma~7.9(a)]{AschbacherChermak2010}, so (d) holds. Finally, we refer to
Lemma~3.1 of \cite{LeviOliver2002} for the $\F_\sigma$-automorphism groups of
$X \in \{Z,U,E,A\}$, where $A$ is denoted ``$E^*$''. 
\end{proof}

\begin{lemma}
\label{L:involutionsconjugate}
All involutions in $S_{\sigma}$ are $\F_{\sigma}$-conjugate. 
\end{lemma}
\begin{proof}
This is a direct consequence of the construction of these systems
\cite[Theorem~2.1]{LeviOliver2002} and can be seen as follows.  Each involution
in $H_\sigma-Z$ has $-1$-eigenspace of dimension $4$ on the orthogonal space
for $H_\sigma/Z(H_\sigma)$ by \cite[Lemma~A.4(b)]{LeviOliver2002}. It follows
from this that $H_\sigma$ has two conjugacy classes of involutions, namely the
classes in $Z$ and in $S_{\sigma}-Z$.  In $\F_{\sigma}$ these two classes
become fused by construction (c.f.  Lemma~\ref{L:standardsequence}(e)). 
\end{proof}

\begin{lemma}
\label{L:lim1=0}
Let $\F \in \{\F_{\sigma}, \H_\sigma\}$, and let $\L$ be the centric linking
system for $\F$.  Then the natural map $\mu\colon \Out(\L) \to \Out(\F)$ is an
isomorphism.
\end{lemma}
\begin{proof}
This follows from \cite[Lemma~3.2]{LeviOliver2002} and the obstruction sequence
in \cite[Proposition~5.12]{AschbacherKessarOliver2011} (that is, from
\eqref{E:diagram} above). 
\end{proof}

The choice of $q_l = 5^{2^l}$ is motivated by the next two lemmas, especially
Lemma~\ref{L:tameSpin}(a).

\begin{lemma}
\label{L:tameSpin}
Let $\H_q$ be the $2$-fusion system of $\Spin_7(q)$ for some odd $q$, let $l+3$
be the $2$-adic valuation of $q^2-1$, and set $H_\sigma = \Spin_7(q_l)$ as
above. Then the following hold. 
\begin{enumerate}
\item[(a)] $\H_q$ is tamely realized by $H_\sigma$. 
\item[(b)] With $R_\sigma$ as in Notation~\ref{N:spin}, each automorphism of
$H_\sigma$ that normalizes $S_\sigma$ and centralizes $R_\sigma$ is conjugation
by an element of $E = Z(R_\sigma)$.
\end{enumerate}
\end{lemma}
\begin{proof}
Since $q^2-1$ and $q_l^2-1$ have the same $2$-adic valuation, the fusion systems
$\H_q$ and $\H_{\sigma}$ are isomorphic by
\cite[Theorem~A(a,c)]{BrotoMollerOliver2012}.  The composition $\Out(H_\sigma)
\to \Out(\H_\sigma)$ of $\mu$ with $\kappa$  (see Section~\ref{SS:aut}) is an
isomorphism with $q_l = 5^{2^l}$ by
\cite[Proposition~5.16]{BrotoMollerOliver2019}. Thus, $\H_q$ is tamely realized
by $H_{\sigma}$ by Lemma~\ref{L:lim1=0} and the definition of tame
(Definition~\ref{D:tame}).

Set $k = l+2$ as before. For the sake of convenience, we make appeals to
\cite[\S5]{BrotoMollerOliver2019} also for (b).  Note that by choice of $q_l$,
$H_\sigma$ satisfies Hypotheses~5.1(III.1) of that reference.  Let $\alpha$ be
an automorphism of $H_\sigma$ that normalizes $S_\sigma$ and centralizes
$R_\sigma$.  Since $R_\sigma \geq T_{2^k}$, $\alpha$ centralizes $T_{2^k}$. Thus, by
\cite[Lemma~5.9]{BrotoMollerOliver2019}, $\alpha \in \Inndiag(H_\sigma) =
\Inn(H_\sigma)\Aut_{T}(H_\sigma)$, and so there is $h \in H_\sigma$ and $t \in
T$ such that $\alpha$ is conjugation by $ht$. Then also $h \in C_{H}(T_{2^k}) = T$,
with the last equality by \cite[Lemma~5.3(a)]{BrotoMollerOliver2019}, so that
$ht \in T$. However, $R_\sigma$ contains the element $w_0$ inverting $T_{2^k}$,
c.f.  Lemma~\ref{L:standardsequence}(a), and so it follows that $ht \in
\Omega_1(T_{2^k}) = Z(R_\sigma)$.
\end{proof}

\begin{lemma}
\label{L:SolEquivalences}
The following hold.
\begin{enumerate}
\item[(a)] The collection $\{\F_{\Sol}(q_l) \mid l \geq 0\}$ gives a
nonredundant list of the isomorphism types of the $2$-fusion systems
$\F_{\Sol}(q)$ as $q$ ranges over odd prime powers. 
\item[(b)] The collection $\{\F_{\Spin}(q_l) \mid l \geq 0\}$ gives a
nonredundant list of the isomorphism types of the $2$-fusion systems
$\F_{\Spin}(q)$ as $q$ ranges over odd prime powers. 
\end{enumerate}
\end{lemma}
\begin{proof}
Part (a) is the content of \cite[Theorem~B]{ChermakOliverShpectorov2008}.  For
each odd prime power $q$, the fusion system of $\Spin_7(q)$ is isomorphic to
some fusion system in the given collection by Lemma~\ref{L:tameSpin}(a).  Then
(b) follows as a Sylow $2$-subgroup of $\Spin_7(q_l)$ has order $2^{10+3l}$ by
Lemma~\ref{L:standardsequence}(a,b).
\end{proof}

The next lemma shows that $\F_\sigma$ has just one more essential subgroup in
addition to the essential subgroups of $\H_\sigma$. 

\begin{lemma}
\label{L:essentials}
Let $P \in \F_{\sigma}^e$ be an essential subgroup. Then one of the following holds.
\begin{enumerate}
\item[(a)] $\Aut_{\F_\sigma}(P) = \Aut_{\H_\sigma}(P)$ and $P$ is $\H_{\sigma}$-essential, or
\item[(b)] $P = C_{S_{\sigma}}(U)$, $\Aut_{\F_{\sigma}}(P) =
\gen{\Aut_{\H_{\sigma}}(P), c_{y}}$, and $\Out_{\F_\sigma}(P) \cong S_3$. 
\end{enumerate}
\end{lemma}
\begin{proof}
Recall that an essential subgroup in a fusion system is in particular both
centric and radical. In \cite{LyndSemeraro}, the centric radical subgroups and
their outer automorphism groups in $\H_\sigma$ and $\F_\sigma$ are explicitly
tabulated. From Tables~1 and 4 there, the only outer automorphism groups having
a strongly embedded subgroup are $S_3$ and a Frobenius group of order $3^2
\cdot 2$. In all cases, either $P$ is essential in $\H_\sigma$ and
$\Out_{\F_\sigma}(P) = \Out_{\H_\sigma}(P)$ so that (a) holds, or $P =
C_{S_\sigma}(U)$ and $\Out_{\F_\sigma}(P) \cong S_3$. In the latter case,
$\Aut_{\F_{\sigma}}(P)$ is generated by $\Aut_{\H_{\sigma}}(P)$ and $c_y$
essentially by the construction of $y$ in Section~5 of
\cite{AschbacherChermak2010}, but it is difficult to find a precise statement
of this claim.  So instead, we appeal to \cite{LeviOliver2005} where $C_{S_\sigma}(U)$
is denoted $S_0(q^{\infty})$ on p. 2400, and where $c_y$ is explicitly
constructed as the automorphism $\widehat{\gamma}_u$ of $C_{S_\sigma}(U)$ in
\cite[Definition~1.6]{LeviOliver2005}.  There, $\Gamma_n$ is used to denote
$\Aut_{\F_{\sigma}}(C_{S_\sigma}(U))$ when $n = 2^l$. 
\end{proof}

The following generation statements will be needed in the process of showing
that a subintrinsic maximal Benson-Solomon subsystem is standard.  The
generation statement of Lemma~\ref{L:gen}(a) is the one which is obtained by
the construction by Levi and Oliver in \cite{LeviOliver2005}. 

\begin{lemma}
\label{L:gen}
The following hold.
\begin{enumerate}
\item[(a)] $\F_\sigma$ is generated by $\H_\sigma$ and $\Aut_{\F_\sigma}(C_{S_{\sigma}}(U))$.
\item[(b)] $\F_{\sigma}$ is generated by $\H_\sigma$ and $N_{\F_\sigma}(R_\sigma)$. 
\end{enumerate}
\end{lemma}
\begin{proof}
Part (a) follows from Lemma~\ref{L:essentials} and the Alperin-Goldschmidt
fusion theorem \cite[Theorem~I.3.6]{AschbacherKessarOliver2011}. As $U \leq E$,
$R_\sigma = C_{S_{\sigma}}(E) \leq C_{S_\sigma}(U)$. Further, $T_{2^k}$ is abelian
and weakly $\F_\sigma$-closed by Lemma~\ref{L:standardsequence}(a), hence also
$R_\sigma = C_{S_\sigma}(\Omega_1(T_{2^k}))$ is weakly $\F_\sigma$-closed.  Thus,
each element of $\Aut_{\F_{\sigma}}(C_{S_\sigma}(U))$ restricts to normalize
$R_\sigma$, and hence lies in $N_{\F_\sigma}(R_{\sigma})$.  This shows that (b)
is a consequence of (a). 
\end{proof}

The next lemma augments the results of \cite{HenkeLynd2018} on automorphisms
and extensions of the Benson-Solomon systems. 

\begin{proposition}
\label{P:fieldconjugate}
Let $\D$ be a saturated fusion system over the $2$-group $D$ such that $F^*(\D)
= \F_\sigma = \F_{\Sol}(q_l)$. 
\begin{enumerate}
\item[(a)] All involutions in $D-S_\sigma$ are $C_\D(z)$-conjugate, hence
$\D$-conjugate. 
\item[(b)] If $f \in D-S_\sigma$ is a fully $\D$-centralized involution, then
\[
C_{\F_{\sigma}}(f) = \F_{\sigma_{l-1}} \cong \F_{\Sol}(q_{l-1}). 
\]
\end{enumerate}
\end{proposition}

\begin{proof}
The almost simple extensions of $\F_\sigma$ were determined in
\cite{HenkeLynd2018}.  By \cite[Theorem~3.10, Theorem~4.3]{HenkeLynd2018}, we
have $O^2(\D) = \F_{\sigma}$. We may further fix a complement $F$ to $S_\sigma$
in $D$ such that $F$ is cyclic of order $2^{l_0}$ with $1 \leq l_0 \leq l$, and
such that the conjugation action of $F$ on $S_\sigma$ is the restriction of the
conjugation action of a group of field automorphisms of $H_\sigma$ to
$S_\sigma$. We may thus assume that $|F| = 2$, and that $F$ is generated by the
restriction of $\sigma_{l-1}$ to $S_\sigma$.  Write $f$ for the
automorphism $\sigma_{l-1}|_{H_\sigma}$ of $H_{\sigma}$, and also write
$f$ for its restriction to $S_\sigma$. Note that $\sigma_{l-1}$
normalizes $H_\sigma$ and $S_\sigma$ by Lemma~\ref{L:sylowchoice}(a). We
also rely on Lemma~\ref{L:sylowchoice}(a) at many places in the proof
below, without explicitly saying so.  By Lemma~\ref{L:tameSpin}(a) and
Theorem~\ref{T:reduct},
\begin{eqnarray} \label{E:CDzsemidirect} \text{$C_\D(z)$ is the fusion system
of the extension $H_\sigma\gen{f}$,} \end{eqnarray}
a semidirect product.

Recall $k = l+2$ as before, and let $H_1 = N_{H}(H_\sigma)$. Then
$H_1/Z(H_\sigma) \cong \Inndiag(H_\sigma)$ by \cite[Lemma~2.5.9(b)]{GLS3}, and
hence $H_1 = H_\sigma N_{T}(H_\sigma)$.  As $\Outdiag(H_\sigma)$ has order $2$,
we may fix $t \in N_{T}(H_\sigma)-H_\sigma$ with order $2^{k+1}$ and powering
to $z$, so that $H_1 = H_\sigma\gen{t}$.  As $\sigma_{l-1}$ normalizes
$H_\sigma$ and $T$, it restricts to an automorphism of $H_1$. Set $g =
\sigma_{l-1}|_{H_1}$, so that $g$ has order $4$, and $g|_{H_\sigma} = f$
has order $2$.  Set $J_1 := H_1\gen{g}$, $J := H_\sigma\gen{f}$, $\widehat{J}_1
= J_1/C_{J_1}(H_\sigma)$, and $\widetilde{J} = J/Z(H_\sigma)$.  Note that
$C_{J_1}(H_\sigma) = \gen{g^2,z}$.  Since $g^2$ centralizes $H_\sigma$,
$\gen{g^2}$ is normal in $H_\sigma\gen{g}$, and $H_\sigma\gen{g}/\gen{g^2}
\cong H_\sigma\gen{f}$ via an isomorphism which sends $g\gen{g^2}$ to $f$.
Hence, there is an isomorphism $\widehat{H}_\sigma\gen{\hat{g}} \to
\widetilde{H}_\sigma\gen{\tilde{f}} = \widetilde{J}$ which is the identity on
$\widehat{H}_\sigma = H_\sigma/Z(H_\sigma) = \widetilde{H}_\sigma$ and which
sends $\widehat{g}$ to $\widetilde{f}$. 

By \cite[Theorem~4.9.1(d)]{GLS3}, $\Inndiag(H_\sigma) \cong \widehat{H}_1$ acts
transitively on the involutions in
$\widehat{H}_\sigma\hat{g}-\widehat{H}_\sigma$, and so each involution in
$\widehat{H}_\sigma\hat{g}$ is $\widehat{H}_\sigma$-conjugate to $\hat{g}$ or
$\hat{g}^{\hat{t}}$. As $t^g = t^{\sigma_{l-1}} = t^{5^{2^{l-1}}}$ and
$5^{2^{l-1}}-1$ has $2$-adic valuation $l+1 = k-1$, we see that there is an
element $u \in \gen{t^{2^{k-1}}} \leq H_\sigma$ of order $4$, such that $[g,u]
= 1$, $u^2 = z$, and $t^g = tu$. Then $g^t = gu^{-1}$, so that
$\hat{g}^{\hat{t}} = \hat{g}\hat{u}^{-1}$. From the isomorphism
$\widehat{H}_\sigma\gen{\widehat{g}} \cong
\widetilde{H}_\sigma\gen{\tilde{f}}$, we conclude that each involution in
$\widetilde{H}_\sigma\gen{\tilde{f}}$ is $\widetilde{H}_\sigma$-conjugate to
either of $\tilde{f}$ or $\tilde{f}\tilde{u}^{-1}$. However, the two preimages
of $\tilde{f}\tilde{u}^{-1}$ in $H_\sigma\gen{f}$ are $fu^{-1}$ and $fu = fu^{-1}z$,
both of which are of order $4$ as $[f,u]=1$. Thus, all four subgroups of
$H_\sigma\gen{f}$ which contain $\gen{z}$ and are not contained in $H_\sigma$
are $H_\sigma$-conjugate.  Since $\gen{f,z}$ is such a four subgroup, it is
enough to show that $f$ is $H_\sigma$-conjugate to $fz$.  But $f^s = fz$ where
$s = t^2 \in H_\sigma$. This completes the proof that all involutions in
$H_\sigma f-H_\sigma$ are $H_\sigma$-conjugate, and this implies (a).

It remains to prove (b).  We keep the notation from above, writing $f$ for
$\sigma_{l-1}|_{H_\sigma}$ and for $\sigma_{l-1}|_{S_{\sigma}}$.

We first prove that $\gen{f}$ itself is fully $\D$-centralized.  By
Lemma~\ref{L:sylowchoice}(c), the $2$-group $S_\sigma$ has a decomposition
$S_\sigma = T_{2^k}W_S$ such that $f$ centralizes $W_S$ and acts on
$T_{2^k}$ via the map $t \mapsto t^{q_{l-1}}$. In particular
$C_{S_{\sigma}}(f) = T_{2^{k-1}}W_S$ is a Sylow $2$-subgroup of
$H_{\sigma_{l-1}}$. The centralizer $C_{H_\sigma}(f) =
H_{\sigma_{l-1}}$ is isomorphic to $\Spin_7(q_{l-1})$ by
\cite[Theorem~4.9.1(a)]{GLS3}, so $C_{C_D(z)}(f) = C_{D}(f) =
C_{S_\sigma}(f)\gen{f}$ is a Sylow $2$-subgroup of $C_{H_\sigma\gen{f}}(f)$.
Hence, $f$ is fully $C_{\D}(z)$-centralized by \eqref{E:CDzsemidirect}, and
\begin{eqnarray}
\label{E:CDzf}
C_{C_\D(z)}(f) \cong \H_{\sigma_{l-1}} \times \gen{f},
\end{eqnarray}
by \cite[I.5.4]{AschbacherKessarOliver2011}.  However, all involutions in
$D-S_{\sigma}$ are $C_\D(z)$-conjugate by (a), so two involutions in
$D-S_\sigma$ are $C_\D(z)$-conjugate if and only if they are $\D$-conjugate.
This completes the proof that $\gen{f}$ is fully $\D$-centralized.

Next, since $\gen{f}$ is fully $\D$-centralized, $C_\D(f)$ is saturated.
Lemma~\ref{L:NEP1} then shows $C_{\F_\sigma}(f)$ is normal in $C_\D(f)$, so
that $C_{C_{\F_{\sigma}}(f)}(z)$ is normal in $C_{C_\D(f)}(z)$ by
Lemma~\ref{L:NEP1} applied with $C_{\D}(z)$ in the role of $\F$.  Observe that
$C_{C_{\F_{\sigma}}(f)}(z)$ has Sylow group $S_{\sigma_{l-1}} =
C_{S_\sigma}(f)$.  Consequently, we have
\begin{eqnarray}
\label{E:CCFsigmafz}
C_{C_{\F_{\sigma}}(f)}(z) = \H_{\sigma_{l-1}}. 
\end{eqnarray}
from \eqref{E:CDzf}. 

Now Lemma~\ref{L:NEP2} shows that it suffices to determine $C_{\F_{\sigma}}(f)$
in order to finish the proof of (b).  For if $f' \in D-S_\sigma$ is another
involution fully centralized in $\D$, then any element $\phi$ of $\AA_{\D}(f)$
with $f^\phi = f'$ will induce an isomorphism $C_{\F_\sigma}(f) \to
C_{\F_\sigma}(f')$ by that lemma.

We argue next within the Aschbacher-Chermak amalgam to show that
$C_\D(f)$ contains $\F_{\sigma_{l-1}}$.  Now $\F_{\sigma_{l-1}}$
is generated by $\H_{\sigma_{l-1}}$ and
$\Aut_{\F_{\sigma_{l-1}}}(C_{S_{\sigma_{l-1}}}(U))$ from Lemma~\ref{L:gen}(a).
So to prove that $C_\D(f)$ contains $\F_{\sigma_{l-1}}$, we are reduced
via \eqref{E:CDzf} to a verification that $C_\D(f)$ contains
$\Aut_{\F_{\sigma}}(C_{S_{\sigma_{l-1}}}(U))$.  Now $y$ (in the notation
of Theorem~\ref{T:uniqueliftsigma}) acts on $C_{S_{\sigma_{l-1}}}(U)$.
Viewing $y$ in the semidirect product of
$K_{\sigma}\gen{\sigma_{l-1}}|_{K_{\sigma}}$ and applying
Theorem~\ref{T:uniqueliftsigma} with $\sigma_{l-1}$ in the role of
$\sigma_l$, we see that $y$ centralizes $\sigma_{l-1}$, so that
$c_y$ extends to the automorphism $c_y$ of
$C_{S_{\sigma_{l-1}}}(U)\gen{f}$ that centralizes $f$. So as
$\Aut_{\F_{\sigma_{l-1}}}(C_{S_{\sigma_{l-1}}}(U))$ is generated
by $\Aut_{\H_{\sigma_{l-1}}}(U)$ and $c_y$ by
Lemma~\ref{L:essentials}(b), it follows that
$\Aut_{\F_{\sigma_{l-1}}}(C_{S_{\sigma_{l-1}}}(U))$ is contained
in $C_\D(f)$. Lemma~\ref{L:gen}(a) with $\sigma_{l-1}$ in the role of
$\sigma$ shows that $C_\D(f)$ contains $\F_{\sigma_{l-1}}$. We have $\D
= \F_{\sigma}\gen{f}$, since $\F_{\sigma} = O^2(\D)$. So also
$C_{\F_{\sigma}}(f) \geq O^2(C_\D(f)) \geq O^2(\F_{\sigma_{l-1}}) =
\F_{\sigma_{l-1}}$.  This shows that $C_{\F_{\sigma}}(f)$ contains
$\F_{\sigma_{l-1}}$. 

Finally, we complete the proof by appealing to Holt's Theorem for fusion
systems \cite[Theorem~2.1.9]{AschbacherQFP}.  The systems $C_{\F_{\sigma}}(f)$
and $\F_{\sigma_{l-1}}$ both have Sylow group $C_{S_{\sigma}}(f)$.
Further, all involutions in $\F_{\sigma_{l-1}}$ are conjugate by (a), and
so $z^{C_{\F_{\sigma}}(f)} = z^{\F_{\sigma_{l-1}}}$ for the involution $z \in
Z(S_{\sigma_{l-1}})$.  Moreover, we showed in \eqref{E:CCFsigmafz} that
$C_{C_{\F_{\sigma}}(f)}(z) = \H_{\sigma_{l-1}} \leq
\F_{\sigma_{l-1}}$. This completes the verification of the hypotheses of
Holt's Theorem, and by that theorem we have $C_{\F_{\sigma}}(f) =
\F_{\sigma_{l-1}}$.
\end{proof}

We close this section by verifying that the Benson-Solomon systems are split.
This allows one, via Theorem~8 of \cite{AschbacherFSCT}, to severely restrict
the Sylow subgroup of the centralizer of a Benson-Solomon standard subsystem.

\begin{lemma}\label{L:Csplit}
$\F_\sigma$ is split.
\end{lemma}
\begin{proof}
Let $(\F,V)$ be a critical split extension of $\F_\sigma$, where $\F$ is a
saturated fusion system over any finite $2$-group $S$. Let $\L_\sigma$ be the
centric linking system for $\F_\sigma$. Note that $S/C_S(\F_\sigma)S_\sigma$
embeds in $\Out(\L_\sigma)$ by \cite[Theorem~A]{Semeraro2015}, while
$\Out(\L_\sigma)$ is cyclic of $2$-power order by
\cite[Theorem~3.10]{HenkeLynd2018}. Hence 
\begin{eqnarray}
\label{E:VcapCSCT}
V \cap C_S(\F_\sigma)S_\sigma > 1.
\end{eqnarray}

Write $V = \gen{u,v}$ with $u \in V \cap C_S(\F_\sigma)S_\sigma$.  Then either
$V \cap C_S(\F_\sigma) > 1$, or there exist elements $c \in C_S(\F_\sigma)$ and
$1 \neq t \in S_\sigma$ such that $u = ct$. 

Assume the former case, and set $X = V \cap C_S(\F_\sigma) > 1$.  As $X$
commutes with $V$ and $S_\sigma$ and $S = S_\sigma V$, we have $X \leq Z(S)$.
Thus, $C_\F(X)$ is a saturated subsystem over $S$ which contains $\F_{\sigma}$,
so $C_\F(X) = \F = N_\F(X)$.  Applying condition (T1) in
Definition~\ref{D:tightlyEmb} with $Q = V$, $\Q = \F_V(V)$, $X = V \cap
C_{S}(\F_{\sigma})$, and $\alpha = \id_X$, we see that $V$ is normal in $\F$,
and hence that $V \leq Z(\F)$ by Lemma~\ref{L:critsplitbasic}(a).
Therefore, $\F$ is the central product of $V = C_S(\F_\sigma)$ and
$\F_\sigma$. 

Consider the latter case. As $C_S(\F_\sigma) \cap S_\sigma \leq Z(\F_\sigma) =
1$ and $u$ is an involution, $t$ is an involution. Then, as $\F_\sigma$ has one
class of involutions by Lemma~\ref{L:involutionsconjugate}, $u$ is $\F$-conjugate
to $cz \in Z(S)$.  However, $\gen{u}$ is itself fully $\F$-centralized by
Lemma~\ref{L:critsplitbasic}(a), and so $u \in Z(S)$. As $\gen{v}$ is fully
$\F$-centralized and $C_\F(u) = C_\F(v)$ by Lemma~\ref{L:critsplitbasic}, we
have $v \in Z(S)$.  But then, using Lemma~\ref{L:lim1=0} to see that
Proposition~\ref{P:CST} applies, we have $V \leq Z(S) = C_S(S_\sigma) =
C_S(\F_\sigma)Z(S_\sigma)$ by that proposition applied with $\F_1 = \F$, so
that $C_S(\F_\sigma)S_\sigma = VS_\sigma = S$.  Thus, $\F$ is the central
product of $C_S(\F_\sigma)$ with $\F_\sigma$ in this case as well.
\end{proof}

\subsection{The known quasisimple groups and quasisimple $2$-fusion
systems}\label{SS:known}

In this section, we define the class $\Chev[\larg]$ of $2$-fusion systems
that appears in Walter's Theorem and which is needed in
Section~\ref{S:general}, and we prove two lemmas about components and
involution centralizers in known almost quasisimple groups that will be used in
Sections~\ref{S:quaternion} and \ref{S:general}.  By a \emph{known finite
simple group} we mean a finite group isomorphic to one of the groups appearing
in the statement of the classification of finite simple groups. By a
\emph{known quasisimple group} we mean a quasisimple covering of a known finite
simple group. Such coverings are listed in \cite[Section~6.1]{GLS3}. An
\emph{almost simple group} is a finite group whose generalized Fitting subgroup
is simple.  Similarly, an \emph{almost quasisimple group} is a finite group
whose generalized Fitting subgroup is quasisimple.  

\begin{definition}
Let $p$ be an odd prime. $\Chev(p)$ is the class of quasisimple groups of Lie
type in characteristic $p$, i.e, the quasisimple groups $K$ for which there is
a simple algebraic group $\ol L$ over $\ol\FF_p$ and a Steinberg endomorphism
$\sigma$ such that $K/Z(K) \cong C_{\ol L}(\sigma)'$ (cf.
\cite[Definition~2.2.8]{GLS3}), and 
\begin{align*}
\Chev^*(p) = \Chev(p) &- \{L_2(q) \mid q = p^a, q \geq 5\} \\
                      &- \{ { }^2G_2(q) \mid q = 3^{2a+1}, a \geq 1\}\\
                      &- \{ { }^2G_2(3)'\}.
\end{align*}
\end{definition}
Thus, the groups $SL_2(q)$ for $q = p^a \geq 5$ lie in $\Chev^*(p)$ (but
$SL_2(3)$ does not). The above working definition of $\Chev^*(p)$ appears to be
consistent with that in \cite{AschbacherWT} and \cite{AschbacherQFP}, but our
main interest is in $\Chev[\larg]$ below.

We use analogous terminology for the class of known simple, quasisimple, almost
simple, and almost quasisimple $2$-fusion systems, respectively. The class of
\emph{known simple $2$-fusion systems} consists of the fusion systems of simple
groups whose $2$-fusion systems are simple, together with the Benson-Solomon
fusion systems. As was shown in \cite[Theorem~5.6.18]{AschbacherQFP}, the
simple groups whose $2$-fusion systems fail to be simple are precisely the
\emph{Goldschmidt groups}, namely the finite simple groups which have a
nontrivial strongly closed abelian $2$-subgroup. By a result of Goldschmidt, a
Goldschmidt group is either a group of Lie type in characteristic $2$ of Lie
rank $1$, or has abelian Sylow $2$-subgroups. For example, $\Omega_7(q)$ is
not a Goldschmidt group, so its $2$-fusion system is simple for any odd prime
power $q$.

In the proof of Walter's Theorem, and in Section~\ref{S:general} below, the
$2$-fusion systems of a certain subclass of $\Chev^*(p)$ consisting of
``large'' groups play an important role.

\begin{definition}[{\cite[Definition~1.1]{AschbacherWT}}] 
\label{D:chevlarge}
Let $\Chev^*$ be the union of $\Chev^*(p)$ as $p$ ranges over the odd primes.
$\Chev[\larg]$ is the class of quasisimple $2$-fusion systems $\K$ such that
$\K = \F_2(K)$ for some finite group $K$ in the collection
\begin{align*}
\Chev^*    &- \{L_3^\epsilon(q) \mid q \equiv \pm 3 \!\!\!\!\pmod{8}, \,\,\epsilon = \pm 1\}\\
           &- \{G_2(q) \mid q \equiv \pm 3 \!\!\!\! \pmod{8}\}\\
           &- \{P\Omega_n^{\epsilon}(q) \mid q \equiv \pm 3 \!\!\!\!\pmod{8}, \,\,\epsilon = \pm 1,\,\, 5 \leq n \leq 8\}\\
           &- \{\Omega_6^+(q) \mid q \equiv \pm 3 \!\!\!\!\pmod{8}\}. 
\end{align*}
\end{definition}
For example, the $2$-fusion system of $\Omega_7(q)$, $q \equiv \pm 3 \pmod{8}$
is excluded from $\Chev[\larg]$, while the $2$-fusion system of $\Spin_7(q)$ is
a member of $\Chev[\larg]$ for each odd prime power $q$.

\begin{lemma}
\label{L:knowncent}
Let $G$ be a finite group with $F^*(G)$ a known quasisimple group. Then for
each involution $t \in G$ and each component $\bar{L}$ of $C_G(t)/O(C_G(t))$,
$\bar{L}$ is a known quasisimple group. 
\end{lemma}
\begin{proof}
This follows from the determination of the conjugacy classes of involutions and
their centralizers in the known almost quasisimple groups in \cite{GLS3}.  More
precisely, see Tables~4.5.1-4.5.3, Section~4.9, and Corollary~3.1.4 of
\cite{GLS3} for these data with regard to the members of $\Chev$, Table~5.3 of
\cite{GLS3} for the sporadic groups and their covers, and Section~5.2 of
\cite{GLS3} for the alternating groups and their covers.
\end{proof}

The next lemma says that each component of the $2$-fusion system of a group $G$
with $O(G) = 1$ is the fusion system of a component of the group provided the
components of $G$ are known finite quasisimple groups. It was suggested to us
by Aschbacher. 

\begin{lemma}\label{L:ComponentsFSGroups}
Let $G$ be a finite group with $O(G)=1$. Assume that for each
component $K$ of $G$, $K/Z(K)$ is a known finite simple group. Let
$S\in\Syl_2(G)$ and let $\C$ be a component of $\F_S(G)$. Then there exists a
component $K$ of $G$ such that $\C=\F_{S\cap K}(K)$.  
\end{lemma}

\begin{proof}
Set $\F=\F_S(G)$ and $\m{E}=\F_{S\cap F^*(G)}(F^*(G))$. Suppose $\C$ is a
subsystem of $\F$ on $T$. As $F^*(G)$ is normal in $G$, the subsystem $\m{E}$
is normal in $\F$ by \cite[Proposition~I.6.2]{AschbacherKessarOliver2011}. 

Assume first that $\C$ is not a component of $\m{E}$. Write $J$ for the set of
components of $\F$ which are not a component of $\m{E}$, and set
$\D=\Pi_{\C'\in J}\C'$. Then by \cite[9.13]{AschbacherGeneralized}, $\F$
contains a subsystem $\D\E$ which is the central product of $\D$ and $\E$. As
$\C\in J$, this implies in particular that $\E\leq C_\F(T)$. Since
$\m{E}=\F_{S\cap F^*(G)}(F^*(G))$, it follows now from
\cite[Theorem~B]{HenkeSemeraro2015} that $T\leq C_S(F^*(G))\leq
C_G(F^*(G))=Z(F^*(G))$. In particular, $T$ is abelian, which by
\cite[9.1]{AschbacherGeneralized} yields a contradiction to $\C$ being
quasisimple. Thus, we have shown that $\C$ is a component of $\m{E}$. 

As $O(G)=1$, $F^*(G)$ is the central product of $O_2(G)$ and the components of
$G$. Thus, $\m{E}$ is a central product of $\F_{O_2(G)}(O_2(G))$ and the
subsystems of the form $\F_{S\cap K}(K)$ where $K$ is a component of $G$. Since
$\C$ is not the fusion system of a $2$-group, it follows now from
Lemma~\ref{L:DirectProductComponent} that $\m{E}$ is a component of $\F_{S\cap
K}(K)$ for some component $K$ of $G$. Set $\ov{K}:=K/Z(K)$. As $O(G)=1$, we
have that $Z(K)\leq S$ and $Z(K)$ is contained in the centre of $\F_{S\cap
K}(K)$. Moreover, $\F_{S\cap K}(K)/Z(K)= \F_{\ov{K\cap S}}(\ov{K})$. Recall
that $\ov{K}$ is a ``known'' finite simple group. So by
\cite[Theorem~5.6.18]{AschbacherQFP}, $\F_{\ov{S\cap K}}(\ov{K})$ is either
simple or $\ov{S\cap K}$ is normal in $\ov{K}$. As the fusion system $\F_{S
\cap K}(K)$ contains $\m{C}$ as a component, it is not constrained
by \cite[9.9.1]{AschbacherGeneralized}. It follows thus from
\cite[Lemma~2.10]{Henke2019} that $\F_{\ov{K\cap S}}(\ov{K})$ is not
constrained, which excludes the case that $\ov{S\cap K}$ is normal in
$\F_{\ov{S\cap K}}(\ov{K})$. Hence, $\F_{S\cap K}(K)/Z(K)=\F_{\ov{S\cap
K}}(\ov{K})$ is simple. In particular, $Z(K)=Z(\F_{S\cap K}(K))$. As $K$ is
quasisimple, we have $K=O^2(K)$. Therefore, it follows from Puig's hyperfocal
subgroup theorem \cite[\S 1.1]{Puig2000} and
\cite[Corollary~I.7.5]{AschbacherKessarOliver2011} that $O^2(\F_{S\cap
K}(K))=\F_{S\cap K}(K)$. So $\F_{S\cap K}(K)$ is quasisimple and thus, by
\cite[9.4]{AschbacherGeneralized}, we have $\C=\F_{S\cap K}(K)$.
\end{proof}

\subsection{Summary of preliminary definitions}\label{SS:prelim-summary}
For the convenience of the reader, we summarize in a quick-reference
table the most important definitions from Section~\ref{S:prelim} for following
the later arguments.

\bigskip
\renewcommand{\arraystretch}{1.3}
\begin{longtable}{|P{1.5in}|P{1in}|P{2.3in}|P{1in}|}
\caption{Summary of Section~\ref{S:prelim}}
\label{Ta:defs}\\
\hline 
\textbf{Notation/Property} & \textbf{Assumption} & \textbf{Description} & \textbf{Reference}\\
\hline
\hline
\endfirsthead
\caption[]{Summary of Section~\ref{S:prelim} (continued from previous page)}\\
\hline 
\textbf{Notation/Property} & \textbf{Assumption} & \textbf{Description} & \textbf{Reference}\\
\cline{1-4}
\cline{1-4}
\endhead
$\AA(X)=\AA_\F(X)$ & $X\leq S$ & set of $\alpha\in\Hom_\F(N_S(X),S)$ with $X^\alpha\in\F^f$ & Notation~\ref{N:AA}\\
\hline
$N_\E(P)$ & $\E\unlhd\F$, $P\leq S$ with $P\in (\E P)^f$ & unique normal subsystem of $N_{\E P}(P)$ over $N_T(P)$ of $p$-power index & Definition~\ref{D:NEP}\\
\hline
component of $\F$ &  & quasisimple subnormal subsystem of $\F$ & \\
\cline{1-1}\cline{3-3}
$E(\F)$ & & central product of the components of $\F$ & Section~\ref{SS:Components}\\
\cline{1-1}\cline{3-3}
$F^*(\F)$ & & central product of $O_p(\F)$ and $E(\F)$ & \\
\hline
$\X(\C)=\X_\F(\C)$ 
&
& set of elements or subgroups $X$ of $C_S(T)$ with $\C\leq C_\F(X)$ 
& \\
\cline{1-1} \cline{3-3}
$\tilde{\X}(\C)=\tilde{\X}_\F(\C)$ 
& $\C$ is a quasisimple subsystem of $\F$ over $T\leq S$
& set of elements or subgroups $X$ of $S$ such that $\C^\alpha$ is a component of $N_\F(X^\alpha)$ for some $\alpha\in\AA(X)$ 
& Notation~\ref{N:XCIC}\\
\cline{1-1} \cline{3-3}
$\I(\C)=\I_\F(\C)$ &   & set of involutions in $\tilde{\X}(\C)$ & \\
\cline{1-4}
\pagebreak

\cline{1-4}
$\CC(\F)$ 
& 
& set of quasisimple subsystems $\C$ of $\F$ with $\I(\C)\neq\varnothing$ 
& Notation~\ref{N:XCIC}\\
\hline
$\D$ proper pump-up of $\C$ 
& $\C\in\CC(\F)$, $\D$ a subsystem of $\F$ 
& hypothesis and conclusion (2) of the ``Pump-Up Lemma'' \ref{L:pumpupbasic} hold
& Definition~\ref{D:pumpup}\\
\cline{1-1}\cline{2-3}

$\C$ maximal 
& $\C\in\CC(\F)$ 
& no proper pump-up of $\C$ 
& \\
\hline
$P_{t,\alpha}$ & $\C\in\CC(\F)$, & $P_{t,\alpha}:=C_{C_S(t^\alpha)}(\C^\alpha)\cap C_S(t)^\alpha$ & Notation~\ref{N:PtalphaQt}\\
\cline{1-1}\cline{3-3}
$Q_t:=Q_{t,\alpha}$ & $t\in\I(\C)$, $\alpha\in\AA(t)$ & $Q_{t,\alpha}:=P_{t,\alpha}^{\alpha^{-1}}$ & \\
\hline
$\Delta(\C)$ 
& 
& set of $\F$-conjugates $\C_1$ of $\C$ over some $T_1\leq S$ such that 
  $T_1^\#\subseteq\tilde{\X}(\C)$ and $T^\#\subseteq\tilde{\X}(\C_1)$
& Notation~\ref{N:DeltaC}\\
\cline{1-1} \cline{3-3}

$\rho(\C)$ 
& 
& set of all $(t^\phi,\C^\phi)$ where $t\in\I(\C)$ and $\phi\in\Hom_\F(\gen{t,T},S)$
&\\
\cline{1-1}\cline{3-4}

$\C$ terminal 
& 
& 
\raggedright
\begin{itemize}
\item[(C1)] $T\in\F^f$,
\item[(C2)] $\Delta(\C)\neq\varnothing$,
\item[(C3)] $(t_1,\C_1)\in\rho(\C)\Longrightarrow Q_{t_1}^\#\subseteq\tilde{\X}(\C_1)$
\end{itemize}
& Definition~\ref{D:terminal}\\
\cline{1-1} \cline{3-4}

$\C$ nearly standard 
& $\C\in\CC(\F)$ a subsystem over $T \leq S$ 
& 
\raggedright
\begin{itemize}
\item[(S1)] $\tilde \X(\C)$ contains a unique maximal member $Q$
\item[(S2)] $\C^\alpha \norm N_\F(X^\alpha)$ for each $1 \neq X \leq Q$ and $\alpha \in \AA(X)$
\item[(S3)] $T^\beta = T$ for each $1 \neq X \leq Q$ and $\beta \in \AA(X)$ with $X^{\beta} \leq Q$
\end{itemize}
& Definition~\ref{D:standard}\\
\cline{1-1} \cline{3-3}

$\C$ standard 
& 
& $\C$ nearly standard and
\begin{itemize}
\item[(S4)] $\Aut_\F(T) \leq \Aut(\C)$
\end{itemize}
(ensures existence of a ``centralizer'' of $\C$, in many cases fulfilled if $\C$ terminal)
& \\
\cline{1-1} \cline{3-4}

$\C$ subintrinsic 
& 
& there is $\H\in\CC(\C)$ with $\I_\F(\H)\cap Z(\H)\neq\varnothing$ 
& Definition~\ref{D:subintrinsic} \\
\hline

$\Q$ tightly embedded
& $\Q$ subsystem over $Q \in \F^f$
& 
\raggedright
\begin{itemize}
\item[(T1)] $O^{p'}(N_\Q(X))^\alpha$ is normal in  $N_\F(X^\alpha)$ for each $1
\neq X \in \Q^f$ and $\alpha \in \AA(X)$
\item[(T2)] $X^\F \cap Q = X^{\Aut_\F(Q)\Q}$ for each $X$ of order $p$. 
\item[(T3)] $\Aut_\F(Q) = \Aut(\Q)$
\end{itemize}

\bigskip
\centering ((T1)--(T3) hold if $\Q$ is the ``centralizer'' of standard subsystem)
& Definition~\ref{D:tightlyEmb}\\
\hline
\end{longtable}

\section{Subintrinsic maximal Benson-Solomon components}\label{S:showStandard}

We assume the following hypothesis throughout this section.

\begin{hypothesis}\label{H:maximal}
Let $\F$ be a saturated fusion system over the $2$-group $S$, and let $q$ be an
odd prime power. Suppose $\C \cong \F_{\Sol}(q)$ is a subsystem of $\F$ over
the subgroup $T \in \F^f$. Let $z \in Z(T)$ be the involution, set $\H =
C_\C(z)$.  Assume $\C$ is maximal in $\CC(\F)$ and $z \in \I_\F(\H)$, but that
$\C$ is not a component of $\F$.
\end{hypothesis}

As we explain in detail in Remark~\ref{R:Subintrinsic} below, this hypothesis
amounts to assuming that the pair $(\F,\C)$ is a counterexample to
Theorem~\ref{T:main}. The purpose of this section is to prove the following
theorem.

\begin{theorem}
\label{T:Cstandard}
Assume Hypothesis~\ref{H:maximal}. Then $\C$ is standard.
\end{theorem} 
\begin{proof}
This is the content of Lemmas~\ref{L:nearlystandard} and \ref{L:AutFTAutC}
below. 
\end{proof}

By Lemma~\ref{L:SolEquivalences}, we may take $q = 5^{2^l}$ for a unique $l
\geq 0$ in Hypothesis~\ref{H:maximal}. Since we will be working in this section
with the internal structure of $\C \cong \F_{\Sol}(q)$ in some detail, we shall
make the following identifications in order to make it simpler to apply the
results of Section~\ref{SS:spinsol}. Here, $\sigma = \sigma_l$ is the
automorphism of the Aschbacher-Chermak free amalgamated product appearing in
Theorem~\ref{T:uniqueliftsigma}.

\medskip

\begin{center}
\begin{tabular}{c|c|c|p{2.2in}}
Hypothesis~\ref{H:maximal} & Identify with & Reference in \S\S\ref{SS:spinsol} & Note\\
\hline
$T$ & $S_\sigma $ & Lemma~\ref{L:sylowchoice}(c) & $T$ not to be confused with a maximal torus \\
$\C$ & $\F_\sigma$ & Notation~\ref{N:Fsigma} & \\
$\H$ & $\H_\sigma$ & Notation~\ref{N:Fsigma} & \\
\end{tabular}
\end{center}

\medskip
We reiterate that the symbol $S$ is now used for the Sylow group in the
ambient system $\F$ of Hypothesis~\ref{H:maximal}. (Previously, it stood for
the Sylow $2$-subgroup of the Aschbacher-Chermak free amalgamated product.)

\medskip

\begin{remark}\label{R:Subintrinsic}
Apart from the assumption that $\C$ is not a component of $\F$,
Hypothesis~\ref{H:maximal} is a restatement of the hypotheses of
Theorem~\ref{T:main}. The assumption $z \in \I(\H) = \I_\F(\H)$ in
Hypothesis~\ref{H:maximal} is equivalent to the condition that $\C$ is
subintrinsic in $\CC(\F)$, because $z$ is the unique involution in $Z(\H)$
and $\CC(\C)=\{\H\}$. To see the latter property recall that every
involution in $T$ is $\C$-conjugate to $z$ by
Lemma~\ref{L:involutionsconjugate} and that $z$ is fully normalized, as $z\in
Z(T)$. Therefore, each member of $\CC(\C)$ is $\C$-conjugate to a component of
$\H$ by Lemma~\ref{AAnonempty} and Lemma~\ref{L:CCconjugate}. As $\H$ is
quasisimple, $\H$ is by \cite[9.4]{AschbacherGeneralized} the only component of
$\H$. Hence, every member of $\CC(\C)$ is $\C$-conjugate to $\H$. For each
$\alpha \in \Hom_\F(T,T)$, $z^\alpha = z$  as $Z(T) = \gen{z}$ is
characteristic in $T$, and so $\H^\alpha = \H$. Hence, $\H$ is the only
$\C$-conjugate of $\H$ and $\CC(\C)=\{\H\}$.
\end{remark}

\begin{lemma}\label{L:Shift}
Let $\alpha\in\Hom_\F(T,S)$. Then $\C^\alpha$ is maximal in $\CC(\F)$,
$\H^\alpha=C_{\C^\alpha}(z^\alpha)$, $z^\alpha\in\I(\H^\alpha)$, and
$\m{C}^\alpha$ is not a component of $\F$.
\end{lemma}

\begin{proof} 
By Lemma~\ref{L:ConjugateComponents}(b), $\C^\alpha$ is not a component of $\F$
as $\C$ is not a component of $\F$. As $T\in\F^f$, it follows from
\cite[6.2.13]{AschbacherFSCT} that $\C^\alpha$ is maximal in $\CC(\F)$. As
$\alpha$ induces an isomorphism from $\C$ to $\C^\alpha$, we have
$\H^\alpha=C_{\C^\alpha}(z^\alpha)$. Since $z\in\I(\H)$,
Lemma~\ref{L:CCconjugate}(b) gives $z^\alpha\in\I(\H^\alpha)$.  
\end{proof}

\begin{lemma}
\label{L:terminal}
$\C$ is terminal in $\CC(\F)$.
\end{lemma}
\begin{proof}
By Hypothesis~\ref{H:maximal}, $\C$ is maximal and subintrinsic in
$\CC(\F)$.  Condition (C0) in Definition~\ref{D:terminal} is $T \in
\F^f$, which holds by Hypothesis~\ref{H:maximal}. Assume $\Delta(\C) \neq
\varnothing$.  Then in the notation of \cite[6.1.17]{AschbacherFSCT}, the set
$\C^\perp = \Delta(\C) \cup \{\C\} \neq \C$. Since $Z(\C) = 1$,
Hypothesis~7.2.1 of \cite{AschbacherFSCT} holds. By
\cite[Theorem~7.2.5]{AschbacherFSCT}, $\C$ is a component of $\F$, contrary to
hypothesis. Thus, $\Delta(\C) = \varnothing$, i.e (C1) is verified. Since
$m(T) = 4$ by Lemma~\ref{L:standardsequence}(d), Theorem~7.4.14 of
\cite{AschbacherFSCT} shows that $\Delta(\C)=\varnothing$ (where $\Delta(\C)$
is defined as in Notation~\ref{N:DeltaC}). 

It remains to verify (C2).  Let $t \in \I(\C)$ and $\phi
\in \Hom_\F(\<t,T\>,S)$ so that $(t^\phi,\C^\phi)\in\rho(\C)$. Fixing $1\neq
a\in Q_{t^\phi}$, we need to show that $a\in \tilde{\X}(\C^\phi)$. Note that
$a\in \X(\C^\phi)$.  So if $\tilde{a}$ is the unique involution in $\<a\>$ and
$\tilde{a}\in \tilde{\X}(\C^\phi)$, then $a\in\tilde{\X}(\C^\phi)$ by
\cite[6.1.5]{AschbacherFSCT}. So we may assume without loss of generality that
$a$ is an involution. Fix $\alpha \in \AA(a)$.  It remains to show that
$\C^{\phi\alpha}$ is a component of $C_\F(a^\alpha)$ and thus $a\in
\tilde{\X}(\C^\phi)$.

Note first that, by definition of $Q_{t^\phi}$, $\C^\phi\leq C_\F(a)$ and thus
$\C^{\phi\alpha}\leq C_\F(a^\alpha)$. By Lemma~\ref{L:CCconjugate}(b) applied
with $(\<t\>,\phi\alpha)$ in place of $(X,\phi)$, we have
$t^{\phi\alpha}\in\tilde{\X}(\C^{\phi\alpha})$.  Moreover, $[t^\phi,a]=1$ by
definition of $Q_{t^\phi}$ and thus $[t^{\phi\alpha},a^\alpha]=1$. Hence,
Lemma~\ref{L:CClocalsubsystem} yields $\C^{\phi\alpha}\in\CC(C_\F(a^\alpha))$. 

We will argue next that $\C^{\phi\alpha}$ is subintrinsic in
$\CC(C_\F(a^\alpha))$. By Lemma~\ref{L:Shift}, we have
$C_{\C^{\phi\alpha}}(z^{\phi\alpha})=\H^{\phi\alpha}$ and
$z^{\phi\alpha}\in\I(\H^{\phi\alpha})$. Recall that $\H^{\phi\alpha}\leq
\C^{\phi\alpha}\leq C_\F(a^\alpha)$. In particular,
$[T^{\phi\alpha},a^\alpha]=1$ and thus $[z^{\phi\alpha},a^\alpha]=1$. Hence, by
Lemma~\ref{L:CClocalsubsystem} applied with
$(\<z^{\phi\alpha}\>,\<a^\alpha\>,\H^{\phi\alpha})$ in place of $(X,Y,\C)$, we
have $z^{\phi\alpha}\in\I_{C_\F(a^\alpha)}(\H^{\phi\alpha})$. As
$z^{\phi\alpha}\in Z(\H^{\phi\alpha})$, this implies that $\C^{\phi\alpha}$ is
indeed subintrinsic in $\CC(C_\F(a^\alpha))$ as we wanted to prove.

As we have verified that $\C^{\phi\alpha}$ is a subintrinsic member of
$\CC(C_\F(a^\alpha))$, it follows now from \cite[1.9.2]{AschbacherWT} applied
with $C_\F(a^\alpha)$ in the role of $\F$ and with $\C^{\phi\alpha}$ in the
role of $\M$ that $\C^{\phi\alpha}$ is contained in some component of
$C_\F(a^\alpha)$. Since $\C^\phi$ is maximal in $\CC(\F)$ by
Lemma~\ref{L:Shift}, it follows from Lemma~\ref{L:pumpupbasic} applied with
$(t^\phi,\C^\phi)$ in place of the pair $(t,\C)$ of that lemma
that $\C^{\phi\alpha}$ is a component of $C_\F(a^\alpha)$. As argued above this
shows (a).
\end{proof}

For the remainder of this section, we adopt the following notation for certain
subgroups of $T$. After this, we will not need to refer explicitly to any
additional notation from in Section~\ref{SS:spinsol}.

\medskip
\begin{center}
\begin{tabular}{c|c|c|p{2.2in}}
Notation below & Subgroup & Reference in \S\S\ref{SS:spinsol} & Description\\
\hline
$T_{2^k}$ & $T_{2^k}$ & Notation~\ref{N:spin} & unique homocyclic subgroup of $T$ of rank $3$ and exponent $2^k$\\
$E$ & $E$ & Notation~\ref{N:spin} & $\Omega_1(T_{2^k})$\\
$R$ & $R_\sigma$ & Notation~\ref{N:spin} & \raggedright$R = C_T(E) = T_{2^k}\gen{w_0}$ for an involution $w_0$ inverting $T_{2^k}$
\end{tabular}
\end{center}

\medskip
\begin{lemma}\label{L:AutR}
The following hold:
\begin{itemize}
\item [(a)] The subgroup $E$ is characteristic in $R$ and thus $\Aut_\F(R)$-invariant. We have $\Aut_\F(R)=C_{\Aut_\F(R)}(E)\Aut_{\C}(R)$ and
$O_2(\Aut_\F(R))=C_{\Aut_\F(R)}(E)$.
\item [(b)] We have $N_S(R)=N_S(T)=C_S(E)T$.
\item [(c)] $R$ is fully $\F$-normalized.
\item [(d)] We have $\Out_\F(R)=O_2(\Out_\F(R))\times \Out_\C(R)$ and
$O_2(\Aut_\F(R))=\Aut_{C_S(E)}(R)$. In particular,
$O^2(\Aut_\F(R))=O^2(\Aut_\C(R))$.
\end{itemize}
\end{lemma}

\begin{proof}
\textbf{(a)}: It follows from Lemma~\ref{L:standardsequence}(a) that $T_{2^k}$ and $E=\Omega_1(T_{2^k})$ are characteristic in $R$. Set $C:=C_{\Aut_\F(R)}(E)$. Observe that $\Aut_\F(R)/C$ embeds into
$\Aut(E)\cong GL_3(2)$. As $\Aut_\C(R)/\Inn(R)\cong GL_3(2)$ and
$C_{\Aut_\C(R)}(E)=\Inn(R)$, it follows that $\Aut_\F(R)=C\Aut_{\C}(R)$.  By
Lemma~\ref{L:standardsequence}(a), $T_{2^k}$ is homocyclic of rank $3$ and exponent
$2^k$. So as $T_{2^k}$ is characteristic in $R$, for every $1\leq i<k$, the
map $\Omega_{i+1}(T_{2^k})/\Omega_i(T_{2^k})\rightarrow
\Omega_i(T_{2^k})/\Omega_{i-1}(T_{2^k}),x\Omega_i(T_{2^k})\mapsto x^2\Omega_{i-1}(T_{2^k})$ is
an isomorphism of $\Aut_\F(R)$-modules. So in particular, $C$ acts trivially on
$\Omega_{i+1}(T_{2^k})/\Omega_i(T_{2^k})$ for all $1\leq i<k$. As $|R/T_{2^k}|=2$, $C$ acts
also trivially on $R/T_{2^k}$. Hence, $C$ is a $2$-group and thus contained in
$O_2(\Aut_\F(R))$. As $E$ is an irreducible $\Aut_\F(R)$-module, it follows
$C=O_2(\Aut_\F(R))$. This shows (a). 

\textbf{(b)}: As $R$ is characteristic in $T$ by
Lemma~\ref{L:standardsequence}(c), we have $N_S(T)\leq N_S(R)$. By (a),
$C\Aut_T(R)$ is the unique Sylow $2$-subgroup of $\Aut_\F(R)$ containing
$\Aut_T(R)$. As $\Aut_S(R)$ is a $2$-group containing $\Aut_T(R)$, it follows
$\Aut_S(R)\leq C\Aut_T(R)$ and thus $N_S(R)\leq C_S(E)T\leq C_S(z)$. Let now
$x\in C_S(E)\leq C_S(z)$. As $z\in\I(\H)$, there exists $\alpha\in\AA(z)$ such
that $\H^\alpha$ is a component of $C_\F(z^\alpha)$. Then $x^\alpha\in
C_S(E^\alpha)\leq C_S(z^\alpha)$ and $(\H^\alpha)^{x^\alpha}$ is a component of
$C_\F(z^\alpha)$ by Lemma~\ref{L:ConjugateComponents}. So by
\cite[9.8.2]{AschbacherGeneralized}, either
$\H^\alpha=(\H^\alpha)^{x^\alpha}$, or $\H^\alpha$ and $(\H^\alpha)^{x^\alpha}$
form a commuting product. In the latter case,
$E^\alpha=(E^\alpha)^{x^\alpha}\leq Z(\H^\alpha)$, a contradiction to
$Z(\H^\alpha)=\<z^\alpha\>$. Hence, $\H^\alpha=(\H^\alpha)^{x^\alpha}$ and thus
$(T^x)^\alpha=(T^\alpha)^{x^\alpha}=T^\alpha$. This implies $x\in N_S(T)$. So
we have shown that $C_S(E)\leq N_S(T)$ and thus $N_S(R)\leq C_S(E)T\leq
N_S(T)\leq N_S(R)$. This yields (b). 

\textbf{(c)}: Let $\gamma\in\AA(R)$. Recall from (b) that $T\leq
N_S(T)=N_S(R)$. So in particular, as $T\in\F^f$,  we have $T^\gamma\in\F^f$ and
$N_S(T)^\gamma=N_S(T^\gamma)$. Thus, along with $T^\gamma \in \F^f$,
Lemma~\ref{L:Shift} says that Hypothesis~\ref{H:maximal} holds for $\C^\gamma$,
$z^\gamma$, and $\H^\gamma$ in place of $\C$, $z$, and $\H$.  So we can apply
(b) with $R^\gamma$ and $T^\gamma$ in place of $R$ and $T$ to obtain
$N_S(R^\gamma)=N_S(T^\gamma)$. This gives
$N_S(T^\gamma)=N_S(T)^\gamma=N_S(R)^\gamma\leq N_S(R^\gamma)=N_S(T^\gamma)$ and
therefore  $N_S(R)^\gamma=N_S(R^\gamma)$. Since $R^\gamma$ is fully normalized,
it follows that $R$ is fully normalized.  This shows (c). 

\textbf{(d)}: By (c) and the Sylow axiom,  $\Aut_S(R)\in\Syl_2(\Aut_\F(R))$ and so
$C=O_2(\Aut_\F(R))\leq \Aut_S(R)$. Thus, $C=\Aut_{C_S(E)}(R)$. By (b),
$[C_S(E),T]\leq C_S(E)\cap T = C_T(E)=R$. Hence, $[C,\Aut_T(R)]\leq \Inn(R)$.
For each $\alpha \in \Aut_\C(R)$, 
\[
[C,\Aut_T(R)^\alpha] = [C^{\alpha^{-1}}, \Aut_T(R)]^\alpha = [C,\Aut_T(R)]^\alpha,
\]
because $C$ is normalized by $\Aut_\C(R)$. So as $\Aut_\C(R) =
\gen{\Aut_T(R)^{\Aut_\C(R)}}$ by Lemma~\ref{L:standardsequence}(c), the
previous equation gives
\[
[C, \Aut_\C(R)] = [C,\gen{\Aut_T(R)^{\Aut_\C(R)}}] =
\gen{[C,\Aut_T(R)]^{\Aut_\C(R)}} \leq \Inn(R). 
\]
This together with (a) implies that (d) holds.  
\end{proof}

Notice that $N_S(T)\leq C_S(z)$ as $z$ is the unique central involution of $T$.
Hence, if $\alpha\in\AA(z)$, then $T^\alpha\in\F^f$ as $T\in\F^f$. So by
Lemma~\ref{L:Shift}, replacing $(\C,T,z)$ by $(\C^\alpha,T^\alpha,z^\alpha)$,
we may assume that $z$ is fully centralized. Moreover, we set 
\[
V_R:=RC_S(R)\mbox{\quad and \quad}Q_0:=C_S(T).
\] 

\begin{lemma}\label{L:towardstandard}
The following hold.
\begin{itemize}
\item [(a)] We have $\H\leq C_\F(Q_0)$.
\item [(b)] We have $C_S(R) = EC_S(T)$, and hence $V_R=RC_S(T)$. 
\item [(c)] $N_{N_\F(R)}(V_R)$ is a constrained fusion system and $N_\C(R)\leq N_{N_\F(R)}(V_R)$.
\item [(d)] Let $G_R$ be a model for $N_{N_\F(R)}(V_R)$ and $N:=C_{G_R}(V_R/R)$. Then $N_1:=\<T^N\>=O^2(N)R$ is a model for $N_\C(R)$.
\item [(e)] We have $Q_0=\<z\>\times Q$ where $Q=C_{Q_0}(N_1)$ with $N_1$ as in (d). 
\item [(f)] If $Q$ is as in (e), then $Q$ is the unique largest subgroup of $S$
centralized by $\m{C}$. More precisely, $\C\leq C_\F(Q)$, and $X\leq
Q$ for all $X\leq S$ with $\C\leq C_\F(X)$.
\item [(g)] If $Q$ is as in (e), then $Q$ is the unique largest member of $\tilde{\X}(\C)$.
\end{itemize}
\end{lemma}

\begin{proof}
\textbf{(a), (b)}: Recall that $E=Z(R)$. As $R \leq T$, clearly
$EC_S(T) \leq C_S(R)$, so for (b) we must show the other inclusion. Since $z
\in R$, we have $C_S(R)=C_{C_S(z)}(R) \leq C_S(z)$. Now by our choice of
notation, $\gen{z}$ is fully $\F$-centralized, so $C_\F(z)$ is a saturated
fusion system on $C_S(z)$. By Hypothesis~\ref{H:maximal}, $\H$ is a component
of $C_{\F}(z)$. The normalizer of a component is constructed in
\cite[\S2.1]{AschbacherFSCT}, and thus, we may form $N_{C_\F(z)}(\H)$ over the
$2$-group $N_S(T)=N_{C_S(z)}(T)$. By Lemma~\ref{L:AutR}(b), $C_S(R)\leq
N_S(R)=N_S(T)$, so we may form the product system $\widehat{\H} := \H C_S(R)$
as in \cite[Chapter~8]{AschbacherGeneralized} or \cite{Henke2013} in the
normalizer $N_{C_\F(z)}(\H)$. Thus $\widehat{\H}$ is a saturated subsystem of
$C_\F(z)$ with $O^2(\widehat{\H}) = O^2(\H) = \H =E(\widehat{\H})$.
Since $\H$ is tamely realized by $H = \Spin_7(5^{2^l})$ by
Lemma~\ref{L:tameSpin}(a), Theorem~\ref{T:reduct} gives an extension
$\widehat{H} = HC_S(R)$ of $H$ that tamely realizes $\widehat{\H}$. By
Lemma~\ref{L:tameSpin}(b), each automorphism of $H$ normalizing $T$ and
centralizing $R$ is conjugation by an element of $E$.  Hence, $Q_0\leq
C_S(R)\leq EC_S(H)\leq EC_S(T)$. This implies $C_S(R)=EC_S(T)$ and
$Q_0=C_E(T)C_S(H)=\<z\>C_S(H)=C_S(H)$. The first property gives (b), and the
latter property yields (a).

\textbf{(c)}: Since $R$ is fully normalized by Lemma~\ref{L:AutR}(c),
$N_\F(R)$ is saturated.  Note that $V_R$ is weakly closed and thus fully
normalized in $N_\F(R)$. So $N_{N_\F(R)}(V_R)$ is saturated. Clearly
$N_{N_\F(R)}(V_R)$ is constrained, as $V_R$ is a centric normal subgroup of
this fusion system. 

We show next that $N_\C(R)\leq N_{N_\F(R)}(V_R)$. Let $R\leq P\leq T$ and
$\phi\in\Aut_{N_\C(R)}(P)$. By Alperin's fusion theorem
\cite[Theorem~I.3.6]{AschbacherKessarOliver2011}, it is enough to show that
$\phi$ extends to an element of $\Aut_\F(PV_R)$ normalizing $V_R$. Let
$\alpha\in\AA(P)$ and observe that $\phi^\alpha\in\Aut_\F(P^\alpha)$. By (b),
$V_R=RC_S(T)\leq PC_S(P)$. Thus $V_R^\alpha\leq P^\alpha C_S(P^\alpha)\leq
N_{\phi^\alpha}$. As $P^\alpha$ is fully normalized, it follows from the
extension axiom that $\phi^\alpha$ extends to a morphism $\psi\colon P^\alpha
V_R^\alpha\rightarrow S$ in $\F$. Note that
$R^{\alpha\psi}=(R^\alpha)^{\phi^\alpha}=R^\alpha$ as $R^\phi=R$. Since $R$ is
fully normalized and thus fully centralized, we have
$C_S(R)^{\alpha\psi}=C_S(R^{\alpha\psi})=C_S(R^\alpha)=C_S(R)^\alpha$ and thus
$V_R^{\alpha\psi}=R^\alpha C_S(R)^\alpha=V_R^\alpha$. So
$\psi\in\Aut_\F(P^\alpha V_R^\alpha)$ extends $\phi^\alpha$ and normalizes
$V_R^\alpha$. Hence, $\hat{\phi}:=\psi^{\alpha^{-1}}\in\Aut_\F(PV_R)$ extends
$\phi$ and normalizes $V_R$. This proves (c).

\textbf{(d)}: Let now $G_R$ and $N$ be as in (d), and set $N_1:=\<T^N\>$.
(The model $G_R$ for $N_{N_\F(R)}(R)$ exists and is unique up to isomorphism by
\cite[Proposition~III.5.8]{AschbacherKessarOliver2011}. Moreover,
$C_{G_R}(V_R)\leq V_R$.) Note that $S_0:=N_S(R)\in\Syl_2(G_R)$.  By (b),
$[V_R,T]\leq R$. As $V_R$ and $R$ are normal in $G_R$, it follows that 
$[V_R,\<T^{G_R}\>]\leq R$ and thus $N_1\leq \<T^{G_R}\>\leq N$. Let $P\leq T$
be essential in $N_\C(R)$. As $\Out_\C(R)\cong GL_3(2)$, we observe that $R\leq
P$, $P/R\cong C_2\times C_2$ and $\Out_{N_\C(R)}(P)\cong GL_2(2)\cong S_3$. In
particular, $\Aut_{N_\C(R)}(P)=\gen{\Aut_T(P)^{\Aut_{N_\C(R)}(P)}}$. Since
$\Aut_{N_\C(R)}(P)\leq \Aut_{G_R}(P)$ by (c), it follows that
$\Aut_{N_\C(R)}(P)\leq \Aut_{\<T^{G_R}\>}(P)\leq \Aut_N(P)$. Now we conclude
similarly that $\Aut_{N_\C(R)}(P)\leq \<\Aut_T(P)^{\Aut_N(P)}\>\leq
\Aut_{N_1}(P)$. As $P$ was arbitrary, the Alperin--Goldschmidt Fusion Theorem
yields that $N_\C(R)\leq \F_{S_0\cap N_1}(N_1)$.  

Note that $N/C_N(R)$ embeds into $\Aut_\F(R)$. As $C_{G_R}(V_R)\leq V_R$, and
$C_N(R)$ centralizes $V_R/R$ and $R$, $C_N(R)$ is a normal $2$-subgroup of $N$.
So it follows from Lemma~\ref{L:AutR}(d) that $N/O_2(N)\cong \Out_\C(R)\cong
GL_3(2)$ and $O_2(N)=C_N(E)\leq C_{S_0}(E)$. Using Lemma~\ref{L:AutR}(b), we
conclude that $O_2(N)\leq C_S(E)\leq N_S(T)$ and thus $[O_2(N),T]\leq
C_T(E)=R$. Since $O_2(N)$ and $R$ are normal in $N$, this implies
$[O_2(N),N_1]\leq R$. In particular, noting $O_2(N_1)=O_2(N)\cap N_1$ and
setting $\ov{N}:=N/R$, it follows that $\ov{O_2(N_1)}$ is abelian. Observe that
$T/(T\cap O_2(N))=T/R$ is isomorphic to a Sylow $2$-subgroup of $GL_3(2)$.
Thus, $TO_2(N)/O_2(N)$ is a Sylow $2$-subgroup of $N/O_2(N)$ and so $TO_2(N)$
is a Sylow $2$-subgroup of $N$. In particular, $TO_2(N_1)=(TO_2(N))\cap N_1$ is
a Sylow $2$-subgroup of $N_1$. Note that $T\cap O_2(N_1)=T\cap
O_2(N)=C_T(E)=R$. Thus $\ov{T}$ is a complement to $\ov{O_2(N_1)}$ in the Sylow
$2$-subgroup $\ov{TO_2(N_1)}$ of $\ov{N_1}$. So by a Theorem of Gasch{\"u}tz
\cite[Theorem~3.3.2]{KurzweilStellmacher2004}, there exists a complement
$\ov{N_0}$ of $\ov{O_2(N_1)}$ in $\ov{N_1}$. We choose a preimage $N_0$ of such
a complement $\ov{N_0}$ with $R\leq N_0\leq N_1$. As $N/O_2(N)\cong GL_3(2)$ is
simple, we have $N=O_2(N)N_1=O_2(N)N_0$. Since $O_2(N)\cap N_0=O_2(N_1)\cap
N_0=R$ and $\ov{O_2(N)}$ is centralized by $\ov{N_1}$, it follows
$\ov{N}=\ov{O_2(N)}\times \ov{N_0}$. In particular, $N_0=O^2(N)R$ is normal in
$G_R$. As $N_\C(R)\leq \F_{S_0\cap N_1}(N_1)\leq \F_{S_0\cap N}(N)$,
we have $\hyp(N_\C(R))\leq \hyp(\F_{S_0\cap N}(N))\leq O^2(N)$. Hence
$T=\hyp(N_\C(R))R\leq O^2(N)R=N_0$. In particular, $N_0=N_1$, $O_2(N_1)=R$,
$T\in\Syl_2(N_1)$ and $N_1/R\cong GL_3(2)$. We show next that
$N_\C(R)=\F_T(N_1)$. We have seen already that $N_\C(R)\leq \F_T(N_1)$. If
$P$ is essential in $\F_T(N_1)$, then it follows from $N_1/R\cong GL_3(2)$
that $R\leq P\leq T$, $P/R\cong C_2\times C_2$ and $\Out_{N_1}(P)\cong
GL_2(2)$. As $GL_2(2)\cong \Out_{N_\C(R)}(P)\leq \Out_{N_1}(P)$, it follows that 
$\Aut_{N_1}(P)=\Aut_\C(P)$. Hence, we have $N_\C(R)=\F_T(N_1)$. Since
$C_{N_1}(O_2(N_1))\leq N_1\cap C_N(E)=N_1\cap O_2(N)=O_2(N_1)$, we conclude
that $N_1$ is a model for $N_\C(R)$. This completes the proof of (d). 

\textbf{(e)}: We consider now the action of $N_1/R\cong GL_3(2)$ on $U_R:=C_S(R)=C_{V_R}(R)$.
Note that $E = Z(R)$ is central in $U_R$ and recall that $U_R=EC_S(T)$ by (b).
In particular, $U_R/\Phi(C_S(T))$ is elementary abelian and thus $\Phi(U_R)\leq
\Phi(C_S(T))$. If $E\cap \Phi(U_R)$ were non-trivial, then we would have $E\leq
\Phi(U_R)$ as $N_1$ acts irreducibly on $E$. So it would follow that $E\leq
C_S(T)$ contradicting $E\not\leq Z(T)$. This shows that $E\cap \Phi(U_R)=1$.
Set
\[
\widetilde{U_R}=U_R/\Phi(U_R).
\] 
As $\widetilde{U_R}=\widetilde{E}\widetilde{C_S(T)}$ is elementary abelian,
there is a complement to $\widetilde{E}$ in $\widetilde{U_R}$ which lies in
$\widetilde{C_S(T)}$. So by a Theorem of Gasch{\"u}tz
\cite[Theorem~3.3.2]{KurzweilStellmacher2004}, applied in the semidirect
product $N_1\ltimes \widetilde{U_R}$, there exists a complement $\widetilde{Q}$
to $\widetilde{E}$ in $\widetilde{U_R}$ which is normalized by $N_1$. We choose
the preimage $Q$ of $\widetilde{Q}$ such that $\Phi(U_R)\leq Q\leq U_R$. 

As $[U_R,N_1]\leq [V_R,N]\leq R$, we have $[Q,N_1]\leq [U_R,N_1]\leq U_R\cap
R=Z(R)=E$. In particular,  $[\widetilde{Q},N_1]\leq \widetilde{Q}\cap
\widetilde{E}=1$. So $[Q,N_1]\leq \Phi(U_R)\cap E=1$.  Recalling $Q_0=C_S(T)$,
we conclude $Q\leq C_{Q_0}(N_1)$. Observe that $Q$ has index $2$ in
$Q_0=C_S(T)$ as $\widetilde{E}\cap \widetilde{C_S(T)}=\gen{\widetilde{z}}$ has
order $2$. Hence, since $[z,N_1]\neq 1$, it follows $Q=C_{Q_0}(N_1)$ and
$Q_0=\<z\>\times Q$. This proves (e).  

\textbf{(f)}: By (a), $Q_0$ centralizes $\H$, and by Lemma~\ref{L:gen}(b), we have
$\m{C}=\gen{\H,N_\C(R)}$. So if $X\leq Q_0=C_S(T)$, then $X$ contains $\C$ in
its centralizer if and only if it contains $N_\C(R)$ in its centralizer. As
$N_\C(R)=\F_T(N_1)$ by (d) and $Q$ is centralized by $N_1$, clearly every
subgroup of $Q$ contains $N_\C(R)$ in its centralizer. 

Fix $X\leq C_S(T)$ with $N_\C(R)\leq C_\F(X)$. To complete the proof of
(f), we need to show that $X\leq Q$. To prove this let $\Theta$ be the set of
all pairs $(Y,\phi)$ such that $RX\leq Y\leq V_R$, $\phi\in\Aut_\F(Y)$,
$[Y,\phi]\leq R$, $\phi|_X=\id_X$, and $\phi|_R\in\Aut_\C(R)$ has order $7$. As
$\Aut_\C(R)/\Inn(R)\cong GL_3(2)$, there exists an element $\phi_0$ of order
$7$ in $\Aut_\C(R)$. As $N_\C(R)\leq C_\F(X)$, $\phi_0$ extends to an
automorphism $\phi\in\Aut_\F(RX)$ with $\phi|_X=\phi_0$, and for such $\phi$ we
have $(RX,\phi)\in\Theta$. Thus $\Theta\neq\varnothing$ and we may fix
$(Y,\phi)\in\Theta$ such that $|Y|$ is maximal. 

Assume first that $Y=V_R$. Then $\phi$ is a morphism in $N_{N_\F(R)}(V_R)$ and
thus realized by conjugation with an element of $G_R$. Recall that
$H_1=O^2(H)R$ is normal in $G_R$ and contains $T$. Hence, $Q=C_{V_R}(H_1)$ is
normal in $G_R$ and thus $\phi$-invariant. As $[V_R,\phi]\leq R$ by definition
of $\Theta$, it follows $[Q,\phi]\leq R\cap Q=1$. As $U_R=EC_S(T)=EQ$ and
$\phi|_R$ acts fixed-point-freely on $E^\#$, it follows $Q=C_{U_R}(\phi)$. By
definition of $\Theta$, we have $\phi|_X=\id_X$ and thus $X\leq
C_{U_R}(\phi)=Q$. So $X\leq Q$ if $Y=V_R$. 

Assume now $Y<V_R$. Recall from above that $N_\F(R)$ is saturated. So we can
fix $\alpha\in\AA_{N_\F(R)}(Y)$. Then $\phi^\alpha\in\Aut_\F(Y^\alpha)$ and
$[Y^\alpha,\phi^\alpha]\leq R$ as $[Y,\phi]\leq R$ by definition of $\Theta$.
Recall also that $\phi|_R\in\Aut_\C(R)$ has order $7$. By
Lemma~\ref{L:AutR}(d), we have $O^2(\Aut_\F(R))=O^2(\Aut_\C(R))$. So we can
conclude that $\phi^\alpha|_R=(\phi|_R)^\alpha\in
O^2(\Aut_\F(R))^\alpha=O^2(\Aut_\F(R))\leq \Aut_\C(R)$. As $N_\C(R)=\F_T(N_1)$
by (d), there exists thus $n\in N_1$ with $\phi^\alpha|_R=c_n|_R$. Set
$\psi:=c_n|_{V_R}\in\Aut_\F(V_R)$. As $N_1\leq N$, we have $[V_R,\psi]\leq R$.
In particular, as $R\leq Y^\alpha\leq V_R$, we have $(Y^\alpha)^\psi=Y^\alpha$.
Thus, $\chi:=(\psi|_{Y^\alpha})^{-1}\circ\phi^\alpha\in\Aut_\F(Y^\alpha)$ is
well-defined. Observe also that $\chi|_R=\id_R$ and $[Y^\alpha,\chi]\leq R$, as
$[Y^\alpha,\psi]\leq [V_R,\psi]\leq R$ and $[Y^\alpha,\phi^\alpha]\leq R$. So
$\chi$ is an element of $C_{\Aut_\F(Y^\alpha)}(R)\cap
C_{\Aut_\F(Y^\alpha)}(Y^\alpha/R)$, which is a normal $2$-subgroup of
$\Aut_{N_\F(R)}(Y^\alpha)$. Since $Y^\alpha\in N_\F(R)^f$, the Sylow axiom
yields that $\Aut_{N_S(R)}(Y^\alpha)$ is a Sylow $2$-subgroup of
$\Aut_{N_\F(R)}(Y^\alpha)$. Hence, there exists $s\in N_{C_S(R)}(Y^\alpha)$
with $\chi|_{Y^\alpha}=c_s|_{Y^\alpha}$. So $\phi^\alpha=\psi|_{Y^\alpha}\circ
c_s|_{Y^\alpha}$ extends to $\rho=\psi\circ c_s|_{V_R}\in\Aut_\F(V_R)$. Since
$[V_R,\psi]\leq R$, the automorphism $\rho$ acts on $V_R/R$ in the same way as
$c_s|_{V_R}$. So writing $m$ for the order of $s$, we have $[V_R,\rho^m]\leq
R$. Moreover, $\rho^m$ extends $(\phi^\alpha)^m$. Since $Y<V_R$, we have
$Y<W:=N_{V_R}(Y)$. Note that $R\leq Y^\alpha\leq W^\alpha\leq V_R$, so
$[W^\alpha,\rho^m]\leq [V_R,\rho^m]\leq R$ and
$\rho^m|_{W^\alpha}\in\Aut_\F(W^\alpha)$. Therefore,
$\hat{\phi}:=(\rho^m|_{W^\alpha})^{\alpha^{-1}}\in\Aut_\F(W)$ with
$[W,\hat{\phi}]\leq R^{\alpha^{-1}}=R$. Moreover,
$\hat{\phi}|_R=(\phi|_R)^m\in\Aut_\C(R)$ has order $7$, as
$\phi|_R\in\Aut_\C(R)$ has order $7$ and $m$ is a power of $2$. Also
$\hat{\phi}|_X=(\phi|_X)^m=\id_X$ as $\phi|_X=\id_X$. This shows
$(W,\hat{\phi})\in\Theta$. As $|W|>|Y|$ and $(Y,\phi)\in\Theta$ was chosen such
that $|Y|$ is maximal, this contradicts the assumption that $Y < V_R$. So
ultimately, we have shown that $Y=V_R$.  As argued in the previous paragraph,
this yields $X\leq Q$, and thus completes the proof of (f). 

\textbf{(g)}: Since $\C \in \CC(\F)$, we may fix an involution $t \in
\I(\C)$.  We have $\I(\C) \subseteq \tilde{\X}(\C) \subseteq \X(\C)$ by
definition of these collections, so $\gen{t} \in \tilde{\X}(\C)$, and also
$\gen{t} \leq Q \leq \X(\C)$ by (f).  Therefore, $Q \in \tilde{\X}(\C)$ by
\cite[6.1.5]{AschbacherFSCT}.
\end{proof}

\begin{lemma}
\label{L:nearlystandard}
$\C$ is nearly standard.
\end{lemma}
\begin{proof}
By Lemma~\ref{L:terminal}(a), $\C$ is terminal in $\CC(\F)$. By
Lemma~\ref{L:towardstandard}(g), the collection $\tilde{\X}(\C)$ has a unique
maximal member. Hence, $\C$ is nearly standard by
\cite[Proposition~7]{AschbacherFSCT}. 
\end{proof}

\begin{lemma}
\label{L:AutFTAutC}
$\Aut_{\F}(T) \leq \Aut(\C)$. 
\end{lemma}
\begin{proof}
Let $\alpha \in \Aut_\F(T)$ and note that $\alpha \in C_\F(z)$. Recall that $z$
was chosen to be fully normalized. Thus, as $z \in \I(\H)$, $\H$ is a component
of $C_\F(z)$.  It follows from \cite[9.7]{AschbacherGeneralized} that there is
a unique component of $C_\F(z)$ with Sylow group $T$, so that $\H^\alpha = \H$
by Lemma~\ref{L:ConjugateComponents}(b).  Since $T$ is fully $\F$-normalized by
Hypothesis~\ref{H:maximal}, $\alpha$ extends to an automorphism
$\tilde{\alpha}$ of $Q_0T = C_S(T)T = C_S(R)T$ with the last equality by
Lemma~\ref{L:towardstandard}(b). From Lemma~\ref{L:standardsequence}(c), $R$ is
characteristic in $T$, so we have that $R^\alpha = R$, and hence that
$\tilde{\alpha}$ normalizes $C_S(R)$.  Thus, $\alpha \in N_{N_{\F}(V_R)}(R)$, a
model for which is, by definition, $G_R$. We may therefore choose $g \in
N_{G_R}(T)$ such that $\alpha = c_g|_T$. As $N := C_{G_R}(V_R/R)$ is a
normal subgroup of $G_R$, $g$ leaves invariant $O^2(N)R=\gen{T^N}$,
which is a model for $N_\C(R)$ by Lemma~\ref{L:towardstandard}(d), whence
$\alpha$ normalizes $N_\C(R)$.  Thus, $\alpha \in \Aut(\gen{\H,N_\C(R)}) =
\Aut(\C)$, the equality coming from the generation statement of
Lemma~\ref{L:gen}(b), and now the assertion follows as $\alpha$ was chosen
arbitrarily.
\end{proof}

\section{The centralizer of $\C$ and the elementary abelian case}\label{S:elemab}

We operate from now until just before the end of Section~\ref{S:cyclic}
under the following hypothesis and notation, although we will often
state it again for emphasis.

\begin{hypothesis}\label{H:standard}
Let $\F$ be a saturated fusion system over the $2$-group $S$, and let $q$
be an odd prime power. Suppose $\C \cong \F_{\Sol}(q)$ is a subsystem of $\F$
over the subgroup $T \in \F^f$.  Assume $\C \in \CC(\F)$ is standard in $\F$,
but that $\C$ is not a component of $\F$.
\end{hypothesis}

By Theorem~\ref{T:Cstandard} (and Remark~\ref{R:Subintrinsic}), if $\C$
is a Benson-Solomon system which is maximal \emph{and subintrinsic} in
$\CC(\F)$, then Hypothesis~\ref{H:standard} holds or $\C$ is a component of
$\F$. But we have not assumed that $\C$ is subintrinsic in $\CC(\F)$ in
Hypothesis~\ref{H:standard} and there is no obvious reason that
Hypothesis~\ref{H:standard} implies this property.  This means that the
results of Section~\ref{S:showStandard} are generally not applicable in
Sections~\ref{S:elemab}-\ref{S:cyclic}. 

\begin{notation}\label{N:standard}
Let $\Q$ be the centralizer of $\C$ (Remark~\ref{R:QCentralizer}), and let $Q$
be the Sylow group of $\Q$.
\end{notation}

Thus, in the case where $\C$ is subintrinsic in $\CC(\F)$, the group $Q$
was ultimately constructed in Lemma~\ref{L:towardstandard}(f,g).  For future
reference, we record the following lemma.

\begin{lemma}
\label{L:Qtight}
The following hold. 
\begin{enumerate}
\item[(a)] $\Q$ is tightly embedded in $\F$, and
\item[(b)] $C_S(T) = QZ(T)$. 
\end{enumerate}
\end{lemma}
\begin{proof}
Since $\C$ is assumed to be a standard subsystem in
Hypothesis~\ref{H:standard}, the subsystem $\Q$ exists and is saturated by
\cite[9.1.4, 9.1.5]{AschbacherFSCT}. Then $\Q$ is a tightly embedded subsystem
of $\F$ by \cite[9.1.6.2]{AschbacherFSCT}. Part (b) then follows from a
combination of Lemmas~\ref{L:CST} and \ref{L:lim1=0}. 
\end{proof}

\begin{lemma}
\label{L:Qelemaborquat}
One of the following holds.
\begin{enumerate}
\item[(a)] $Q$ is elementary abelian, or
\item[(b)] $Q$ is of $2$-rank $1$. 
\end{enumerate}
\end{lemma}
\begin{proof}
This is a direct consequence of Hypothesis~\ref{H:standard},
Lemma~\ref{L:Csplit}, and \cite[Theorem~8]{AschbacherFSCT}. 
\end{proof}

\begin{proposition}\label{P:m(Q)=1}
Assume Hypothesis~\ref{H:standard}. Then $Q$ has $2$-rank $1$. 
\end{proposition}

\begin{proof}
The subsystem $\Q$ is tightly embedded by Lemma~\ref{L:Qtight}(a).  Assume that
$Q$ has $2$-rank larger than $1$. Then by Lemma~\ref{L:Qelemaborquat}, $Q$ is
elementary abelian and $|Q|>2$. Moreover, by Lemma~\ref{QQ}, $\F_Q(Q)$ is
tightly embedded in $\F$. By \cite[9.4.11]{AschbacherFSCT}, we can fix $P\in
Q^\F$ such that $P\leq N_S(Q)$ and $P\neq Q$. By \cite[3.1.8]{AschbacherFSCT},
we have 
\[
P\cap Q=1.
\] 
As $\C$ is standard, we have $\C\unlhd N_\F(Q)$. In particular, we can form the
product $\C P$ inside of $N_\F(Q)$. As $Q$ is normal in $N_\F(Q)$, we have
$Q\not\in P^{\C P}$. Furthermore, if $\alpha\in \Hom_{\C P}(P,TP)$ then
$\alpha$ induces the identity on $PT/T$ by the construction of $\C P$ in
\cite{Henke2013} and since $P\cong Q$ is abelian. So $TP=TP^\alpha$. Hence,
replacing $P$ by a suitable $\C P$-conjugate of $P$, we may assume 
\[
P\in (\C P)^f.
\]
Then by \cite[Theorem~3.4.2]{AschbacherFSCT}, $\F_P(P)$ is tightly embedded in
$\C P$.

By \cite[Theorem~3.10]{HenkeLynd2018}, $\Out(\C)$ is cyclic. Note that $N_S(Q)$
induces automorphisms of $\C$ via conjugation as $\C\unlhd N_\F(Q)$. Moreover,
the elements of $N_S(Q)$ inducing inner automorphisms of $\C$ are precisely the
elements in $TC_S(T)$. Thus, $N_S(Q)/TC_S(T)$ is cyclic. Writing $z$ for
the unique involution in $Z(T)$, by Lemma~\ref{L:Qtight}(b),
$C_S(T)=\langle z\rangle Q$ and so $TC_S(T)=TQ$.  Since $P\cong Q$ is
elementary abelian of $2$-rank at least $2$, it follows that $P\cap (TQ)\neq
1$. Let $1\neq x\in P\cap (TQ)$ and write $x=uv$ with $u\in T$ and $v\in Q$.
Note that $u$ and $v$ commute. As $x$ is an involution, it follows that $u$ and
$v$ have order at most $2$. If $u=1$ then $x=v\in P\cap Q$ contradicting $P\cap
Q=1$.  Hence $u$ is an involution. Let $\alpha\in \Hom_{\C P}(C_{TP}(x),TP)$
such that $x^\alpha\in (\C P)^f$. We proceed now in several steps to reach a
contradiction.

\smallskip
\noindent
\emph{Step~1:} We show that $x^\alpha\in C_S(T)$ and $x^\alpha=zv^\alpha$ with
$v^\alpha\in Q$. 

For the proof note first that, as $\C\unlhd N_\F(Q)$, we have $T\unlhd N_S(Q)$
and thus $Z(T)=\langle z\rangle\unlhd N_S(Q)$. Hence, $z$ is central in
$N_S(Q)$ and thus fully centralized in $\C P$. As $u\in T$ is an involution and
all involutions in $T$ are $\C$-conjugate by Lemma~\ref{L:involutionsconjugate},
the element $u$ is $\C P$-conjugate to $z$. Hence, there exists $\phi\in
\Hom_{\C P}(C_{T P}(u),TP)$ such that $u^\phi=z$. Note that $x,v\in C_{TP}(u)$,
since $x=uv$ and $u$ and $v$ commute.  We obtain $x^\phi=zv^\phi$, where
$v^\phi\in Q\leq C_S(T)$, as $v\in Q$ and $\phi$ is a morphism in $N_\F(Q)$.
Since  $z\in Z(T)$, it follows $T\leq C_S(x^\phi)$.  Recall that $\alpha$ was
chosen such that $x^\alpha\in (\C P)^f$. Thus, using
Lemma~\ref{GetStronglyClosedinNormalizer}, we can conclude that $T\leq
C_S(x^\alpha)$ and so $x^\alpha\in C_S(T)$. Note that $u,v\in C_{TP}(x)$ as $u$
and $v$ commute. Moreover, since $\alpha$ is a morphism in $N_\F(Q)$, we have
$v^\alpha\in Q\leq C_S(T)$. So $x^\alpha=u^\alpha v^\alpha$ and
$u^\alpha=x^\alpha(v^\alpha)^{-1}\in C_S(T)$.  As $T$ is strongly closed in
$N_\F(Q)$, we have $u^\alpha\in T$ and thus $u^\alpha\in Z(T)=\langle
z\rangle$. As $u\neq 1$, it follows $u^\alpha=z$ and $x^\alpha=zv^\alpha$ with
$v^\alpha\in Q$. This completes Step~1.

\smallskip
\noindent
\emph{Step~2:} We show $C_\C(z)\leq C_{\C P}(x^\alpha)$. 

For the proof, we may assume that $x^\alpha\neq z$. By definition of $Q$, we
have $\C\leq C_\F(Q)$. By Step~1, $x^\alpha\in Q\langle z\rangle$.
Therefore $C_\C(z)\leq C_\F(Q\langle z\rangle)\leq
C_{N_\F(Q)}(x^\alpha)$. Let $D\in C_\C(z)^{fc}$ and let
$\chi\in\Aut_{C_\C(z)}(D)$ be an arbitrary element of odd order. Then $\chi$
extends to some $\hat{\chi}\in\Aut_{N_\F(Q)}(D\<x^\alpha\>)$ with
$(x^\alpha)^{\hat{\chi}}=x^\alpha$. The order of $\hat{\chi}$ equals the order
of $\chi$ and is therefore odd. As $x^\alpha\neq z$ is by Step~1 an involution
centralizing $T$, we have $x^\alpha\not\in T$ and thus $(D\<x^\alpha\>)\cap
T=D$. Moreover, clearly $[D\<x^\alpha\>,\hat{\chi}]\leq D$ and
$\hat{\chi}|_D=\chi$ is a morphism in $\C$. By
\cite[Lemma~6.2]{BrotoLeviOliver2003}, we have $D\in \C^c$. So it follows from
the definition of $\C P$ in \cite{Henke2013} that $\hat{\chi}$ is a morphism in
$\C P$. Hence, $\chi$ is a morphism in $C_{\C P}(x^\alpha)$. By Alperin's
fusion theorem \cite[Theorem~I.3.6]{AschbacherKessarOliver2011}, $C_\C(z)$ is
generated by the collection of automorphism groups $\Aut_{C_\C(z)}(D)$ as
$D$ ranges over $C_\C(z)^{fc}$.  But $\Aut_{C_\C(z)}(D) =
O^2(\Aut_{C_\C(z)}(D))\Aut_T(D)$ since $\Aut_T(D)$ is a Sylow $2$-subgroup of
$\Aut_{C_{\C}(z)}(D)$ for such $D$, so $C_\C(z)$ is generated by $\Inn(T)$
together with $O^2(\Aut_{C_\C(z)}(D))$ as $D$ ranges over $C_\C(z)^{fc}$. As
$T\leq C_S(x^\alpha)$, it follows that $C_\C(z)\leq C_{\C P}(x^\alpha)$.

\smallskip
\noindent
\emph{Step~3:} We show that $P^\alpha\unlhd C_{\C P}(x^\alpha)$ and
$P^\alpha\cap T\leq \langle z\rangle$.  

As remarked above, $\F_P(P)$ is tightly embedded in $\C P$. Hence, it follows
from (T1) that $P^\alpha\unlhd N_{\C P}(\langle x^\alpha\rangle)=C_{\C
P}(x^\alpha)$. In particular, as $C_\C(z)\leq C_{\C P}(x^\alpha)$ by
Step~2, it follows that $P^\alpha\cap T$ is strongly closed in $C_\C(z)$. As
$P^\alpha\cap T$ is abelian, \cite[Corollary~I.4.7]{AschbacherKessarOliver2011}
gives that $P^\alpha\cap T$ is normal in $C_\C(z)$. Since $C_\C(z)/\langle
z\rangle$ is the $2$-fusion system of $\Omega_7(q)$, which is not a Goldschmidt
group, $C_\C(z)/\gen{z}$ is simple by \cite[Theorem~5.6.1]{AschbacherQFP} (see
also \cite[Proposition~1.17]{AOV2017}).  Therefore, $P^\alpha\cap T\leq
\langle z\rangle$ as required.

\smallskip
\noindent
\emph{Step~4:} We show that $[T,P^\alpha]=1$. 

As $C_\C(z)=O^2(C_\C(z))$, we have 
\[
T=\mathfrak{hyp}(C_\C(z))=\langle [Y,\beta]\colon Y\leq T,\;\beta\in
\Aut_{C_\C(z)}(Y)\mbox{ of odd order}\rangle.
\] 
Let $Y\leq T$ and $\beta\in\Aut_{C_\C(z)}(Y)$ of odd order. We will show that
$[Y,\beta,P^\alpha]=1$, which is sufficient to complete Step~4. By Step~2,
$C_\C(z)\leq C_{\C P}(x^\alpha)$. As $P^\alpha\unlhd C_{\C P}(x^\alpha)$
by Step~3, we can thus extend $\beta$ to $\hat{\beta}\in\Aut_{\C P}(YP^\alpha)$
with $(P^\alpha)^{\hat{\beta}}=P^\alpha$. By the definition of $\C P$ in
\cite{Henke2013} and since $P$ is abelian, we have $[P^\alpha,\hat{\beta}]\leq
P^\alpha\cap T\leq \<z\>$, where the last inclusion uses Step~3. In particular,
$[P^\alpha,\hat{\beta},Y]=1$. As $P^\alpha\unlhd C_{\C P}(x^\alpha)$ and $T$
centralizes $x^\alpha$ by Step~1, $T$ normalizes $P^\alpha$. Hence, again using
Step~3, we conclude $[Y,P^\alpha]\leq [T,P^\alpha]\leq T\cap P^\alpha\leq
\langle z\rangle$ and so $[Y,P^\alpha,\hat{\beta}]=1$.  It follows now from the
Three-Subgroup-Lemma that $[Y,\beta,P^\alpha]=[\hat{\beta},Y,P^\alpha]=1$. This
finishes Step~4.

\smallskip
\noindent
\emph{Step~5:} We now derive the final contradiction. 

By Step~4, we have $P^\alpha\leq C_S(T)$. As we saw above, $C_S(T)=Q\langle
z\rangle$ and thus $Q$ has index $2$ in $C_S(T)$. Since $|P^\alpha|=|Q|>2$, it
follows $P^\alpha\cap Q\neq 1$. However, as $Q\not\in P^{\C P}$, the subgroup
$P^\alpha$ is an $\F$-conjugate of $Q$ not equal to $Q$. Hence, by
\cite[3.1.8]{AschbacherFSCT}, we have $P^\alpha\cap Q=1$. This contradiction
completes the proof.  
\end{proof}

We are thus left with the case that $Q$ has $2$-rank $1$, i.e. is either
cyclic or quaternion. We end this section with a lemma which handles a residual
situation occurring in this context. It will be needed both in
Section~\ref{S:quaternion} to exclude the quaternion case and in
Section~\ref{S:cyclic} to handle the cyclic case.  When $Q$ is of
$2$-rank $1$, the unique involution in $Q$ lies in $\I(\C)$ by (S2).  This
explains the choice of notation for it below.

\begin{lemma}
\label{L:2rank1residual}
Assume Hypothesis~\ref{H:standard} with $Q$ of $2$-rank $1$. Let $t$ be the
unique involution in $Q$ and fix a subnormal subsystem $\F_0$ of $\F$ over $S_0
\leq S$ such that $t\in S_0$. Write $z$ for the unique involution in $Z(T)$. Then the following hold:
\begin{enumerate}
\item[(a)] $\gen{t}$ is fully $\F_0$-normalized. 
\item[(b)] If $[T,C_{S_0}(t)]\neq 1$, then $\C$ is a component of
$C_{\F_0}(t)$. Moreover,
\[
\Omega_1(C_{S_0}(C_{S_0}(t)))=\Omega_1(Z(C_{S_0}(t)))=\<t,z\>.
\]
\item[(c)] Assume $Q \leq S_0$ and $\C \leq C_{\F_0}(t)$. If $\gen{t}
\leq Z(S_0)$, then $\gen{t}$ is not weakly $\F_0$-closed in $Z(S_0)$.
\end{enumerate}
\end{lemma}
\begin{proof}
As $\Q$ is tightly embedded from Lemma~\ref{L:Qtight}(a), there is a fully
$\F$-normalized $\F$-conjugate of $\gen{t}$ in $Q$ by
\cite[3.1.5]{AschbacherFSCT}. It follows that $\gen{t}$ is fully
$\F$-normalized, since $t$ is the unique involution in $Q$. So (a) follows from
\ref{L:Subnormalfn}. In particular, $C_{\F_0}(t)$ is saturated.

In the proof of (b) and (c), we will use that $\C$ is normal in $C_\F(t)$
by (S2). In particular $\C$ is a component of $C_\F(t)$. In addition, we will
use that $C_S(T)=\langle z\rangle Q$ from Lemma~\ref{L:Qtight}(b). 

For the proof of (b) assume that $[T,C_{S_0}(t)]\neq 1$. By (a) and
\cite[8.23.2]{AschbacherGeneralized}, $C_{\F_0}(t)$ is subnormal in
$C_\F(t)$. So by \cite[9.6]{AschbacherGeneralized}, $\C$ is a component of
$C_{\F_0}(t)$ as $[T,C_{S_0}(t)]\neq 1$. In particular, $T\leq C_{S_0}(t)$. As
$\C$ is normal in $C_\F(t)$, we have $T\unlhd C_S(t)$ and in particular, $z\in
Z(C_{S_0}(t))$. As $C_S(T)=\langle z\rangle Q$, we obtain $\<t,z\>\leq
\Omega_1(Z(C_{S_0}(t)))\leq \Omega_1(C_{S_0}(C_{S_0}(t))) \leq
\Omega_1(C_S(T))=\<t,z\>$ and this implies that (b) holds.

For the proof of (c) assume now that $Q \leq S_0$, $\C \leq C_{\F_0}(t)$,
and $\gen{t} \leq Z(S_0)$ is weakly $\F_0$-closed in $Z(S_0)$. Then in
particular, $S_0=C_{S_0}(t)$ and $C_{\F_0}(t)$ is saturated. As $T\leq
C_{S_0}(t)$ is nonabelian, (b) gives that $\C$ is a component of $C_{\F_0}(t)$
and $\Omega_1(Z(S_0))=\<t,z\>$. As $\C$ is normal in $C_\F(t)$, one easily
checks that $\C$ is $C_{\F_0}(t)$-invariant (using the equivalent definition of
$\F$-invariant subsystems given in
\cite[Proposition~I.6.4(d)]{AschbacherKessarOliver2011}). Hence, by a theorem
of Craven \cite{Craven2011}, $\C=O^{p^\prime}(\C)$ is normal in $C_{\F_0}(t)$.
We proceed now in several steps to reach a contradiction.

\smallskip
\noindent
\emph{Step~1:} We show that $\gen{t}$ is not weakly $\F_0$-closed.
Suppose this is false. Then for every essential subgroup $P$ of $\F_0$, we have
$t\in Z(S_0)\leq C_{S_0}(P)\leq P$ and $t$ is fixed by $\Aut_{\F_0}(P)$. So by
Alperin's Fusion Theorem \cite[Theorem~I.3.6]{AschbacherKessarOliver2011}, we
have $t\in Z(\F_0)$. Hence, $\C$ is a component of $C_{\F_0}(t) = \F_0$. As
$\F_0$ is subnormal in $\F$, it follows that $\C$ is subnormal in $\F$ and thus
a component of $\F$, contradicting Hypothesis \ref{H:standard}. This completes
Step~1.

As shown in Step~1, there exists an $\F_0$-conjugate $f$ of $t$ with $f\neq t$. Fix
such $f$ from now on. 

\smallskip
\noindent
\emph{Step~2:} We show that $f\not\in QT$ and $t$ is weakly $\F_0$-closed in
$QT$. 

Assuming $f\in QT$, we would have $f\in \Omega_1(QT)\leq T\<t\>$. So $f\in T$
or $f=ut$ with $u\in T$. In the latter case, since $t\in Q\leq C_S(T)$ and $f$
is an involution, $u$ is an involution. By Lemma~\ref{L:involutionsconjugate}, all
involutions in $T$ are $\C$-conjugate.  Moreover $\C\leq C_{\F_0}(t)$. So
if $f\in T$, then $f$ is $\F_0$ conjugate to $z$, and if $f=ut$ for some
involution $u\in T$, then $f$ is $\F_0$-conjugate to $zt$. In both cases we get
a contradiction to the assumption that $t$ is $\F_0$-closed in $Z(S_0)$. So
$f\not\in QT$.  Because of the arbitrary choice of $f$, this completes Step~2. 

As $\C$ is normal in $C_{\F_0}(t)$, we can form the product system $\C\gen{f}$
(as defined in \cite{Henke2013}) in $C_{\F_0}(t)$ over the $2$-group
$T\gen{f}$. 

\smallskip
\noindent
\emph{Step~3:} We show that $f$ is $\C\gen{f}$-conjugate to every element of the
coset $fE$.

Note first that $F^*(\C\gen{f}) = \C$. As all
involutions in $T\gen{f}-T$ are $\C\gen{f}$-conjugate by
Proposition~\ref{P:fieldconjugate}(a), we see that indeed $f$ is
$\C\gen{f}$-conjugate to every element of $fE$.

\smallskip
\noindent
\emph{Step~4:} We derive the final contradiction. 

Since $\C$ is normal in $C_\F(t)$ and $S_0\leq C_S(t)$, $S_0$ induces
automorphisms of $\C$ by conjugation. As $C_S(T) = Q\gen{z}$ and $\Aut(\C)$ is
cyclic by \cite[Theorem~3.10]{HenkeLynd2018}, it follows that $QT=TC_S(T)$ is
normal in $S_0$ and $S_0/QT$ is cyclic. Now let $\alpha \in \AA_{\F_0}(f)$
with $f^\alpha = t$.  Then $\alpha$ is defined on $\gen{f}E$ and, hence $t$ is
$\F_0$-conjugate to every member of the coset $tE^{\alpha}$ by Step~3. Since
$E^{\alpha}$ is of $2$-rank $3$, while $S_0/QT$ is cyclic, it follows that
$E^{\alpha} \cap QT \neq 1$. For $1\neq e\in E^{\alpha} \cap QT$, $t$ is
conjugate to $te\in QT$.  This contradicts Step~2.
\end{proof}

\section{The quaternion case}\label{S:quaternion}

In this section we show, assuming Hypothesis~\ref{H:standard}, that $Q$ is not
quaternion using Aschbacher's classification of quaternion fusion packets
\cite{AschbacherQFP}. When combined with Proposition~\ref{P:m(Q)=1}, the
results of this section reduce to the case in which $Q$ is cyclic, which is
handled in Section~\ref{S:cyclic}.

The Classical Involution Theorem identifies the finite simple groups which have
a \emph{classical involution}, that is, an involution whose centralizer has a
component (or solvable component) isomorphic to $SL_2(q)$ (or $SL_2(3)$)
\cite{Aschbacher1977I, Aschbacher1977II}. With the exception of $M_{11}$, the
simple groups having a classical involution are exactly the groups of Lie type
in odd characteristic other than $L_2(q)$ or ${ }^2 G_2(q)$, where the
$SL_2(q)$ components in involution centralizers are fundamental subgroups
generated by the center of a long root subgroup and its opposite.

In a group with a classical involution, the collection of these $SL_2(q)$
subgroups satisfies special fusion theoretic properties that were identified
and abstracted by Aschbacher in \cite[Hypothesis $\Omega$]{Aschbacher1977I}.
More recently, Aschbacher has formulated these conditions in fusion systems in
the definition of a \emph{quaternion fusion packet}, and his memoir
\cite{AschbacherQFP} classifies all such packets.

\begin{definition}
A \emph{quaternion fusion packet} is a pair $\tau = (\F, \Omega)$, where $\F$
is a saturated fusion system on a finite $2$-group $S$, and $\Omega$ is an
$\F$-invariant collection of subgroups of $S$ such that
\begin{enumerate}
\item[(QFP1)] There exists an integer $m$ such that for all $K \in \Omega$, $K$
has a unique involution $z(K)$ and is nonabelian of order $m$. 
\item[(QFP2)] For each pair of distinct $K, J \in \Omega$, $|K \cap J| \leq 2$.
\item[(QFP3)] If $K, J \in \Omega$ and $v \in J-Z(J)$, then $v^\F \cap
C_S(z(K)) \subseteq N_S(K)$. 
\item[(QFP4)] If $K, J \in \Omega$ with $z = z(K) = z(J)$, $v \in K$, and $\phi
\in \Hom_{C_\F(z)}(\gen{v}, S)$, then either $v^\phi \in J$ or $v^\phi$
centralizes $J$. 
\end{enumerate}
\end{definition}

We assume the following hypothesis until the last result in this section.
(See Section~\ref{SSS:subnormalclosure} for a discussion of normal and
subnormal closures in fusion systems.)

\begin{hypothesis}\label{H:quaternion}
Hypothesis~\ref{H:standard} and Notation~\ref{N:standard} hold with $Q$
quaternion. Let $t$ be the unique involution in $Q$. Set $\Omega = Q^\F$,
denote by $\F^\circ$ the subnormal closure of $Q$ in $\F$ over the subgroup
$S^\circ \leq S$, and set $\Omega^\circ = Q^{\F^\circ}$. 
\end{hypothesis}

A tightly embedded subsystem with quaternion Sylow $2$-subgroups, such as the
centralizer system $\Q$ in Hypothesis~\ref{H:quaternion}, always yields a
quaternion fusion packet in a straightforward way.

\begin{lemma}
\label{L:qfp}
$(\F,\Omega)$ is a quaternion fusion packet.
\end{lemma}
\begin{proof}
We go through the list of axioms. (QFP1) holds by definition of $\Omega$. Note
that $\Omega \subseteq \mathcal{P}^*$ in the sense of Definition~3.1.9 of
\cite{AschbacherFSCT}. Hence, by \cite[3.1.12.2]{AschbacherFSCT}, $K \cap J =
1$ for each pair of distinct $K, J \in \Omega$.  This shows that (QFP2) holds,
and that any element of $S$ centralizing $z(K)$ must normalize $K$, so that
(QFP3) also holds. Finally, under the hypotheses of (QFP4), $K = J$ in the
current situation. Fix $1 \neq v \in K$ and $\phi \in
\Hom_{C_\F(z(K))}(\gen{v},S)$.  Then $z(K) \in \gen{v}$, and $z(K)^\phi =
z(K)$. Also, $\gen{v} \in \mathcal{P}$, and $K \in \mathcal{P}^*$, in the sense
of Definition~3.1.9 of \cite{AschbacherFSCT}. Since $\gen{v}^\phi \cap K > 1$,
we see from \cite[3.1.14]{AschbacherFSCT} (applied with $\gen{v}$, $\phi$, and
$K$ in the role of $P$, $\psi$, and $R$) that $\gen{v}^\phi \leq K$. This shows
that (QFP4) holds.
\end{proof}

\begin{lemma}
\label{L:red1}
Let $\F_0$ be a subnormal subsystem of $\F$ over the subgroup $S_0 \leq S$.
Assume that $Q \leq S_0$, and that $\C \leq C_{\F_0}(t)$. Then $Q^{\F_0} \neq
\{Q\}$.
\end{lemma}
\begin{proof}
Suppose on the contrary that $Q^{\F_0} = \{Q\}$. Then $Q$ is normal in $S_0$,
and so $t \in Z(S_0)$. Let $\alpha$ be a morphism in $\F_0$ with $t^\alpha \in
Z(S_0)$. By the extension axiom, we may assume that $\alpha$ is defined on $Q$,
and then $Q^\alpha = Q$ by assumption, so that $t^\alpha = t$.  This shows that
$\gen{t}$ is weakly $\F_0$-closed in $Z(S_0)$, contradicting Lemma~\ref{L:2rank1residual}(c). 
\end{proof}

\begin{lemma}
\label{L:red2}
$\C$ is a component of $C_{\F^\circ}(t)$. In particular, $\C$ is contained in
$C_{\F^\circ}(t)$. 
\end{lemma}
\begin{proof}
For each $i\geq 0$ set $\F_i:=\sub_i(\F,Q)$ and write $S_i$ for the Sylow
of $\F_i$. Recall from Section~\ref{SSS:subnormalclosure} that
$\F_{i+1} \norm \F_i$ for each $i \geq 0$, and $\F^{\circ}$ is by definition
the terminal member of this series. By Lemma~\ref{L:2rank1residual}(a),
$\gen{t}$ is fully normalized in $\F_i$ for $i \geq 0$, so $C_{\F_i}(t)$ is
saturated for each $i$.  

We assume now that the assertion is false. As $\C$ is normal in $C_\F(t)$ by
(S2), there exists $i\geq 0$ such that $\C$ is not a component of
$C_{\F_{i+1}}(t)$.  Fix the smallest such $i$. Then $\C$ is a component
of $C_{\F_i}(t)$.  By Lemma~\ref{L:red1}, we have that $Q^{\F_i} \neq \{Q\}$.
Fix $Q' \in Q^{\F_i}-\{Q\}$.  As $\Q$ is tightly embedded in $\F$
(Lemma~\ref{L:Qtight}(a)), we have $Q \cap Q' = 1$ by
\cite[3.1.12.2]{AschbacherFSCT}, and we have 
\[
Q' \leq C_{S_i}(Q) \leq C_{S_i}(t)
\]
by \cite[3.3.5]{AschbacherFSCT}.  By definition of $\F_{i+1}$, we have $Q'\leq
S_{i+1}$ and thus $Q'\leq C_{S_{i+1}}(t)$. As $C_S(T)=QZ(T)$ by
Lemma~\ref{L:Qtight}(b), it follows $[Q',T]\neq 1$,  and thus
$[T,C_{S_{i+1}}(t)]\neq 1$. Hence, $\C$ is a component of $C_{\F_{i+1}}(t)$ by
Lemma~\ref{L:2rank1residual}(b). This contradicts the
choice of $i$. 
\end{proof}

\begin{lemma}
\label{L:Eqfp} 
The pair $(\F^\circ, \Omega^\circ)$ is a quaternion fusion packet, $\F^{\circ}$
is the normal closure of $Q$ in $\F^{\circ}$, and $\F^{\circ}$ is transitive on
$\Omega^\circ$.  
\end{lemma}
\begin{proof}
Note that $(\F^{\circ}, \Omega^\circ)$ is a quaternion fusion packet by
Lemma~\ref{L:qfp} and \cite[Lemma~6.4.2.1]{AschbacherQFP}. Recall that 
$\F^{\circ}$ is the subnormal closure of $Q$ in $\F$. So the second statement
follows from the definition of subnormal closure, while the third holds by
definition of $\Omega^\circ$.
\end{proof}

Now remove the standing assumption that Hypothesis~\ref{H:quaternion} holds. 

\begin{proposition}\label{P:quaternion}
Assume Hypothesis~\ref{H:standard}. Then $Q$ is cyclic. 
\end{proposition}
\begin{proof}
We argue by contradiction, so that $Q$ is quaternion by
Proposition~\ref{P:m(Q)=1}. Hence, Hypothesis~\ref{H:quaternion} holds, and so
we adopt the notation there.  By Lemma~\ref{L:Eqfp}, the pair $(\F^{\circ},
\Omega^\circ)$ satisfies the hypotheses of Theorem~1 of \cite{AschbacherQFP}.
From that theorem, one of the following holds: either 
\begin{enumerate}
\item $t \in Z(\F^{\circ})$, or
\item $t \in O_2(\F^{\circ})-Z(\F^{\circ})$, or 
\item there is a finite group $G$ with Sylow $2$-subgroup $S^{\circ}$ such that
$\F^\circ = \F_{S^\circ}(G)$, and one of the following holds,
\begin{enumerate}
\item $S^{\circ}$ has $2$-rank at most $3$, or 
\item $G \in \Chev^*(p)$ for some odd prime $p$, or
\item $G$ is quasisimple with $Z(G)$ a $2$-group, and $G/Z(G) \cong Sp_6(2)$ or
$\Omega^+_8(2)$. 
\end{enumerate}
\end{enumerate}
Observe that in all cases, 
\begin{eqnarray}
\label{E:CcompCEt}
\text{$\C$ is a component of $C_{\F^\circ}(t)$}
\end{eqnarray}
by Lemma~\ref{L:red2}.

In Case (1), $\C$ is a component of $C_{\F^{\circ}}(t) = \F^{\circ}$. Hence
$\C$ is a component of $\F$ since $\F^{\circ}$ is subnormal in $\F$, contrary
to Hypothesis~\ref{H:standard}.  In Case (2), the hypotheses of
\cite[Theorem~2]{AschbacherQFP} hold for $(\F^{\circ}, \Omega^\circ)$, and then
by \cite[Lemma~6.7.3]{AschbacherQFP}, we have that $\F^\circ$ is constrained.
Thus, $C_{\F^\circ}(t)$ is also constrained, and hence $\C \leq
E(C_{\F^\circ}(t)) = 1$, a contradiction.

Case (3)(a) yields a contradiction, since $QT \leq S^{\circ}$ is of $2$-rank
$5$ by Lemma~\ref{L:standardsequence}(d).  In Case (3)(b), note that
$C_{\F^\circ}(t)$ is the fusion system of $C_G(t)/O(C_G(t))$ by
\cite[I.5.4]{AschbacherKessarOliver2011}.  By \eqref{E:CcompCEt} and
Lemma~\ref{L:knowncent}, the hypothesis of Lemma~\ref{L:ComponentsFSGroups}
hold, and so there is a component $K$ of $C_G(t)/O(C_G(t))$ such that $\C$ is
the $2$-fusion system of $K$ by that lemma. This contradicts the fact that $\C$
is exotic \cite[Proposition~3.4]{LeviOliver2002}.  

In Case 3(c) we may assume that Case (2) does not hold, so that $t \notin
Z(\F^\circ)$. Then $t \notin Z(G)$. As $Sp_6(2)$ and $\Omega^+_8(2)$ are of
characteristic $2$-type and as $t \notin Z(G)$, we have that $C_G(t)$ is of
characteristic $2$. Hence $C_{\F^\circ}(t)$ is constrained.  We therefore
obtain the same contradiction here as in Case (2).
\end{proof}

\section{The cyclic case and the proof of Theorem~\ref{T:main}}\label{S:cyclic}

In this section, we finish the proof of Theorem~\ref{T:main}. Using
Theorem~\ref{T:Cstandard} and Proposition~\ref{P:quaternion}, one quickly
reduces to the case that $\C$ is standard and the centralizer $Q$ of $\C$ in
$S$ is cyclic. We therefore assume the following hypothesis and notation for
most of this section.

\begin{hypothesis}\label{H:cyclic}
Hypothesis~\ref{H:standard} and Notation~\ref{N:standard} hold with $Q$ cyclic.
Write $Z(T)=\gen{z}$, $\Omega_1(Q) = \gen{t}$, $S_t = C_S(t)$, and $\F_t
= C_\F(t)$. 
\end{hypothesis}

\begin{lemma}
\label{L:cyclicbasic}
Assume Hypothesis~\ref{H:standard}. Then the following hold.
\begin{enumerate}
\item[(a)] $\gen{t} \in \F^f$,
\item[(b)] $C_S(T) = Q\gen{z}$, 
\item[(c)] $\Omega_1(C_S(S_t)) = \Omega_1(Z(S_t)) = \gen{t,z}$, 
\item[(d)] if $t\in Z(S)$, then $\gen{t}$ is not weakly $\F$-closed in $Z(S)$, and 
\item[(e)] $t$ is not $\F$-conjugate to $z$. 
\end{enumerate}
\end{lemma}
\begin{proof}
Parts (a), (c), and (d) follow from Lemma~\ref{L:2rank1residual} applied
with $\F_0=\F$, while (b) is just a recollection of
Lemma~\ref{L:Qtight}(b). 

It remains to prove (e). As $Q$ is cyclic, we have $\Q=Q$. By (T1) in the
definition of tight embedding (Definition~\ref{D:tightlyEmb}), we have $Q = \Q
\norm \F_t$. Further, $Q = C_{S_t}(\C)$ by \cite[9.1.6.3]{AschbacherFSCT}.
Write quotients by $Q$ with bars.  Note that $C_{\bar{S}_t}(\bar{\C})$ is
trivial by \cite[Lemma~1.14]{Lynd2015}, and $\bar{\C} \cong \C$. Thus,
$F^*(\bar{\F}_t) = \bar{\C}$ is isomorphic to a Benson-Solomon system. By
\cite[Theorem~4.3]{HenkeLynd2018}, this quotient is therefore a split extension
of $\bar{\C}$ by a $2$-group of outer automorphisms, and in particular,
$O^2(\bar{\F}_t) = \bar{\C}$. It follows that $O^2(\F_t) \leq Q\C$. Since
$O^2(Q\C) = \C$ and since $O^2(O^2(\F_t)) = O^2(\F_t)$, we have that $O^2(\F_t)
= \C$. Hence, $t$ is fully normalized and not in the hyperfocal subgroup of
$\F_t$, while $z^\alpha$ is contained in the hyperfocal subgroup of $\H^\alpha
\leq C_\F(z^\alpha)$ for every $\alpha \in \AA(\gen{z})$. Thus, $t$ and $z$ are
not $\F$-conjugate.
\end{proof}

\begin{lemma}
\label{L:t2central}
If Hypothesis~\ref{H:cyclic} holds, then $\C$ is not subintrinsic in $\CC(\F)$.  
\end{lemma}
\begin{proof}
Assume on the contrary that $\C$ is subintrinsic in $\CC(\F)$. As argued in
Remark~\ref{R:Subintrinsic}, this means that $z\in\I_\F(\H)$ where
$\H=C_\C(z)$. 

Assume first that $t \notin Z(S)$. Then $S_t < S$, so that $S_t < N_S(S_t)$.
Fix $a \in N_S(S_t)-S_t$.  Then $t^a = tz$ and $z^a = z$ by
Lemma~\ref{L:cyclicbasic}(c,e).

As $z\in\I_\F(\H)$, we may pick $\alpha\in\AA(z)$ such that $\H^\alpha$ is a
component of $C_\F(z^\alpha)$. Since $z^a = z$, we may define
$\ov{a}:=a^\alpha\in C_S(z^\alpha)$. Then $(\H^\alpha)^{\ov{a}}$ is a component
of $C_\F(z)$ on $(T^\alpha)^{\ov{a}}=(T^a)^\alpha$. However, if
$(\H^\alpha)^{\ov{a}} \neq \H^\alpha$, then since Sylow subgroups of distinct
components commute, we would have $[(T^a)^\alpha,T^\alpha]=1$ and thus $T^a
\leq C_{S_t}(T) \leq Q\gen{z}$ by Lemma~\ref{L:cyclicbasic}(b), and we would be
forced to conclude that $T$ is abelian. Since this is not the case, $\ov{a}$
normalizes $T^\alpha$, which implies that $a$ normalizes $T$.  Hence, by (S4)
in Definition~\ref{D:standard}, conjugation by $a$ restricts to an automorphism
of $\C$. As $t\in\F^f$ by Lemma~\ref{L:cyclicbasic}(a), it follows from
\cite[9.1.6.3]{AschbacherFSCT} that $Q = C_{S_t}(\C)$. Thus, $a$ acts also on
$Q = C_{S_t}(\C)$, so that $t^a = t$. This contradicts the choice of $a$. 

We have shown that $t \in Z(S)$. So Lemma~\ref{L:cyclicbasic}(c) yields
$V:=\Omega_1(Z(S)) = \Omega_1(Z(S_t)) = \gen{t,z}$, while
Lemma~\ref{L:cyclicbasic}(e) says that $t$ is not $\F$-conjugate to $z$. Notice
that $\Aut_\F(V)$ is by the Sylow axiom of odd order, since $S$ centralizes
$V$. As $\Aut(V)\cong S_3$ and every element of $\Aut(V)$ of order $3$ acts
transitively on $V^\#$, it follows that $\Aut_\F(V)=1$. If $t$ is
$\F$-conjugate to an element of $Z(S)$ under an $\F$-morphism $\alpha$, then by
Lemma~\ref{AAnonempty}, $\alpha$ can be assumed to be an $\F$-automorphism of
$S$, which thus restricts to an element of $\Aut_\F(V)$.  This shows that
$\gen{t}$ is weakly closed in $Z(S)$, contradicting
Lemma~\ref{L:cyclicbasic}(d). 
\end{proof}

We are now in the position to complete the proof of Theorem~\ref{T:main}, so we
now drop the standing assumption from the beginning of Section~\ref{S:elemab}
that Hypothesis~\ref{H:standard} holds.

\begin{proof}[Proof of Theorem~\ref{T:main}]
If $\F$ is a counterexample to Theorem~\ref{T:main}, then we may choose the
notation such that Hypothesis~\ref{H:maximal} holds (cf.
Remark~\ref{R:Subintrinsic}). So by Theorem~\ref{T:Cstandard},
Hypothesis~\ref{H:standard} holds. Adopt Notation~\ref{N:standard}. Thus,
Proposition~\ref{P:quaternion} yields that $Q$ is cyclic, so that
Hypothesis~\ref{H:cyclic} holds. However, now Lemma~\ref{L:t2central} yields a
contradiction to the assumption that $\C$ is subintrinsic. 
\end{proof}

\section{General Benson-Solomon components}\label{S:general}

In this section, we apply Walter's Theorem for fusion systems
\cite[Theorem]{AschbacherWT} to treat the general Benson-Solomon component
problem under the assumption that all components in involution centralizers are
on the list of currently known quasisimple $2$-fusion systems, and we thus
complete the proof of Theorem~\ref{T:main2}.  We stress that the proof of
Walter's theorem relies in turn on our Theorem~\ref{T:main}. 

\textbf{Throughout this section, let $\F$ be a saturated fusion system over a
$2$-group $S$.}

The reader is referred to Section~\ref{SS:known} for more information on the
class of known quasisimple fusion systems, as well as for the definition of the
subclass $\Chev[\larg]$. The following theorem is essentially a restatement of
Theorem~\ref{T:main2}. For if Theorem~\ref{T:maingen}(2) holds, then it follows
from \cite[10.11.3]{AschbacherGeneralized} that $\C$ is diagonally embedded in
$\D\D^t$ with respect to $t$.

\begin{theorem}
\label{T:maingen}
Let $\F$ be a saturated fusion system over the $2$-group $S$.  Assume that all
members of $\CC(\F)$ are known and that some fixed member $\C \in \CC(\F)$ is
isomorphic to $\F_{\Sol}(q)$ for some odd $q$. Then for each $t \in \I(\C)$,
there exists a component $\D$ of $\F$ such that one of the following holds.
\begin{enumerate}
\item $\D = \C$; 
\item $\D \cong \C$, $\D^t \neq \D$, $t\in (\D\D^t\gen{t})^f$ and $\C=C_{\D\D^t}(t)$; or
\item $\D \cong \F_{\Sol}(q^2)$, $t \nin \D$, $t\in(\D\gen{t})^f$ and $\C =
C_\D(t)$.
\end{enumerate}
\end{theorem}

It is worth remarking on a few technical points: if $t\in S$
normalizes a component $\D$ of $\F$, then $\D$ is normal in $E(\F)\gen{t}$ and
thus we may form $\D\gen{t}$ inside of $E(\F)\gen{t}$.  Moreover, if $t$ is
fully $\D\gen{t}$-normalized, then $C_\D(t)=C_\D(t)_{E(\F)\gen{t}}$ is defined
(cf. Definition~\ref{D:NEP}).  Whenever we write $C_\D(t)$, we mean
implicitly that $t$ normalizes $\D$.  If $\D$ is a component of $\F$ with
$\D\neq\D^t$, then similarly $\D\D^t$ is normal in $E(\F)\gen{t}$, thus we may
form $\D\D^t\gen{t}=(\D\D^t\gen{t})_{E(\F)\gen{t}}$ and, if $t\in
(\D\D^t\gen{t})^f$, then also $C_{\D\D^t}(t)=C_{\D\D^t}(t)_{E(\F)\gen{t}}$.

The remainder of this section is devoted to the proof of
Theorem~\ref{T:maingen}. As Walter's theorem for fusion systems is applied
twice in the proof, we restate that theorem here. 

\begin{theorem}[Walter's Theorem for fusion systems, \cite{AschbacherWT}]\label{T:WT}
Suppose that all members of $\CC(\F)$ are known and that there exists $\L\in
\CC(\F)$ such that $\L$ is in $\Chev[\larg]$. Let $t\in \I(\L)$  such that $t$
is fully $\F$-centralized. Then there exists a component $\D$ of $\F$ such that
one of the following holds.  
\begin{itemize}
\item [(1)] $\D\in\Chev[\larg]$ and $\L\in\Comp(C_\D(t))$;
\item [(2)] $\L=E(C_{\D\D^t}(t))$ is a homomorphic image of $\D$, so $\D\in\Chev[\larg]$;
\item [(3)] $\L$ is the $2$-fusion system of $\Spin_7(q)$, $\L=C_\D(t)$, and
$\D = \F_{\Sol}(q)$ for some odd prime power $q$; or
\item [(4)] $\L$ is the $2$-fusion system of $SL_2(9)$, $\L\in\Comp(C_\D(t))$,
and $\D$ is the $2$-fusion system of a finite group $D$ such that either
$D\cong 2A_n$ or $D/Z(D)\cong L_3(4)$ and $Z(\L)\leq\Phi(Z(D))$.
\end{itemize}
\end{theorem}

As in the statement of Walter's Theorem, $\L$ in this section is always
some type of fusion subsystem, and not a linking system. We will not need to
make any explicit reference to linking systems in this section.

We continue now with three technical lemmas that will be needed in the proof of
Theorem~\ref{T:maingen}.

\begin{lemma}\label{L:invcentknown}
Assume all members of $\CC(\F)$ are known quasisimple $2$-fusion systems, and
fix a fully centralized involution $z$ of $\F$. Then all members of
$\CC(C_\F(z))$ are known.
\end{lemma}
\begin{proof}
Given a fully centralized involution $t$ in $C_\F(z)$ and a component $\K$ of
$C_{C_\F(z)}(t)$, it follows from \cite[1.3]{AschbacherWT} that there is a
component $\L$ of $C_\F(z)$ such that either $\K$ is a homomorphic image
of $\L$, or $\L$ is $t$-invariant, $t$ does not centralize $\L$, and $\K \in
\Comp(C_{\L\gen{t}}(t))$. In the former case, $\K$ is known, so assume the
latter case. Then $\K$ is a component in the centralizer of some involution in
an almost quasisimple extension $\L\gen{t}$ of $\L \in \CC(\F)$.  By assumption
and Theorem~2.29 of \cite{AschbacherOliver2016}, either $\L$ is a
Benson-Solomon system, or $\L$ is tamely realized by a finite quasisimple
group. If $\L$ is a quasisimple extension of $\F_{\Sol}(q)$, then $\L =
\F_{\Sol}(q)$ by a result of Linckelmann \cite[Theorem~4.2]{HenkeLynd2018}.
Thus Lemma~\ref{L:involutionsconjugate} and Proposition~\ref{P:fieldconjugate}
yield that $\K$ is either the fusion system of $\Spin_7(q)$ (if $t$ is inner)
or a Benson-Solomon system (if $t$ is not inner).  Hence, $\K$ is known in
this case. On the other hand, if $\L$ is tamely realized by a finite
quasisimple group $L$, then $\L\gen{t}$ is tamely realized by an extension
$L\gen{t}$ of $L$ by Theorem~\ref{T:reduct}. Hence, components in centralizers
of involutions in $L\gen{t}$ are known by Lemma~\ref{L:knowncent}, and so $\K$
is known by Lemma~\ref{L:ComponentsFSGroups}.
\end{proof}

\begin{lemma}\label{L:CCM}
If $\M$ is a saturated subsystem of $\F$ such that $O^2(\M)$ is a component of
$\F$, then $\CC(\M)\subseteq \CC(\F)$.
\end{lemma}
\begin{proof}
Let $\M$ be a saturated subsystem of $\F$ over the subgroup $M \leq S$ such
that $\D:=O^2(\M)$ is a component of $\F$. Write $D$ for the Sylow subgroup of
$\D$. Fix $\C \in \CC(\M)$ and $t \in \I_\M(\C)$. We will show that
$\C\in\CC(\F)$. By definition of $\I_\M(\C)$ and by
Lemma~\ref{L:CCconjugate}(b), replacing $(t,\C)$ by $(t^\alpha,\C^\alpha)$ for
some suitable $\alpha\in\AA_\M(t)$, we may assume that $\gen{t}$ is fully
$\M$-normalized and $\C$ is a component of $C_\M(t)$.

As $\D$ is a component of $\F$, the normalizer $N_\F(\D)$ is defined in
\cite[Definition~2.2.1]{AschbacherFSCT}. By construction, this is a subsystem
of $\F$ over $N_S(D)$. Moreover, by \cite[Theorem~2.1.14, 2.1.15,
2.1.16]{AschbacherFSCT}, $N_\F(\D)$ is a saturated, $\D$ is normal in
$N_\F(\D)$, and every saturated subsystem of $\F$ in which $\D$ is normal is
contained in $N_\F(\D)$. In particular, $\M\leq N_\F(\D)$. Observe that
$\D$ is also normal in $E(\F)\gen{t}$ and thus $E(\F)\gen{t}\leq
N_\F(\D)$. By \cite[Theorem~1]{Henke2013}, $(\D\gen{t})_{N_\F(\D)}$ is the
unique saturated subsystem $\m{Y}$ of $N_\F(\D)$ over $D\gen{t}$ such that
$O^2(\m{Y})=O^2(\D)=\D$. Thus,
$(\D\gen{t})_\M=(\D\gen{t})_{N_\F(\D)}=(\D\gen{t})_{E(\F)\gen{t}}$ and we will
denote this subsystem by $\D\gen{t}$. As a consequence,
$C_\D(t)_\M=C_\D(t)_{E(\F)\gen{t}}$ and again we will denote this subsystem
just by $C_\D(t)$. 

As $\gen{t}$ is fully $\M$-normalized, it follows from Lemma~\ref{L:NEP1}
(applied with $\M$ and $\D$ in place of $\F$ and $\E$) that $\gen{t}$ is fully
$\D\gen{t}$-normalized and $C_\D(t)$ is a normal subsystem of $C_\M(t)$.
Recall that $\C$ is a component of $C_\M(t)$ and write $T$ for the Sylow of
$\C$. Observe that $\C=O^2(\C) \leq O^2(C_{\M}(t)) \leq \D$ and thus
$T\leq C_D(t)$. By \cite[9.1.2]{AschbacherGeneralized}, $T$ is nonabelian and
thus $[T,C_D(t)]\neq 1$. Hence, \cite[9.6]{AschbacherGeneralized} gives that
$\C$ is a component of $C_\D(t)$. 

Let $\phi\in\AA_{E(\F)\gen{t}}(t)$ so that $t^\phi\in (E(\F)\gen{t})^f$. Note
that $E(\F)\gen{t}=E(\F)\gen{t^\phi}$. By Lemma~\ref{L:NEP2} applied with
$E(\F)\gen{t}$ and $\D$ in place of $\F$ and $\E$, we have
$\gen{t^\phi}\in(\D\gen{t^\phi})^f$, $C_D(t)^\phi=C_D(t^\phi)$, and
$\phi|_{C_D(t)}$ induces an isomorphism from $C_\D(t)$ to $C_\D(t^\phi)$. Thus
$\C^\phi$ is a component of $C_\D(t^\phi)$. 

By Lemma~\ref{L:NEP1} applied with $E(\F)\gen{t}$, $\D$ and $\gen{t^\phi}$ in
place of $\F$, $\E$ and $P$, we have $C_\D(t^\phi)\unlhd
C_{E(\F)\gen{t^\phi}}(t^\phi)$. So $\C^\phi$ is subnormal in
$C_{E(\F)\gen{t^\phi}}(t^\phi)$ and thus a component of
$C_{E(\F)\gen{t^\phi}}(t^\phi)$. As $t^\phi\in (E(\F)\gen{t})^f$, it follows
from Lemma~\ref{L:NEP1} applied with $E(\F)\gen{t^\phi}$ and $E(\F)$ in place
of $\F$ and $\E$ that $C_{E(\F)}(t^\phi)\unlhd C_{E(\F)\gen{t^\phi}}(t^\phi)$.
Writing $E$ for the Sylow subgroup of $E(\F)$, we have $T^\phi\leq
C_D(t^\phi)\leq C_E(t^\phi)$. As $T^\phi$ is nonabelian, it follows thus from
\cite[9.6]{AschbacherGeneralized} that $\C^\phi$ is a component of
$C_{E(\F)}(t^\phi)$.

Let now $\alpha\in\AA_\F(t^\phi)$. By Lemma~\ref{L:NEP2}, $t^{\phi\alpha}$ is
fully $E(\F)\gen{t^{\phi\alpha}}$-centralized,
$C_E(t^\phi)^\alpha=C_E(t^{\phi\alpha})$, and $\alpha|_{C_E(t^\phi)}$ induces an
isomorphism from $C_{E(\F)}(t^\phi)$ to $C_{E(\F)}(t^{\phi\alpha})$. So
$\C^{\phi\alpha}$ is a component of $C_{E(\F)}(t^{\phi\alpha})$. By
Lemma~\ref{L:NEP1}, $C_{E(\F)}(t^{\phi\alpha})\unlhd C_\F(t^{\phi\alpha})$ and
thus $\C^{\phi\alpha}$ is a component of $C_\F(t^{\phi\alpha})$. In particular,
$t^{\phi\alpha}\in\tilde{\X}(\C^{\phi\alpha})$ and so $t\in\tilde{\X}(\C)$ by
Lemma~\ref{L:CCconjugate}(b) applied with $t^{\phi\alpha}$,
$\C^{\phi\alpha}$ and $(\phi\alpha|_{T\<t\>})^{-1}$ in place of $X$, $\C$ and
$\phi$. This implies $\C\in\CC(\F)$ as required.
\end{proof}

\begin{lemma}\label{L:ConjugateDDt}
Let $\C\in\CC(\F)$ be a subsystem over $T\leq S$. Fix $t\in \I(\C)$ and let
$\gamma\in\Hom_\F(\gen{T,t},S)$. If there exists a component $\D$ of $\F$ such
that one of the conditions (1), (2), or (3) in Theorem~\ref{T:maingen} holds,
then there exists a component $\hat{\D}$ of $\F$ such that the same condition
holds with $t^\gamma$, $\C^\gamma$ and $\hat{\D}$ in place of $t$, $\C$ and
$\D$. 
\end{lemma}
\begin{proof}
By Lemma~\ref{L:ConjugateComponents} the claim is clear if (1) holds, so
suppose (2) or (3) holds for some component $\D$ of $\F$. Write $S_0\leq S$ for
the Sylow subgroup of $E(\F)$ and $D$ for the Sylow subgroup of $\D$.  By
\cite[1.3.2]{AschbacherFSCT}, we have $\F=\gen{E(\F)S,N_\F(S_0)}$. So it is
sufficient to show the assertion when $\gamma$ is a morphism in $E(\F)S$ or in
$N_\F(S_0)$. However, if $\gamma$ is a morphism in $E(\F)S$, then by
Lemma~\ref{L:ProductFactorize} applied with
$(\gamma,E(\F),S_0,\gen{t},\gen{T,t})$ in place of $(\phi,\E,T,P,X)$, we can
write $\gamma$ as the composition of a morphism in $E(\F)\gen{t}$ and a
morphism in $\F_S(S)\leq N_\F(S_0)$. Thus, it is enough to show the claim if
$\gamma$ is a morphism in $E(\F)\gen{t}$ or in $N_\F(S_0)$. Notice that our
assumption implies $\C\leq E(\F)$ and thus $\gen{T,t}\leq S_0\gen{t}$.

Suppose that (2) holds. Assume first that $\gamma$ is a morphism in
$E(\F)\gen{t}$. Then $S_0\gen{t}=S_0\gen{t^\gamma}$. As $S_0$ normalizes every
component of $\F$, the action of $t^\gamma$ on the components of $\F$
coincides with the one of $t$. So $\D^t=\D^{t^\gamma}$. In particular,
$\D\neq\D^{t^\gamma}$ and $\D\D^t=\D\D^{t^\gamma}$. Note that (2) implies in
particular $T=C_{DD^t}(t)$, so $\gamma$ is defined on $\<C_{DD^t}(t),t\>$.
Thus, Lemma~\ref{L:NEP2} applied with $E(\F)\gen{t}$ and $\D\D^t$ in place of
$\F$ and $\E$ gives that $t^\gamma$ is fully normalized in
$\D\D^t\gen{t^\gamma}=\D\D^{t^\gamma}\gen{t^\gamma}$ and
$\C=C_{\D\D^t}(t^\gamma)=C_{\D\D^{t^\gamma}}(t^\gamma)$. So (2) holds with
$t^\gamma$ and $\C^\gamma$ in place of $t$ and $\C$, i.e. the assertion is
true for $\hat{\D}=\D$.

Assume now that $\gamma$ is a morphism in $N_\F(S_0)$. Then $\gamma$ extends to
$\alpha\in\Hom_\F(\gen{S_0,t},S)$. So by Lemma~\ref{L:ConjugateComponents}(b),
$\hat{\D}:=\D^\alpha$ is a component of $\F$. Clearly,
$\hat{\D}\cong\D\cong\C\cong\C^\gamma$. Notice that
$\hat{\D}^{t^\gamma}=(\D^\alpha)^{t^\alpha}=(\D^t)^\alpha$. In particular,
$\hat{\D}^{t^\gamma}\neq \hat{\D}$ and $\C^\gamma=\C^\alpha\leq
(\D\D^t)^\alpha=\hat{\D}\hat{\D}^{t^\gamma}$. Moreover, $\alpha$ induces an
isomorphism from $E(\F)\gen{t}$ to $E(\F)\gen{t^\gamma}$ which takes $t$ to
$t^\gamma$ and $\D\D^t$ to $\hat{\D}\hat{\D}^{t^\gamma}$. Thus, $\alpha$
induces an isomorphism from $\D\D^t\gen{t}$ to
$\hat{\D}\hat{\D}^{t^\gamma}\gen{t^\gamma}$. This implies that (2) holds with
$t^\gamma$, $\C^\gamma$ and $\hat{\D}$ in place of $t$, $\C$ and $\D$.

Suppose now that (3) holds and assume again first that $\gamma$ is a morphism
in $E(\F)\gen{t}$. As observed above, $\D$ is normal in $E(\F)\gen{t}$. As
$t\not\in D$, it follows that $t^\gamma\not\in D$. Note moreover that
$S_0\gen{t}=S_0\gen{t^\gamma}$. In particular, as $S_0\gen{t}$ normalizes $\D$,
we have $\D^{t^\gamma}=\D$. As $\C=C_\D(t)$ by assumption, we have $C_D(t)=T$
and thus $\gamma\in\Hom_{E(\F)\gen{t}}(\gen{C_D(t),t},S_0\gen{t})$. Hence, by
Lemma~\ref{L:NEP2} applied with $E(\F)\gen{t}$ and $\D$ in place of $\F$ and
$\E$, we get that $\gen{t^\gamma}\in(\D\gen{t^\gamma})^f$ and
$\C^\gamma=C_\D(t)^\gamma=C_\D(t^\gamma)$. So the assertion holds in this case
for $\hat{\D}=\D$.

Assume now that $\gamma$ is a morphism in $N_\F(S_0)$ and choose a morphism
$\alpha\in\Hom_\F(S_0\gen{t},S)$ which extends $\gamma$. Then
$\hat{\D}:=\D^\alpha$ is a component of $\F$ over $\hat{D}:=D^\alpha$, and
$\alpha$ induces an isomorphism from $E(\F)\gen{t}$ to $E(\F)\gen{t^\gamma}$
which takes $\D$ to $\hat{\D}$ and $t$ to $t^\gamma$. Hence,
$t^\gamma\not\in\hat{D}$ and $\hat{\D}^{t^\gamma}=\hat{\D}$. Moreover, $\alpha$
induces also an isomorphism from $\D\gen{t}=(\D\gen{t})_{E(\F)\gen{t}}$ to
$\hat{\D}\gen{t^\gamma}=(\hat{\D}\gen{t^\gamma})_{E(\F)\gen{t^\gamma}}$. As $t$
is fully $\D\gen{t}$-normalized, it follows that $t^\gamma=t^\alpha$ is fully
$\hat{\D}\gen{t^\gamma}$-normalized. Moreover, $\alpha$ induces an isomorphism
from $C_{\D\gen{t}}(t)$ to $C_{\hat{\D}\gen{t^\gamma}}(t^\gamma)$. Observe also
that $C_D(t)^\alpha=C_{D^\alpha}(t^\alpha)=C_{\hat{D}}(t^\gamma)$. So $\alpha$
takes the unique normal subsystem of $C_{\D\gen{t}}(t)$ over $C_D(t)$ of
$2$-power index to the unique normal subsystem of
$C_{\hat{\D}\gen{t^\gamma}}(t^\gamma)$ over $C_{\hat{D}}(t^\gamma)$ of
$2$-power index. In other words, we have
$C_\D(t)^\alpha=C_{\hat{\D}}(t^\gamma)$ and thus
$\C^\gamma=\C^\alpha=C_\D(t)^\alpha=C_{\hat{\D}}(t^\gamma)$. So (2) holds with
$t^\gamma$, $\C^\gamma$ and $\hat{\D}$ in place of $t$, $\C$, and $\D$. 
\end{proof}

\begin{proof}[Proof of Theorem~\ref{T:maingen}]
Let $\F$ be a counterexample having a minimal number of morphisms. Fix $\C \in
\CC(\F)$ and $t\in\I(\C)$ such that $\C \cong \F_{\Sol}(q)$ and none of the
conclusions (1), (2), or (3) hold for any choice of $\D$. Let $T$ be the
Sylow of $\C$, and write $Z(T) = \gen{z}$. Set $\H = C_\C(z)$.

We may and do assume that $\gen{z}$ is fully $\F$-centralized and that
$\gen{t}$ is fully $C_\F(z)$-centralized. This follows in the standard way by
choosing $\beta \in \AA(z)$, choosing $\gamma \in
\AA_{C_\F(z^\beta)}(t^\beta)$, setting $\phi = \beta\gamma$, and replacing $z$
by $z^\phi$, $t$ by $t^{\phi}$, and
$\C$ by $\C^{\phi}$.  In this process, note that Lemma~\ref{L:CCconjugate}(b)
shows that we still have $t^{\phi} \in \I(\C^\phi)$. Also by
Lemma~\ref{L:ConjugateDDt} applied with $\phi^{-1}$ in the role of $\gamma$,
if there exists $\hat{\D}$ such that one of the conclusions (1)--(3) holds with
$(t^{\phi},\C^{\phi},\hat{\D})$ in place of $(t,\C,\D)$, then one of the
conclusions (1)--(3) holds with respect to $t,\C$ and some suitable $\D$.

Fix $\alpha \in \AA(t)$. Then
\begin{eqnarray}
\label{E:shift}
\text{$z^\alpha \in C_\F(t^\alpha)^f$ and the map $\alpha\colon C_{C_\F(z)}(t)
\to C_{C_\F(t^\alpha)}(z^\alpha)$ is an isomorphism}
\end{eqnarray}
by \cite[2.2]{AschbacherGeneration}. Moreover, Lemma~\ref{L:CCconjugate}(a) yields
\begin{eqnarray}\label{E:talpha} 
\text{$\C^\alpha$ is a component of $C_\F(t^\alpha)$.} 
\end{eqnarray}
If there exists a component $\D$ of $\F$ such that one of the conclusions
(1)--(3) holds with $(t^\alpha,\C^\alpha)$ in place of $(t,\C)$, then it
follows from  Lemma~\ref{L:ConjugateDDt} that for some (possibly different)
choice of $\D$, one of the conclusions (1)--(3) holds. As this would contradict
our assumption, it follows that

\vspace{-1cm}
\begin{eqnarray}\label{E:shift2}
\begin{gathered}
\parbox[t]{0.80\linewidth}{
\begin{center}
there does not exist a component $\D$ of $\F$ such that one of the conclusions
(1)--(3) holds with $(t^\alpha,\C^\alpha)$ in place of $(t,\C)$.
\end{center}
}
\end{gathered}
\end{eqnarray} 

\vspace{-0.30cm}
Since for any component $\D$, neither conclusion (1) nor (2) holds with
$(t^\alpha,\C^\alpha)$ in place of $(t,\C)$, it follows from
\cite[1.3]{AschbacherWT} and the fact that the Benson-Solomon systems have no
proper quasisimple coverings \cite[Theorem~4.2]{HenkeLynd2018} that there is a
unique $t^\alpha$-invariant component $\D$ of $\F$ containing $\C^\alpha$ such
that $\D \neq \C^\alpha$ and $\C^\alpha \in
\Comp(C_{\D\gen{t^\alpha}}(t^\alpha))$. In particular,
$t^\alpha\in\I_{\D\gen{t^\alpha}}(\C^\alpha)$.  All members of
$\CC(\D\gen{t^\alpha})$ are known by Lemma~\ref{L:CCM}. Moreover, notice that
$\D$ is the unique component of $\D\gen{t^\alpha}$. So if $\D\gen{t^\alpha}$ is
not a counterexample, then conclusion (3) holds with $(t^\alpha,\C^\alpha)$ in
place of $(t,\C)$ contradicting \eqref{E:shift2}. Hence
$\D\gen{t^\alpha}$ is a counterexample, and so $\F = \D\gen{t^\alpha}$ by
minimality of $\F$. This implies that 
\begin{eqnarray}
\label{E:F*Fqs}
F^*(\F) = O^2(\F) = \D \text{ is quasisimple}.
\end{eqnarray}
It is possible at this point that $\F = \D$. 

We first prove that
\begin{eqnarray}
\label{E:step1}
\text{$\H$ is a component of $C_{C_{\F}(z)}(t)$.}
\end{eqnarray}
Recall that $\gen{z^\alpha}$ is fully $C_{\F}(t^\alpha)$-normalized by
\eqref{E:shift} and that $\C^\alpha$ is subnormal in $C_\F(t^\alpha)$ by
\eqref{E:talpha}. Fix a subnormal series $\C^\alpha = \F_0 \norm \cdots \norm
\F_n = C_\F(t^\alpha)$ for $\C^\alpha$. Then $\gen{z^\alpha}$ is fully
$\F_i$-normalized for each $i$ by Lemma~\ref{L:Subnormalfn}. Also, $\H^\alpha =
C_{\C^\alpha}(z^\alpha) \norm C_{\F_1}(z^\alpha) \norm \cdots \norm
C_{C_{\F}(t^\alpha)}(z^\alpha)$ is a subnormal series for $\H^\alpha$ in
$C_{C_\F(t^\alpha)}(z^\alpha)$ by application of
\cite[8.23.2]{AschbacherGeneralized} and induction on $n$.  Hence, $\H$ is a
component of $C_{C_\F(z)}(t)$ by the isomorphism in \eqref{E:shift}. 

We next apply Walter's Theorem in $C_\F(z)$. Recall that we took $t$ to be
fully $C_\F(z)$-centralized. By \eqref{E:step1}, $\H \in \Chev[\larg]$ is a
component of $C_{C_\F(z)}(t)$, so $t \in \I_{C_\F(z)}(\H)$. All members of
$\CC(C_\F(z))$ are known by Lemma~\ref{L:invcentknown}. Thus, the hypotheses of
Walter's Theorem (Theorem~\ref{T:WT}) are satisfied with $C_\F(z)$, $\H$, and
$t$ in the roles of $\F$, $\L$, and $t$. Let $\M \in \Comp(C_\F(z))$ be as
given by Theorem~\ref{T:WT} (in the role of $\D$).  Since $\H$ is not the
$2$-fusion system of $SL_2(9)$, we are not in Theorem~\ref{T:WT}(4). Consider
the case that (1) or (3) of Theorem~\ref{T:WT} holds. Then $\H = O^2(\H) \leq
O^2(\M\gen{t}) = \M$, so $z$ is an element of the Sylow of $\M$.  Since $\M
\leq C_\F(z)$, it follows that $\gen{z}$ is strongly $\M$-closed.  Thus, $z \in
Z(\M)$ in this case by \cite[Corollary~I.4.7(a)]{AschbacherKessarOliver2011}.
As $\F_{\Sol}(q)$ has a trivial center, this shows that conclusion (3) of
Walter's Theorem does not hold. Hence, in any case, Theorem~\ref{T:WT}(1) or
(2) holds.  Further, as $\Spin_7(q)$ has no proper quasisimple $2$-coverings
\cite[Tables~6.1.2,6.1.3]{GLS3}, neither does $\H$
\cite[Corollary~6.4]{BCGLO2007}. So if (2) holds, then $\H \cong \M$.
Therefore, 
\begin{eqnarray}
\label{E:step2}
\text{$\M \in \Chev[\larg]$, $z \in \I(\M)$, and $\H$ isomorphic to a subsystem of $\M$.}
\end{eqnarray}

We next apply Walter's Theorem with $\F$, $\M$, and $z$ in the roles of of
$\F$, $\L$, and $t$. Recall that we took $z$ to be fully $\F$-centralized. So
by \eqref{E:step2} and assumption on $\CC(\F)$, the hypotheses of Walter's
Theorem apply. By \eqref{E:F*Fqs}, $\D$ is the unique component of $\F$, and
$\C=O^2(\C)\leq O^2(\F)=\D$.  By \eqref{E:step2}, $\H$ is isomorphic
to a subsystem of $\M$, so $\M$ is not the $2$-fusion system of $SL_2(9)$.
Thus, Theorem~\ref{T:WT}(1), (2), or (3) holds.  In particular, either $\D \in
\Chev[\larg]$ or $\D$ is isomorphic to a Benson-Solomon system. 

Assume the former holds, namely $\D \in \Chev[\larg]$. All members of
$\Chev[\larg]$ are tamely realized by some member of $\Chev^*(p)$ for some odd
prime $p$ by \cite{BrotoMollerOliver2019}, so we may fix $D \in \Chev^*(p)$
tamely realizing $\D$. By Theorem~\ref{T:reduct}, we may further fix a finite
group $G$ with Sylow $2$-subgroup $S$ such that $F^*(G) = D$ and $\F \cong
\F_S(G)$. As $\C^\alpha$ is a component of $C_\F(t^\alpha)$, we obtain from
Lemmas~\ref{L:knowncent} and \ref{L:ComponentsFSGroups} the contradiction that
$\C$ is not exotic.

Therefore, $\D$ is isomorphic to a Benson-Solomon system.  Assume first that
$t^\alpha \in \foc(\F)$. Then $\F = \D$. In particular, $Z(S)$ is the only
fully normalized subgroup of $S$ of order $2$ by
Lemma~\ref{L:involutionsconjugate}. Hence, $Z(S)=\gen{t^\alpha}$ as
$\gen{t^\alpha}$ is fully $\F$-normalized. So $C_{\F}(t^\alpha)$ is the
$2$-fusion system of $\Spin_7(q')$ for some odd prime power $q'$, which
contradicts \eqref{E:talpha}.  Hence, $t^\alpha \notin \foc(\F)$ and so
$t^\alpha\notin\D$.  Applying Proposition~\ref{P:fieldconjugate}, we see that
$\D \cong \F_{\Sol}(q^2)$ and $C_\D(t^\alpha)=\C$. Therefore, (3) holds after
all with $(t^\alpha,\C^\alpha)$ in place of $(t,\C)$, a contradiction to
\eqref{E:shift2}.
\end{proof}

\bibliographystyle{amsalpha}{ }
\bibliography{/home/justin/work/math/research/mybib.bib}
\end{document}